\documentclass[12pt]{article}

\usepackage[colorlinks=true, pdfstartview=FitV, linkcolor=blue, citecolor=blue, urlcolor=blue]{hyperref}
\usepackage[OT2,OT1]{fontenc}
\usepackage{amsfonts,amsmath, amssymb,latexsym,mathrsfs}
\usepackage{amssymb,amsmath, amscd,mathdots}
\usepackage{times, verbatim}
\usepackage{graphicx}
\usepackage[all,cmtip]{xy}

\DeclareFontFamily{OT1}{rsfs}{}
\DeclareFontShape{OT1}{rsfs}{n}{it}{<-> rsfs10}{}
\DeclareMathAlphabet{\mathscr}{OT1}{rsfs}{n}{it}

\newtheorem{theorem}{Theorem}[section]
\newtheorem{corollary}[theorem]{Corollary}
\newtheorem{lemma}[theorem]{Lemma}
\newtheorem{proposition}[theorem]{Proposition}
\newtheorem{remark}[theorem]{Remark}
\newenvironment{proof}{\noindent {\bf Proof:}}{$\Box$ \vspace{2 ex}}

\numberwithin{equation}{section}

\DeclareMathOperator{\End}{End}
\DeclareMathOperator{\Hom}{Hom}
\DeclareMathOperator{\Gal}{Gal}
\DeclareMathOperator{\Aut}{Aut}

\DeclareMathOperator{\SO}{SO}

\DeclareMathOperator{\SL}{SL}
\DeclareMathOperator{\GL}{GL}
\DeclareMathOperator{\PGL}{PGL}
\DeclareMathOperator{\Trace}{Trace}
\DeclareMathOperator{\disc}{disc}
\DeclareMathOperator{\Res}{Res}
\DeclareMathOperator{\Spec}{Spec}
\DeclareMathOperator{\Stab}{Stab}

\DeclareMathOperator{\Pic}{Pic}

\def\Z{{\mathbb Z}}
\def\End{{\rm End}}
\def\SL{{\rm SL}}
\def\GL{{\rm GL}}
\def\PGL{{\rm PGL}}

\def\Stab{{\rm Stab}}
\def\bigstab{{\rm bigstab}}
\def\Sym{{\rm Sym}}
\def\Jac{{\rm Jac}}

\def\O{{\rm O}}
\def\SO{{\rm SO}}

\def\P{{\mathbb P}}
\def\Disc{{\rm Disc}}

\def\ss{{\rm ss}}
\def\disc{{\rm disc}}
\def\Aut{{\rm Aut}}
\def\irr{{\rm irr}}

\def\Inv{{\rm Inv}}
\def\red{{\rm red}}
\def\Det{{\rm Det}}
\def\Vol{{\rm Vol}}
\def\R{{\mathbb R}}
\def\F{{\mathbb F}}
\def\FF{{\mathcal F}}
\def\RR{{\mathcal R}}

\def\Q{{\mathbb Q}}

\def\H{{\mathcal H}}
\def\G{{\mathcal G}}

\def\C{{\mathcal C}}

\def\Z{{\mathbb Z}}
\def\P{{\mathbb P}}
\def\F{{\mathbb F}}
\def\Q{{\mathbb Q}}
\def\C{{\mathbb C}}

\def\H{{\mathcal H}}

\def\BB{{\mathcal E}}
\def\Sha{\mbox{\fontencoding{OT2}\selectfont\char88}}

\title{The average size of the $2$-Selmer group of Jacobians of \\ hyperelliptic curves having a rational Weierstrass point}

\author{Manjul Bhargava and Benedict H.\ Gross}

\begin{document}
\maketitle

\begin{abstract}
We prove that when all hyperelliptic curves of genus $n \geq 1$ having
a rational Weierstrass point are ordered by height, the average size of the $2$-Selmer group of their Jacobians is equal to $3$.  
It~follows that
(the limsup of)  the average rank of the Mordell-Weil group of their Jacobians is at~most~$3/2$.  

The method of Chabauty can then be used to obtain an effective bound on the number of rational points on most of these hyperelliptic curves; for example, we show that a majority of hyperelliptic curves of genus $n\geq 3$ with a rational Weierstrass point have fewer than 20 rational points.
\end{abstract}

\tableofcontents
\section{Introduction}

A hyperelliptic curve $C$ of genus $n \geq 1$ over $\Q$ with a marked rational Weierstrass point $O$ has an affine equation of the form
\begin{equation}\label{Ccdef}
y^2 = x^{2n+1} + c_2x^{2n-1} + c_3x^{2n-2} + \ldots + c_{2n+1} = f(x)
\end{equation}
with rational coefficients $c_m$. The point $O$ lies above $x = \infty$, the polynomial $f(x)$ is separable, and the ring of functions which are regular outside of $O$ is the Dedekind domain $\Q[x,y] = \Q[x,\sqrt{f(x)}]$. The change of variable $(x' = u^{2}x, y' = u^{2n+1}y)$ results in a change in the coefficients $c_m' = u^{2m} c_m$.  Hence we may assume the coefficients are all integers. These integers are unique if we assume further that, for every prime $p$, the integral coefficients $c_m$ are not all divisible by $p^{2m}$. In this case we say the coefficients are {\it indivisible}.

To make the next two definitions, we assume that $C$ is given by its unique equation with indivisible integral coefficients.
The discriminant  of $f(x)$ is a polynomial $D(c_2,c_3,\ldots, c_{2n+1})$ of weighted homogeneous degree $2n(2n+1)$ in the coefficients $c_m$, where $c_m$ has degree $m$. Since $f(x)$ is separable, this discriminant is nonzero. We define the {\it discriminant} $\Delta$ of the curve $C$ by the formula $\Delta(C) := 4^{2n}D(c_2,c_3, \ldots, c_{2n+1})$, and the (naive) {\it height} $H$ of the curve $C$ by
$$H(C) :=  \max \{|c_k|^{{2n(2n+1)}/{k}}\}_{k=2}^{2n+1}.$$
We include the expression $2n(2n+1)$ in the definition so that the weighted homogeneous degree of the height function $H$ is the same as that of the discriminant $\Delta$.
If $\Delta(C)$ is prime to $p$, the hyperelliptic curve $C$ has good reduction modulo $p$. The height $H(C)$ gives a concrete way to enumerate all of the hyperelliptic curves of a fixed genus with a rational Weierstrass point: for any real number $X > 0$ there are clearly only finitely many  curves with $H(C) < X$.  

The discriminant and height extend the classical notions in the case
of elliptic curves (which is the case $n = 1$).  Any elliptic curve $E$ over $\Q$ is given by a unique equation of the form $y^2=x^3+c_2x+c_3$, where $c_2,c_3 \in \Z$
and for all primes $p$:\, $p^6 \nmid c_3$ whenever $p^4 \mid c_2$. The discriminant
is then defined by the formula  $\Delta(E) := 2^4(-4c_2^3 - 27c_3^2)$ and the
naive height 
by 
$H(E):= \max \{|c_2|^3,|c_3|^2\}$.

Recall that the $2$-Selmer group $S_2(J)$ of the Jacobian $J=\Jac(C)$ of $C$ is a finite subgroup of the Galois cohomology group $H^1(\Q,J[2])$, which is defined by local conditions and fits into an exact sequence
$$
0\to J(\Q)/2J(\Q)\to S_2(J) \to \Sha_J[2]\to 0,$$
where $\Sha_J$ denotes the Tate-Shafarevich group of $J$ over $\Q$.

The purpose of this paper is to prove the following theorem. 

\begin{theorem}\label{main}
When all hyperelliptic curves of fixed genus $n \geq 1$ over $\Q$ having a rational
Weierstrass point are ordered by height,
the average size of the $2$-Selmer groups of their Jacobians is equal to~$3$.
\end{theorem}

More precisely, we show that
$$
\lim_{X\to\infty} \frac{\sum_{H(C)<X} \#S_2(\Jac(C))}{\sum_{H(C)<X} 1} = 3.
$$
In fact, we prove that the same result remains true even when we average over any subset of hyperelliptic curves $C$ defined by a finite set of congruence conditions on the coefficients $c_2, c_3, \ldots\,, c_{2n+1}$.


Since the 2-rank $r_2(S_2(J))$ of the 2-Selmer group of the Jacobian $J$ bounds the rank of the Mordell-Weil group $J(\Q)$, and since 
$
2r_2(S_2(J))\leq 2^{r_2(S_2(J))}=\#S_2(J), 
$
by taking averages 
we immediately obtain:
\begin{corollary}\label{cor}
  When all rational hyperelliptic curves of fixed genus $n \geq 1$ with a
  rational Weierstrass point are ordered by height, the average
  $2$-rank of the $2$-Selmer groups of their Jacobians is at most~$3/2$. Thus the
  average rank of the Mordell-Weil groups of their Jacobians is 
  at most $3/2$.
\end{corollary}
 
Another corollary of Theorem~\ref{main} is that the average $2$-rank of the $2$-torsion subgroup of the Tate-Shafarevich group $\Sha_J[2]$ is also at most~$3/2$. 
The analogous statements for elliptic curves were proven in \cite{BS}. 

We suspect that the average rank of the Mordell--Weil group is equal to $1/2$. However, the fact that the average rank is at most $3/2$ (independent of the genus) already gives some interesting arithmetic applications.  
Recall that the method of Chabauty~\cite{Cha}, as refined by Coleman~\cite{Col}, yields a finite and effective bound on the number of rational points on a curve over $\Q$ whenever its genus is greater than the rank of its Jacobian. Theorem~\ref{main} implies
\begin{corollary}\label{cdensity}
Let $\delta_n$ denote the lower density of hyperelliptic curves of genus $n$ with a rational Weierstrass point 
satisfying Chabauty's condition.  Then $\delta_n\to1$ as $n\to\infty$.
\end{corollary}
Indeed, we have $\delta_n\geq 1-2/(2^{n}-1)$, so $\delta_n\to 1$ quite rapidly as $n$ gets large.  Thus for an asymptotic density of 1 of hyperelliptic curves with a rational Weierstrass point, one can effectively
bound the number of rational points.

As an explicit consequence, we may use Theorem~\ref{main}, together with the arguments in \cite{St}, to prove the following 
\begin{corollary}\label{ptsbound}
.

\begin{itemize}
\item[{\rm (a)}] For any $n \geq 2$, a positive proportion of hyperelliptic curves of genus $n$ with a rational Weierstrass point have at most $3$ rational points.
\item[{\rm (b)}] For any $n\geq 3$, a majority $($i.e., a proportion of $>50\%)$ of all hyperelliptic curves of genus $n$ with a rational Weierstrass point have fewer than $20$ rational points.
\end{itemize}
\end{corollary}
In fact, we will give a lower bound on the proportion in (a) which is independent of the genus $n$.  
In~(b), when $n=2$, we may still deduce that more 
than a quarter of all such curves have fewer than $20$ rational points. 

The numbers in 
Corollary~\ref{ptsbound} can certainly be
 improved 
with a more careful analysis.
Bjorn Poonen and Michael Stoll have recently shown using Theorem~\ref{main} and Chabauty's method that (provided $n\geq 3$) a positive proportion of such curves 
have the point at infinity as their only rational point;
moreover, the density of such curves having only this one rational point approaches 1 as $n\to\infty$.  See~\cite{PS} for details and further such results.




Note that Corollary~\ref{ptsbound}(a) is 
also 
true when $n=1$:
it follows from \cite{BS2} that a positive proportion of elliptic curves have only one rational point (the origin of the group law). Meanwhile, Corollary~\ref{ptsbound}(b) is not expected to be true when $n=1$, since half of all elliptic curves are expected to have rank one and thus infinitely many rational points.
We suspect that $100\%$ of all hyperelliptic curves of genus $n \geq 2$ over $\Q$ with a rational Weierstrass point have only  one rational point. 

\section{The method of proof}

The proof of Theorem~\ref{main} follows the argument in \cite{BS} for elliptic curves.
There the main idea was to use the classical parametrization \cite[Lemma~2]{BSD} of 2-Selmer elements of elliptic curves by certain binary quartic forms to transform the problem into one of counting the integral orbits of the group $\PGL_2$ on the space $\Sym_4$ of binary quartics.
Over $\Q$, the group $\PGL_2$ is isomorphic to the special orthogonal group of the three-dimensional quadratic space $W$ of $2 \times 2$ matrices of trace zero, with quadratic form given by the determinant, and the representation $\Sym_4$ of $\PGL_2$ on the space of binary quartic forms is given by the action of $\SO(W)$ by conjugation on the space $V$ of self-adjoint operators $T: W \rightarrow W$ with trace zero~\cite{BG}. 

Let $W$ now be a bilinear space of rank $2n+1$ over $\Q$, where the Gram matrix
$A$ of the bilinear form $\langle w,u \rangle$ on $\Q^{2n+1}$ consists of 1's on the anti-diagonal and 0's elsewhere:
\begin{equation}\label{defa}
A = \left[\begin{array}{ccccc}
 & & & & \,1 \\
 & & &\,1\, & \\
 & &\iddots & & \\
  & \,1\, & & & \\
 1\, & & & & 
\end{array}\right].
\end{equation}
Let $\SO(W)$ be the special orthogonal group of $W$ over $\Q$, i.e., the subgroup of $\GL(W)$ defined by the algebraic equations $\langle gw,gu \rangle = \langle w,u \rangle$ and $\det(g) = +1$.  (One can define this group scheme over~$\Z$, using the unimodular lattice defined by $A$, but it is only smooth and reductive over $\Z[1/2]$). Let $V$ be the representation of
$\SO(W)$ given by conjugation on the self-adjoint operators $T:W \rightarrow W$ of trace zero. With respect to the above basis, an operator $T$ is self-adjoint if and only if its matrix $M$ is symmetric when reflected about the 
anti-diagonal. The matrix $B = AM$ is then symmetric in the usual sense. The symmetric matrix $B$ is the Gram matrix of the associated bilinear form $\langle w,u \rangle_T = \langle w,Tu \rangle$, and has anti-trace zero.

The coefficients of the monic characteristic polynomial of $T$, defined by
\begin{equation}\label{feq}
f(x) = \det(xI - T) = (-1)^n \det(xA - B) = x^{2n+1} + c_2x^{2n-1}+ \cdots + c_{2n+1},
\end{equation} 
give $2n$ independent invariant polynomials $c_2, c_3,\ldots c_{2n+1}$
on $V$, having degrees $2, 3, \ldots, 2n+1$ respectively, which generate the (free) ring of $\SO(W)$-invariants \cite{BG}.
When the discriminant $\disc(f)$ of $f$ is nonzero, we show how classes in the $2$-Selmer group of the Jacobian of the hyperelliptic curve $C$ with equation $y^2 = f(x)$ over $\Q$ correspond to certain orbits of $\SO(W)(\Q)$ on $V(\Q)$ with these polynomial invariants. 

More precisely, to each self-adjoint operator $T$ with invariants $c_2, c_3, \ldots, c_{2n+1}$ satisfying $\Delta(T):= 4^{2n}\cdot\disc(f)\neq 0$, we associate a non-degenerate pencil of quadrics in projective space $\P(W\oplus \Q) =\P^{2n+1}$. Two quadrics generating this pencil are $Q(w, z) = \langle w,w \rangle$ and $Q'(w,z) = \langle w,Tw \rangle + z^2$. The discriminant locus $\disc(xQ - x'Q')$ of this pencil is a homogeneous polynomial $g(x,x')$ of degree~$2n+2$ satisfying $g(1,0) = 0$ and $g(x,1) = f(x)$.  
The Fano variety $F_T$ of maximal linear isotropic subspaces of the base locus is smooth of dimension~$n$ over $\Q$ and forms a principal homogeneous space for the Jacobian $J$ of the curve $C$ (see~\cite{D}). 

Both the pencil and the Fano variety have an involution $\tau$ induced by the involution $\tau(w,z) = (w,-z)$ of $W\oplus \Q$. The involution $\tau$ of $F_T$ has $2^{2n}$ fixed points over an algebraic closure of $\Q$, which form a principal homogeneous space $P_T$ for the $2$-torsion subgroup $J[2]$ of the Jacobian $J$. Indeed, one can define the structure of a commutative algebraic group on the disconnected variety $G = J \cup F_T$ such that the involution $-1_G$ of $G$ induces the involution $\tau$ on the non-trivial component $F_T$ (see \cite{W}). We prove that the isomorphism class of the principal homogeneous space $P_T$ over $\Q$ determines the orbit of $T$. This gives an injective map from the set of rational orbits of $\SO(W)$ on $V$ with characteristic polynomial $f(x)$ to the set of elements in the Galois cohomology group $H^1(\Q,J[2])$, which classifies principal homogeneous spaces for this finite group scheme.

\begin{theorem}\label{orbit}
Let $C$ be the hyperelliptic curve of genus $n$ which is defined by the affine equation $y^2=f(x)$, where $f(x)$ is a monic, separable polynomial of degree $2n+1$ over $\Q$.  Then
the classes in the $2$-Selmer group of the Jacobian $J$ of $C$ over $\Q$ correspond bijectively to the orbits of $\SO(W)(\Q)$ on self-adjoint operators $T:W\to W$ with characteristic polynomial $f(x)$ such that the associated Fano variety $F_T$ 
has points over $\Q_v$ for all places $v$.
\end{theorem}

We call such special orbits of $\SO(W)(\Q)$ on $V(\Q)$ 
{\it locally soluble}. Since they are in bijection with elements in the $2$-Selmer group of $\Jac(C)$, they are finite in number (whereas the number of rational orbits with characteristic polynomial $f(x)$ turns out to be infinite). Local solubility is a subtle concept. For example, for any $d$ in $\Q^*$ there is an obvious bijection $(T \rightarrow dT)$ between the rational orbits with characteristic polynomial $f(x)$ and the rational orbits whose characteristic polynomial $f^*(x)$ has coefficients $c_k^* = d^k c_k$. But this bijection does not necessarily preserve the locally soluble orbits, as the hyperelliptic curves $C$ and $C^*$ are isomorphic over $\Q$ only when $d$ is a square.

When $n = 1$, the Fano variety $F = F_T$ is the intersection of two quadrics in $\P^3$. This intersection defines a curve of genus $1$ whose Jacobian is the elliptic curve $J = C$. The quotient of $F$ by the involution $\tau$ is a curve $X$ of genus $0$, and the covering $F \rightarrow X$ is  ramified at the $4$ fixed points comprising $P_T$. Since the curve $X$ injects into the variety of isotropic lines in the split quadratic space $W$ of dimension $2n+1 = 3$, we conclude that $X$ is isomorphic to $\P^1$ over $\Q$. (For $n > 1$, the quotient $X$ of $F$ by $\tau$ embeds as singular subvariety of the Lagrangian Grassmannian of $W$, with $2^{2n}$ double points, which in turn embeds in projective space of dimension $2^n - 1$ via the spin representation. The composition  is the Kummer embedding of $X$.) In the case where $n = 1$, once we fix an isomorphism of $X$ with $\P^1$, the Fano variety $F$ has an equation of the form $z^2 = T(x,y)$, where $T(x,y)$ is a binary quartic form. The orbit is locally soluble if and only if the quartic form $T(x,y)$ represents a square in $\Q_v$ for all places $v$, and the homogeneous space $P_T$ is trivial if and only if $T(x,y)$ has a linear factor over $\Q$.  This is the point of view taken in \cite{BS}, where the latter orbits are called reducible.

In order to obtain Theorem \ref{main} from Theorem \ref{orbit},
we are reduced to counting the $\Q$-orbits in the representation $V$ that are
both locally soluble and have bounded integral invariants.
To count these locally soluble orbits, we prove that every such orbit of $V$ with integral polynomial invariants
has an integral representative having those same
invariants (at least away from the prime 2).
By a suitable adaptation of the counting techniques of \cite{dodpf}
and \cite{BS}, we first carry out a count of the {total
  number} of integral orbits in these representations having bounded
height satisfying certain {\it irreducibility} conditions over $\Q$ (to
be defined).  The primary obstacle in this counting, as in 
representations encountered previously, is that the
fundamental region in which one has to count points is not compact but
instead has a rather complex system of cusps going off to infinity.  A priori, it
could be difficult to obtain exact counts of points of bounded height
in the cusps of these fundamental regions.  We show however that, for all $n$, most 
of the integer points in the cusps correspond to points that are reducible;  meanwhile, most 
of the points in the main bodies of these fundamental regions are irreducible.
The orbits which contain a reducible point turn out to correspond to the identity
classes in the corresponding $2$-Selmer groups.

Since not all integral orbits correspond to Selmer elements, the proof of Theorem \ref{main} requires a {sieve} to those points that correspond to locally soluble
$\Q$-orbits on $V$, for which the squarefree sieve of \cite{geosieve} is applied.  By carrying out this sieve, we prove that
the average occurring in Theorem 1 arises naturally as the sum of two
contributions. One comes from the main body of the fundamental region, which corresponds to the average number of non-identity elements in the 2-Selmer group and which we show is given by the
Tamagawa number ($=2$) of the adjoint group $\SO(W)$ over $\Q$. The other comes from the cusp of the
fundamental region, which counts the average number ($=1$) of identity elements 
in the 2-Selmer group.  The sum $2 + 1 = 3$ then gives us the average size of the 2-Selmer group, as stated in Theorem~1.

To obtain Corollary~\ref{cdensity}, we note that by Theorem~\ref{main} the density $\delta_n$ must satisfy
\begin{equation}
\delta_n \cdot 1 + (1-\delta_n)\cdot 2^n \leq 3
\end{equation}
yielding $\delta_n\geq 1-2/(2^{n}-1)$.

To obtain explicit bounds on the number of rational points on these curves, we follow the arguments in \cite{St}.  Assume that $n \geq 2$.
Using congruence conditions at the prime $p = 3$, we may assume that all of the hyperelliptic curves $C$ of genus $n$ in our sample have good reduction modulo $3$, and that the only point on $C$ over $\Z/3\Z$ is the Weierstrass point at $\infty$. This simply means that the discriminant $\Delta$ of the integral polynomial $f(x)$ is prime to $3$ and that the three values $f(0)$, $f(1)$, and $f(-1)$ are all $\equiv -1$ modulo $3$. Hence any rational point $P$ on $C$ lies in the $3$-adic disc reducing to the Weierstrass point.

Since the average rank of the Jacobians in our sample is at most $3/2$, a positive proportion (indeed, at least a third) of these Jacobians have rank zero or one. Assume that this is the case for such a Jacobian $J$, let $T$ be the (finite) torsion subgroup of $J(\Q_3)$, and let $A$ be the $\Z_3$-submodule of $J(\Q_3)/T$ which is generated by $J(\Q)$. Then the $\Z_3$-rank of $A$ is either zero or one. Let $\Omega$ be the $\Q_3$-vector space of invariant differentials on $J$,
or equivalently the regular $1$-forms on the curve $C$ over $\Q_3$. The logarithm on the formal group gives a $\Z_3$-bilinear pairing $\Omega \times J(\Q_3)/T \rightarrow \Q_3$ with trivial left and right kernels. Let $\Omega_0$ be the subspace of $\Omega$ which annihilates the submodule $A$. The codimension of $\Omega_0$ is either zero or one, and any differential $\omega$ in the subspace $\Omega_0$ pairs trivially with the class of the divisor $(P) - (O)$, for any rational point $P$. Since $P$ and $O$ both lie in the disc reducing to the Weierstrass point, this pairing is computed by $3$-adic integration. Namely, the function $z = x^n/y$ is a uniformizing parameter on this disc which vanishes at $O$ and is taken to its negative by the hyperelliptic involution. If $z(P)$ is the parameter of the rational point $P$ and we expand
$\omega = (a_0 + a_2 z^2 + a_4 z^4+ \cdots)dz$, then the integral of $\omega$ from $O$ to $P$ is given by the value of the formal integral $F(z) = a_0z + (a_2/3) z^3 + (a_4/5) z^5 + \cdots$ at the point $z(P)$ in the maximal ideal of $\Z_3$. We can choose $\omega$ integral, and since we are in a subspace $\Omega_0$ of codimension at most one, the reduction of $\omega$ has at most a double zero at the Weierstrass point. Hence the coefficients $a_{2i}$ are all integral, and either $a_0$ or $a_2$ is a unit. It follows that the formal integral $F(z)$ has at most three zeroes in the disc $3\Z_3$, and hence that there are at most three rational points on the curve $C$, proving Corollary~\ref{ptsbound}(a).

To obtain Corollary~\ref{ptsbound}(b), we
note that a              
density of $1 - 1/7$ of hyperelliptic curves with a rational Weierstrass point have good reduction at 7.
The congruence version of Theorem~\ref{main} implies that among these curves,
the density of curves having rank    
$\leq 2$ is at least $1 - 2/(2^{3}-1).$
Since $\#C(\F_7) \leq 15$, Stoll's refinement~\cite[Cor.\ 6.7]{St} of Coleman's effective bound for Chabauty's method then results in
$\#C(\Q) \leq 15 + 2\cdot 2 = 19$ for such curves.  Hence the majority (indeed greater than $(1 - \frac17)(1 - \frac2{2^{2+1}-1}) = 30/49\approx61\%$) of our curves have fewer than 20 rational points, as desired.

\section{A representation of the orthogonal group} 

Let $W$ be a lattice of rank $2n+1$ having a non-degenerate integral bilinear form with signature $(n+1,n)$. By this we mean that $W$ is a free abelian group of rank $2n+1$ with a non-degenerate symmetric bilinear pairing  $\langle \;,\, \rangle: W \times W \rightarrow \Z$ having signature $(n+1,n)$ over $\R$. It follows that the determinant of $W$ is equal to $(-1)^n$. 

Such a lattice is unique up to isomorphism \cite[Ch.\ V]{S3}. It therefore has an ordered basis
$$\{e_1, e_2,\ldots,e_n, u, f_n,\ldots,f_2,f_1\}$$ with inner products given by
$$\begin{array}{c}
\langle e_i,e_j \rangle = \langle f_i,f_j \rangle =\langle e_i,u \rangle = \langle f_i,u \rangle= 0,\\[.075in]
 \langle e_i, f_j \rangle = \delta_{ij},\\[.085in]
\langle u,u \rangle = 1.
\end{array}
$$
The Gram matrix $A$ of the bilinear form with respect to this basis (which we will call the {\it standard basis}) is the anti-diagonal matrix (\ref{defa}). 

The lattice $W$ gives a split orthogonal space $W \otimes k$ over any field $k$ with char$(k) \neq 2$. This space has dimension $2n+1$ and determinant $\equiv (-1)^n$ in $k^*/k^{*2}$. Such an orthogonal space over $k$ is unique up to isomorphism, and has $\disc(W) = (-1)^n\det(W) \equiv 1$ (see~\cite{MH}). 

Let $T: W \rightarrow W$ be a $\Z$-linear transformation (i.e., an endomorphism of the abelian group $W$). We define the {\it adjoint transformation} $T^*$ uniquely by the formula
$$\langle Tv,w \rangle = \langle v,T^*w \rangle.$$
The matrix $M$ of $T^*$ with respect to our standard basis is obtained from the matrix of $T$ by reflection around the anti-diagonal. In particular, we have the identity $\det(T) = \det(T^*)$.
We say that $T$ is {\it self-adjoint} if $T = T^*$. Any self-adjoint transformation $T$ defines a new symmetric bilinear form on~$W$ by the formula
$$\langle v,w \rangle_T = \langle v,Tw \rangle.$$
This bilinear form has symmetric Gram matrix $B = AM$ in our standard basis.

We say that a $\Z$-linear transformation $g: W \rightarrow W$ is {\it orthogonal} if it preserves the bilinear structure on $W$, i.e., $\langle gv,gw \rangle = \langle v,w \rangle$ for all $w\in W$.
In that case $g$ is invertible, with $g^{-1} = g^*$, and $\det(g) = \pm1$. We define the group scheme $G = \SO(W)$ over $\Z$ by
$$ G := \SO(W) = \{g \in \GL(W): g^*g = 1, ~ \det(g) = 1\}.$$

Over $\Z[1/2]$ the group scheme $G$ is smooth and reductive. It defines the split reductive adjoint group with root system of type $B_n$ (see~\cite{SGA}). In particular, for every field $k$ with char$(k) \neq 2$, the group $G(k)$ gives the points of the split special orthogonal group $\SO_{2n+1}$ of the space $W \otimes k$.  

We note that the group $G$ is not smooth over $\Z_2$. Indeed, the involution $g\in G(\Z/2\Z)$ with matrix
\begin{equation}\label{defaa}
\left[\begin{array}{ccccc}
 \,1 & & & & \,1 \\
 & \,1 & & & \\
 & &\ddots & & \\
  & &  & \,1 & \\
 & & & &\,1
\end{array}\right]
\end{equation}
does not lift to an element of $G(\Z_2)$, or even to an element of $G(\Z/4\Z)$. The theory of Bruhat and Tits provides a smooth model $\G$ over $\Z_2$ with the same points in \'etale $\Z_2$-algebras (see~\cite{T}), namely, $\G$ is the normalizer of a parahoric subgroup $P$ of $G(\Q_2)$. When $n=1$, the subgroup~$P$ is an Iwahori subgroup, and when $n\geq2$, the subgroup $P$ is a maximal parahoric subgroup. The reductive quotient of $P$ is the split group $\SO_{2n}$. We will not need this smooth model in what follows.

If $g$ is orthogonal and $T$ is self-adjoint, then the transformation $gTg^{-1} = gTg^*$ is also self-adjoint. The self-adjoint operators form a free submodule $Y$ of $\End(W)$, of rank $(2n+1)(n+1)$, and the action of special orthogonal transformations by conjugation $T \to gTg^{-1}$  gives a  linear representation $G  \rightarrow \GL(Y)$.
Over $\Z[1/2]$ the representation $Y$ is orthogonal, via the non-degenerate symmetric pairing $\langle T_1,T_2 \rangle = \Trace(T_1T_2)$. If we identify the module of self-adjoint operators with the module of symmetric matrices, via the map $B = AM$ described above, then $Y$ is isomorphic to the symmetric square of the standard representation $W$ over $\Z[1/2]$.

The trace of an operator gives a nonzero, $G$-invariant linear form $Y \rightarrow \Z$. We define the representation $V$ of $G$ over $\Z$ as its kernel, i.e., 
$$V = \{T
\in \End(W) \,|\,\,T = T^*, \,\Trace(T) = 0\}.$$
The free $\Z$-module $V(\Z)$ has rank $2n^2 + 3n$, and can be identified with the submodule of symmetric matrices having anti-trace zero (where the anti-trace is the sum of the matrix entries on the anti-diagonal). We are going to study the orbits of $G(\Z)$ on $V(\Z)$.

We note that the representation $V$ of $G$ is algebraically irreducible in all characteristics which do not divide $2$ or $2n+1$. If the characteristic divides $2n+1$, then $V$  contains a copy of the trivial representation (as the scalar multiples of the identity operator $I$ are self-adjoint and have trace zero) and the quotient module is algebraically irreducible. In characteristic $2$, there are even more irreducible factors in the Jordan-H\"older series of $V$.

\section{Orbits with a fixed characteristic polynomial}

Let $k$ be a field, with char$(k) \neq 2$. In this section and the next, we consider the orbits of the group $G(k) = \SO(W)(k)$ on the representation $V(k) = V \otimes k$. A discussion of these orbits can also be found in our expository paper \cite{BG}.

Since the group $G(k)$ acts on the operators $T$ in $V(k)$ by conjugation $T \to gTg^{-1}$, the characteristic polynomial of $T$ is an invariant of its orbit. Since $T$ has trace zero, this monic characteristic polynomial has the form
\begin{equation}\label{fdef}
f(x) = \det(xI - T) =  x^{2n+1} + c_2x^{2n-1} + c_3x^{2n-2} + \cdots + c_{2n+1}
\end{equation}
with coefficients $c_m\in k$. The coefficients $c_m$ are polynomial invariants of the representation, with $\deg(c_m) = m$. These invariants are algebraically independent, and generate the full ring of polynomial invariants on $V$ (cf.~\cite{BG}).
An important polynomial invariant, of degree $2n(2n+1)$, is the multiple
$$\Delta = \Delta(c_2,c_3,\ldots,c_{2n+1}) = 4^{2n}\disc f(x)$$
of the discriminant of the characteristic polynomial of $T$, which gives the discriminant $\Delta$ of the associated hyperelliptic curve $y^2=f(x)$ (see~\cite{Lockhart}).
The discriminant $\Delta(T)$ is nonzero precisely when the polynomial $f(x)$ is separable over $k$, so has $2n+1$ distinct roots in a separable closure $k^s$.

Since the characteristic polynomial is an invariant of the $G(k)$-orbit, it makes sense to describe the set of orbits with a fixed characteristic polynomial $f(x)$.  We begin by showing that, for any separable $f(x)$, there is a naturally associated {orbit} of $G(k)$ on $V(k)$, which we call the {\it distinguished orbit}, that has associated characteristic polynomial $f(x)$.

\subsection{The distinguished orbit}\label{distsec}

We assume throughout that $f(x)$ is {\it separable} (i.e., $\disc(f)\neq 0)$ and is of the form (\ref{fdef}).

\begin{proposition}\label{simpletransitive}
The group $G(k)$ acts simply transitively on the pairs $(T,X)$, where $T$ is a self-adjoint operator in $V(k)$ with characteristic polynomial $f(x)$ and $X \subset W(k)$ is a maximal isotropic subspace of dimension $n$ with $T(X) \subset X^{\perp}$.
\end{proposition}

The action of $g$ on pairs $(T,X)$ is by the formula $(gTg^{-1}, g(X))$. As a corollary, we can identify a {\it distinguished orbit} of $G(k)$ on the set of self-adjoint operators with characteristic polynomial $f(x)$, namely, the orbit consisting of those $T$ for which a maximal isotropic subspace $X$ exists over $k$ satisfying $T(X) \subset X^{\perp}$. We will frequently use $S$ to denote a self-adjoint operator in this distinguished~orbit. 

Next, let $L = k[x]/ (f(x))$. Since we have assumed that $f(x)$ is separable, $L$ is an \'etale $k$-algebra of rank $2n+1$.

\begin{proposition}\label{stabilizer}
The stabilizer $G_S$ of a vector $S$ in the distinguished orbit with characteristic polynomial $f(x)$ is $($uniquely$)$ isomorphic to the finite commutative \'etale group scheme $D$ of order $2^{2n}$ over $k$, which is the kernel of the norm map $\Res_{L/k}(\mu_2)\rightarrow \mu_2$.
\end{proposition}

To prove these results, we need a method to construct orbits over $k$ with a fixed separable characteristic polynomial $f(x)$. 
Let $\beta$ be the image of $x$ in $L$, so a power basis of $L$ over $k$ is given by $\{1,\beta,\beta^2,\ldots,\beta^{2n}\}$. We define a symmetric bilinear form $(\;,\,)$ on $L$ by the formula
\begin{equation}\label{lndef}
(\lambda, \nu) :=\mbox{ the coefficient of $\beta^{2n}$ in the product $\lambda\nu$.}
\end{equation}
 Since $f(x)$ is separable, it is relatively prime to the derivative $f'(x)$ in $k[x]$ and the value $f'(\beta)$ is a unit in $L^*$. We then have another formula for this inner product which is due to Euler~\cite[Ch III, \S6]{S2}: 
$$ (\lambda,\nu) = \Trace_{L/k}(\lambda\nu/f'(\beta)).$$

The orthogonal space we have defined on $L$ is non-degenerate, of dimension $2n+1$ and discriminant $\equiv 1$ over $k$. Indeed, the entries of its Gram matrix with respect to the power basis are $0$ above the anti-diagonal, and $1$ on the anti-diagonal.
We next observe that the $k$-subspace $M$ of $L$ spanned by $\{1,\beta,\beta^2,\ldots,\beta^{n-1}\}$ has dimension $n$ and is isotropic for the bilinear form, so the orthogonal space $L$ is split \cite{Scharlau}. Hence we have an isometry $\theta: L \rightarrow W(k)$ defined over $k$. This isometry is well-defined up to composition with an orthogonal transformation $g$ of $W(k)$---any other isometry has the form $\theta' = g\theta$ where $g$ is an element of $\O(W) = \SO(W) \times \langle\pm I\rangle$.

The $k$-linear map $\beta: L \rightarrow L$ given by multiplication by $\beta$ is a self-adjoint transformation with characteristic polynomial equal to $f(x)$. Hence the operator $T = \theta\beta\theta^{-1}: W(k) \rightarrow W(k)$ is self-adjoint with the same characteristic polynomial. If we compose the isometry $\theta$ with the orthogonal transformation $g$, we obtain an operator $gTg^{-1}$ in the same $G(k) = \SO(W)(k)$ orbit. This follows from the fact that the dimension of $W$ is odd, so the orthogonal group and its special orthogonal subgroup have the same orbits on the representation $V$.

This orbit we have just constructed is distinguished, as $\beta(M)$ is contained in $M^{\perp} = \{1,\beta,\ldots,\beta^n\}$. Hence, $X = \theta(M)$ is a maximal isotropic subspace of $W$ over $k$ with $T(X) \subset X^{\perp}$, and $(T,X)$ is a distinguished pair in $V$, defined over $k$. 

To see that $G(k)$ acts simply-transitively on the pairs $(T,X)$, we first show that this is true when the field $k = k^s$ is {\it separably closed}. Since $f(x)$ is separable, it is also the minimal polynomial of $T$ on $W$. By the cyclic decomposition theorem, the centralizer of $T$ in $\End(W)$ is the algebra $k[T]$ generated by $T$, which is isomorphic to $L$. When $k = k^s$ is separably closed, this algebra is isomorphic to the product of $2n+1$ copies of the field $k$. If $T'$ has the same characteristic polynomial, then $T' = hTh^{-1}$, where $h\in\GL(W)(k)$ is well-defined up to right multiplication by $(k[T])^*$. Since both operators $T,T'$ are self-adjoint, the composition $h^*h$ centralizes $T$, so $\alpha = h^*h$ lies in $(k[T])^*$. In the separably closed case, every unit in the algebra $k[T]$ is a square, so we may modify $h$ by a square root of $\alpha$ to arrange that $h^*h = 1$. Modifying $h$ by $\pm I$ if necessary, we may also arrange that $\det(h) = 1$. It follows that $T'$ is in the same $G(k)$ orbit as $T$. 

The subgroup $G_T(k)$ stabilizing $T$ is the intersection of two subgroups of $\GL(W)(k)$. The first is the centralizer $(k[T])^*$, where all the operators are self-adjoint ($g^* = g$). The second is the subgroup $\SO(W)(k)$, where all the operators are special orthogonal ($g^* = g^{-1}, ~\det(g) = 1$) . Hence $G_T(k)$ is isomorphic to the group
$$D = \{g \in L^*: g^2 = 1, ~ N(g) =1\}.$$

Since $L$ is the product of $2n+1$ copies of $k$, the group $D$ is an elementary abelian $2$-group of order~$2^{2n}$. For future reference, we record what we have just proved.

\begin{proposition}\label{sepclosed}
When $k = k^s$ is separably closed, the group $G(k) = \SO(W)(k)$ acts transitively on the set of self-adjoint operators with a fixed separable characteristic polynomial $f(x)$. The stabilizer $G_T(k)$ of an operator $T$ in this orbit is an elementary abelian $2$-group of order $2^{2n}$.
\end{proposition}

We remark that $G_T(k^s)$ embeds as a Jordan subgroup of $\SO(W)(k^s)$, stabilizing a decomposition of $W$ into a direct sum of orthogonal lines. This subgroup is unique up to conjugacy, and is not contained in any maximal torus. It is its own centralizer, and its normalizer is a finite group isomorphic to the Weyl group $W(D_{2n+1}) = 2^{2n}.S_{2n+1}$, where $S_{2n+1}$ is the symmetric group (see~\cite{KT}). 

We continue with the proof of Proposition~\ref{simpletransitive} in the case when $k$ is separably closed. The finite stabilizer $G_T(k)$ acts on the set of maximal isotropic subspaces $X$ in $V$ with $T(X)$ contained in $X^{\perp}$, and we must show that it acts simply transitively.  There are precisely $2^{2n}$ such subspaces $X$ defined over an algebraic closure of $k$: these are the points of the Fano variety of the complete intersection of the two quadrics 
$$Q(w) = \langle w,w\rangle, ~~~ Q_T(w) = \langle w,Tw\rangle$$
in $\P(W)$. For an excellent discussion of the Fano variety of the complete intersection of two quadrics over the complex numbers, see~\cite{D}. This material was also treated in the unpublished Cambridge Ph.D.\ thesis of M.\ Reid~\cite{R}. A full treatment of the arithmetic theory sketched below will appear in the Harvard Ph.D.\ thesis of X.\ Wang~\cite{W}.

The pencil $xQ - x'Q_T$ is non-degenerate. Using the matrices in our standard basis of $W$ we have the formula
$$\disc(xA - B) = (-1)^n\det(xA - AM) = (-1)^n\det(A)\det(xI - M) =  f(x)$$
and $f(x)$ is separable. To show that these subspaces $X$ of $W$ are all defined over the separably closed field $k$, and that $G_T(k)$ acts simply-transitively on them, we consider one such isotropic subspace in the model orthogonal space $L$, with form $(\lambda, \nu)$ given by (\ref{lndef}) and the self-adjoint operator given by multiplication by $\beta$. The subspace $M$ with basis $\{1,\beta,\ldots,\beta^{n-1}\}$ is maximal isotropic, and $\beta(M)$ is contained in $M^{\perp}$. The subgroup of $L^*$ which fixes $M$ under multiplication is clearly $k^*$. Hence the subgroup of $D$ which fixes  $M$ is trivial and the finite group $D$ must act simply-transitively on the $2^{2n}$ such subspaces (as it has the same order). This also shows that all of the $2^{2n}$ points of the Fano variety are defined over the separably closed field $k$. Translating this to $W$ via the isometry $\theta$ completes the proof of Proposition~\ref{simpletransitive} when the base field $k$ is separably closed.

For general fields $k$, we use Galois descent. Let $k^s$ denote a separable closure of $k$ and let $\Gal(k^s/k)$ denote the Galois group. Consider two distinguished pairs $(T,X)$ and $(T',X')$ which are defined over $k$. By the simple transitivity result we have just proved over $k^s$, there is a unique element $g$ in the group $\SO(W)(k^s)$ with $g(T,X) = (T',X')$. For any $\sigma$ in $\Gal(k^s/k)$, the element $g^{\sigma}$ has the same property. Hence $g = g^{\sigma}$ lies in $\SO(W)(k)$, which  completes the proof of Proposition~\ref{simpletransitive} for general $k$.

To determine the stabilizer of a vector in the distinguished orbit, as a finite group scheme over~$k$, we again may use the model orthogonal space $L$ with inner product $(\lambda, \nu)$ given by (\ref{lndef}) and the distinguished self-adjoint operator given by multiplication by $\beta$. The centralizer $D$ of $\beta$ has points
$$D(E) =  \{g \in (L\otimes E)^*: g^2 = 1, ~ N(g) =1\}.$$
in any $k$-algebra $E$.
Hence $D = \Res_{L/k}(\mu_2)_{N = 1}$ is the finite group scheme over $k$ described in Proposition \ref{stabilizer}. The choice of isometry $\theta: L \rightarrow W$ gives a vector $S = \theta\beta\theta^{-1}$ in the distinguished orbit and an isomorphism of stabilizers $G_S = \theta D \theta^{-1}$. Since $\theta$ is well-defined up to left multiplication by $G_S$, which is a commutative group scheme, the isomorphism with $D$ is uniquely defined. This completes the proof of Proposition~\ref{stabilizer}.


Here is a condition on the Gram matrix $B$ of the bilinear form $\langle v,Tw\rangle$ which implies that the orbit of $T$ is distinguished.

\begin{proposition} \label{dist}  An orbit of $\SO(W)(k)$ on the representation $V(k)$ having nonzero discriminant is distinguished if and only if it contains a vector $T$ whose Gram matrix $B = AM$ with respect to the standard basis 
has the form
\begin{equation}\label{distform}
\begin{pmatrix} 0 & 0 &\cdots&0& 0 & \ast & \ast \\
0 & 0 &\cdots&0& \ast & \ast & \ast \\
\vdots & \vdots&\iddots & \iddots & \vdots & \vdots&\vdots \\
0 & 0 & \iddots & \cdots & \vdots & \vdots  & \vdots \\
0 & \ast &\cdots &\cdots&\ast&\ast & \ast\\
\ast & \ast&\cdots&\cdots&\ast&\ast&\ast\\
\ast &\ast&\cdots&\cdots&\ast&\ast&\ast\\
\end{pmatrix}.
\end{equation}
\end{proposition}

\begin{proof}
If the Gram matrix $B$ of $T$ has the form (\ref{distform}), then the subspace $X$ spanned by the basis elements $e_1$, $e_2$,~$\ldots$\,, $e_n$ is isotropic and $T(X)$ is contained in $X^{\perp}$, which is spanned by the elements $e_1,e_2,\ldots,e_n,u$. Hence the orbit is distinguished. 

Conversely, if the operator $T$ lies in a distinguished orbit, we may identify~$W$ with the space $L = k[x]/ (f(x))$ having bilinear form $(\lambda,\nu)$ given by (\ref{lndef}), via an isometry $\theta: L \rightarrow W$. The self-adjoint operator $T = \theta\beta\theta^{-1}$ has Gram matrix of the form (\ref{distform}) with respect to the basis $1,\beta,\beta^2,\ldots,\beta^{2n}$ of $L$. It also has this form with respect to the basis $1,p_1(\beta),p_2(\beta),\ldots,p_{2n}(\beta)$ for any monic polynomials $p_i(x)$ of degree $i$.
But we can choose the monic polynomials $p_i(x)$ so that the inner products $(p_i(\beta),p_j(\beta))$ on $L$ match the inner products of our standard basis elements on $W$. There is then an isometry
$\theta' = g\theta: L \rightarrow V$ taking this basis to our standard basis, and the vector $\theta'\beta\theta'^{-1} = gTg^{-1}$ in the orbit of $T$ has Gram matrix $B$ of the desired form. 
\end{proof}

\subsection{The remaining orbits via Galois cohomology}

To describe the remaining orbits with a fixed separable characteristic polynomial, we use Galois cohomology. Let $k^s$ denote a separable closure of $k$. If $M$ is a commutative finite \'etale group scheme over $k$ we let $H^i(k,M) = H^i(\Gal(k^s/k),M(k^s))$ be the corresponding Galois cohomology groups. Similarly, if $J$ is a smooth algebraic group over $k$, we let $H^1(k,J) = H^1(\Gal(k^s/k),J(k^s))$ be the corresponding pointed set of first cohomology classes~\cite[Ch I, \S5]{S}. If we have a homomorphism $M \rightarrow J$ over $k$, we obtain an induced map of pointed sets $H^1(k,M) \rightarrow H^1(k,J)$. The kernel of this map is by definition the subset of classes in $H^1(k,M)$ that map to the trivial class in $H^1(k,J)$.

Let $S$ be a self-adjoint operator with characteristic polynomial $f(x)$ in the distinguished orbit of $G = \SO(W)$ over $k$, and let $G_S$ be the finite \'etale subgroup stabilizing $S$. Let $T$ be any self-adjoint operator with this characteristic polynomial over $k$. By Proposition~\ref{sepclosed}, the vector $T$ is in the same orbit as $S$ over the separable closure $k^s$, so $T = gSg^{-1}$ for some $g \in \SO(W)(k^s)$. For any $\sigma\in\Gal(k^s/k)$, the element $c_{\sigma} = g^{-1}g^{\sigma}$ lies in the subgroup $G_S(k^s)$, and the map $\sigma \rightarrow c_{\sigma}$ defines a $1$-cocycle on the Galois group with values in $G_S(k^s)$, whose cohomology class depends only on the $G(k)$-orbit of $T$. This class clearly lies in the kernel of the map 
\begin{equation}\label{mapeta}
\eta: H^1(k,G_S) \rightarrow H^1(k, G).
\end{equation}
Conversely, if we have a cocycle $c_{\sigma}$ whose cohomology class lies in the kernel of $\eta$, then it has the form $g^{-1}g^{\sigma}$ and the operator $T = gSg^{-1}$ is defined on $W$ over $k$. Since $T$ has the same characteristic polynomial as $S$, we have established the following.

\begin{proposition}\label{cohomology}
The map which takes an operator $T$ to the cohomology class of the cocycle $c_{\sigma}$ induces a bijection  between the set of orbits of $G(k)$ on self-adjoint operators $T$ with characteristic polynomial $f(x)$ and the set of elements in the kernel of the map
$\eta: H^1(k,G_S) \rightarrow H^1(k, G).$
The trivial class in $H^1(k,G_S)$ corresponds to the distinguished orbit.
\end{proposition}

We will make the map $\eta$ more explicit shortly, via a computation of the pointed cohomology sets of $G_S$ and $G$. But we can use our discussion of the Fano variety to give an explicit principal homogeneous space $P_T$ for the finite group scheme $G_S$ over $k$, corresponding to the self-adjoint operator $T$. Choose $g\in G(k^s)/G_S(k^s)$ with $gSg^{-1} = T$. Since $G_S$ is abelian we obtain a unique isomorphism of the stablizers $G_S \cong G_T$ over $k^s$, which is given by conjugation by any $g$ in this coset. The unicity implies that this isomorphism is defined over $k$. It therefore suffices to construct a principal homogeneous space $P_T$ for the finite group scheme $G_T$. We let $P_T$ be the set of maximal isotropic subspaces $X \subset W$ defined over $k^s$ with the property that $T(X) \subset X^{\perp}$. This set is finite, of cardinality $2^{2n}$. The Galois group $\Gal(k^s/k)$ acts, so $P_T$ has the structure of a reduced scheme of dimension zero over $k$ and forms a principal homogeneous space for the group scheme $G_T$.  The isomorphism class of the principal homogeneous space $P_T$ over $k$ depends only on the orbit of $T$, and corresponds to the cohomology class of the cocycle $c_{\sigma}$. Note that the scheme $P_T$ has a $k$-rational point if and only if $T$ lies in the distinguished orbit.

We now make the above map $\eta$ of pointed sets more explicit, using our model orthogonal space $L$ with form $(\lambda,\nu)$ given by (\ref{lndef}).  First, since $2n+1$ is odd, we have a splitting
$$\Res_{L/k}(\mu_2) = \mu_2 \oplus D.$$
This allows us to compute the Galois cohomology groups of $D$ using Kummer theory. We find that
\begin{eqnarray*}
H^0(k, D) &=& L^*[2]_{N=1} \\[.075in]
H^1(k,D) &=& (L^*/L^{*2})_{N\equiv1}
\end{eqnarray*}
where we use the subscript $N\equiv1$ to mean the subgroup of elements in $L^*/L^{*2}$ whose norms are squares in $k^*$.

Let $\alpha$ be a class in $H^1(k,D)$, viewed as an element of the group $L^*$ whose norm to $k^*$ is a square.
The class $\eta(\alpha)$ lies in the pointed set $H^1(k,\SO(L))$, which classifies non-degenerate orthogonal spaces of the same dimension $2n+1$ and discriminant $\equiv 1$ as our model space $L$ over $k$ (see~\cite{KMRT}). Unwinding the definitions as in \cite{BG}, we find that the orthogonal space corresponding to the class $\eta(\alpha)$ is isometric to $L$ with the modified bilinear form $(\lambda,\nu)_{\alpha} =$ the coefficient of $\beta^{2n}$ in the product $\alpha\lambda\nu$. Equivalently, we have the formula:
$$(\lambda,\nu)_{\alpha} = \Trace_{L/k}(\alpha\lambda\nu/ f'(\beta)).$$
We note that the isomorphism class of this space depends only on the element $\alpha$ in the quotient group $(L^*/L^{*2})$, and that the condition that the norm of $\alpha$ is a square in $k^*$ translates into the fact that the discriminant of the form $(\lambda,\nu)_{\alpha}$ is $\equiv 1$. 

If we translate this over to the orthogonal space $W$ via the isometry $\theta$, we see that a class $\alpha$ in
the group $H^1(k,G_S) = H^1(k,D) = (L^*/L^{*2})_{N\equiv1}$ lies in the kernel of the map
$\eta: H^1(k,G_S) \rightarrow H^1(k,\SO(W))$ precisely when the form $(\lambda,\nu)_{\alpha}$ on
$L$ gives a split orthogonal space over $k$. In that case $\alpha$ gives an orbit of self-adjoint operators $T$ on $W$, and the principal homogeneous space $P_T$ splits over the extension of $L$ where $\alpha$ is a square. We will produce examples of non-trivial $\alpha$ where the bilinear form $(\lambda,\nu)_{\alpha}$ is split in the next section, using the Jacobians of hyperelliptic curves.

\section{Hyperelliptic curves}

We now use some results in algebraic geometry, on hyperelliptic curves and the Fano variety of the complete intersection of two quadrics in $\mathbb P(W \oplus k)$, to produce classes in the kernel of the map in cohomology $\eta: H^1(k,G_S) \rightarrow H^1(k,\SO(W))$, and hence orbits with a fixed separable characteristic polynomial $f(x)$.

Let $C$ be the smooth projective hyperelliptic curve of genus $n$ over $k$ with affine equation
$y^2 = f(x)$ and a $k$-rational Weierstrass point $O$ above $x = \infty$. The functions on $C$ which are regular outside of $O$ form an integral domain:
$$H^0(C - O, \mathscr O_{C - O}) = k[x,y]/ (y^2 = f(x)) = k[x, \sqrt{f(x)}].$$
The complete curve $C$ is covered by this affine open subset $U_1$, together with the affine open
subset $U_2$ associated to the equation $w^2 = v^{2n+2}f(1/v) = c_{2n+1}v^{2n+2} + \cdots + c_2v^3 + v$ and containing the point $O = (0,0)$. The gluing of $U_1$ and $U_2$ is by $(v,w) = (1/x, y/x^{n+1})$ and $(x,y) = (1/v,w/v^{n+1})$ wherever these maps are defined.

\subsection{The $2$-torsion subgroup}

Let $J = \Pic^0(C)$ denote the Jacobian of the hyperelliptic curve $C$ over $k$, which is an abelian variety of dimension $n$ over $k$, and let $J[2]$ denote the $2$-torsion subgroup of $J$. This is a finite \'etale group scheme of order $2^{2n}$. Recall that we have defined a finite group scheme $D =(\Res_{L/k}\mu_2)_{N=1}$ of this order, associated to the \'etale algebra $L = k[x]/ (f(x))$. We begin by constructing an isomorphism $D \cong J[2]$ of finite group schemes over $k$.

The other Weierstrass points $P_{\gamma}\neq O$ of $C(k^s)$ correspond bijectively to $k$-algebra homomorphisms $\gamma: L \rightarrow k^s$: we have $x(P_{\gamma}) = \gamma(\beta)$ and
$y(P_{\gamma}) = 0$. Associated to such a point we have the divisor $d_{\gamma} = (P_{\gamma}) - (O)$ of degree zero. The divisor class of $d_{\gamma}$ lies in the $2$-torsion subgroup
$J[2](k^s)$ of the Jacobian, as $2(d_{\gamma})$ is the divisor of the function  $(x - \gamma(\beta))$. The Riemann-Roch theorem shows that the classes $d_{\gamma}$ generate the finite group $J[2](k^s)$, and satisfy the single relation $\sum(d_{\gamma}) \equiv 0.$ Indeed, this sum is the divisor of the function $y$. This gives an isomorphism $\Res_{L/k}(\mu_2)/\mu_2 \cong J[2]$.  We recall that $D$ is defined as the kernel of the map $\Res_{L/k}(\mu_2)\rightarrow\mu_2$. 

Since the degree of $L$ over $k$ is odd, the composition $D \rightarrow           
\Res_{L/k}(\mu_2) \rightarrow \Res_{L/k}(\mu_2)/\mu_2$ is an isomorphism. We      
summarize what we have proved:
                                                                     
\begin{proposition}\label{isom}                                                          
There is an isomorphism of finite group schemes $G_S \rightarrow J[2]$ over 
$k$, where $G_S$ is the stabilizer of a point $S$ in the distinguished orbit 
with characteristic polynomial $f(x)$ and $J[2]$ is the $2$-torsion subgroup 
of the Jacobian of the hyperelliptic curve with equation $y^2 = f(x)$. Both 
group schemes are isomorphic to $D = (\Res_{L/k}(\mu_2))_{N=1};$ their points 
over any extension field $E$ of $k$ correspond bijectively to the monic 
factorizations $f(x) = g(x)h(x)$ of $f(x)$ over $E$. 
\end{proposition}                                                      

The exact sequence $0 \rightarrow J[2] \rightarrow J \rightarrow J \rightarrow 0$ of commutative group schemes over $k^s$ gives an exact descent sequence in Galois cohomology:
$$0 \rightarrow J(k)/2J(k) \rightarrow H^1(k, J[2]) \rightarrow H^1(k,J)[2] \rightarrow 0.$$
The middle term $H^1(k, J[2])$ can be identified with the group $H^1(k,D) =  (L^*/L^{*2})_{N\equiv1}$ via the above isomorphism. We need an explicit description of the coboundary map  
$$\delta: J(k)/2J(k) \rightarrow  (L^*/L^{*2})_{N\equiv1}.$$
If $P = (a,b)$ is a point of $C$ over $k$ with $b \neq 0$, and $D =
(P) - (O)$ is the corresponding divisor class in $J(k)$, then (cf.\
\cite{Schaefer}, \cite[p.\ 7]{Stoll}) we have
$$\delta(D) \equiv (a - \beta).$$

\subsection{A subgroup in the kernel of $\eta$}

For any self-adjoint operator $S$ in the distinguished orbit, we have constructed an isomorphism from $D$ to the stabilizer $G_S$ in $\SO(W)$. Hence we may identify the cohomology groups $H^1(k,J[2]) = H^1(k,G_S)$. With this identification, we have the following important result.



\begin{proposition}\label{kernel}                                                           
The subgroup $J(k)/2J(k)$ of  $H^1(k,J[2])$ lies in the kernel of the map 
$\eta: H^1(k,D) \rightarrow H^1(k,\SO(W))$. The classes 
$\alpha$ which lie in $J(k)/2J(k)$ correspond bijectively to the orbits of
$\SO(W)(k)$ on $V(k)$ with characteristic polynomial $f(x)$, where the
homogeneous space $F_{\alpha}$ for the Jacobian $J$ $($to be defined below$)$ has 
a $k$-rational point.
\end{proposition} 

\begin{proof}                                                                     
This is proved in \cite[Prop.\ 6]{BG}, and we briefly recall the argument here.     
We first need to make the homomorphism from $H^1(k,D) = H^1(k,J[2])$ to           
$H^1(k,J)[2]$ which arises in the descent sequence more explicit. That is, we     
need to associate to each class $\alpha$ in the group $H^1(k,D) =                 
(L^*/L^{*2})_{N\equiv1}$                                                     
a principal homogeneous space $F_{\alpha}$ of order $2$ for the Jacobian $J$      
over $k$. The classes in the subgroup $J(k)/2J(k)$ will correspond to the         
homogeneous spaces with a $k$-rational point. We will use the existence of        
$k$-rational points on $F_{\alpha}$ to show that the bilinear form                
$(\;,\,)_{\alpha}$ on $L$ is split. Thus every $\alpha$ in $J(k)/2J(k)$           
corresponds to a class in the kernel of $\eta$, and hence to an orbit of          
$\SO(W)(k)$ on $V(k)$.                                                            

More generally, let $\alpha$ be any class in $H^1(k,D) =
(L^*/L^{*2})_{N\equiv1}$. The vector space $L \oplus k$ has dimension
$2n+2$ over $k$. Consider the two quadrics on $\mathbb P(L \oplus k)$:
$$Q(\lambda,a) = (\lambda,\lambda)_{\alpha} =  \Trace(\alpha\lambda^2/ f'(\beta)),$$
$$ Q'(\lambda,a) = (\lambda, \beta\lambda)_{\alpha} + a^2 =  \Trace(\alpha\beta\lambda^2/ f'(\beta)) + a^2.$$
The pencil $uQ - vQ'$ is non-degenerate and contains exactly $2n+2$ singular elements: the quadric $Q$ at $v=0$ and the $2n+1$ quadrics $\gamma Q - Q'$ at the points  $(\gamma, 1)$ where $f(\gamma) = 0$.  More precisely, the discriminant of the pencil is given by the binary form
$$\disc(uQ - vQ') = v^{2n+2}f(u/v)$$
of degree $2n+2$ over $k$. Hence the base locus $Q \cap Q'$ of the pencil is non-singular. The Fano variety $F = F_{\alpha}$ of this complete intersection, consisting of the $n$-dimensional subspaces $Z$ of the $2n+2$ dimensional space $L \oplus k$ which are isotropic for all of the quadrics in the pencil, is smooth of dimension~$n$ (see~\cite{D}). 

Since the discriminant of the  the pencil defines the hyperelliptic curve $C$ over $k$, a point $c = (x,y)$ on the curve determines both a quadric $Q_x = xQ - Q'$ in the pencil together with a ruling of $Q_x$. (A {\it ruling} of $Q_x$ is, by definition, a component of the variety of $n+1$ dimensional subspaces in $L \oplus k$ which are isotropic for $Q_x$.) Each point of $C$ gives an involution of the corresponding Fano variety $\tau(c): F \rightarrow F$,  which has $2^{2n}$ fixed points over a separable closure $k^s$, and is defined as follows. A point of $F$ over $k^s$ consists of a common isotropic subspace $Z$ of dimension $n$ for all quadrics in the pencil. The point $c$ then gives a fixed maximal isotropic subspace $Y$ of dimension $n+1$ for the quadric $Q_x$ which contains $Z$. If we restrict any non-singular quadric other than $Q_x$ in the pencil to the subspace $Y$, we get a reducible quadric which is the sum of two hyperplanes: $Z$ and another common isotropic subspace $\tau(c)(Z)$. This defines the involution, and hence a morphism $\phi: C \times F \rightarrow F$ with $\phi(c,Z) = \tau(c)(Z)$.

The morphism $C \times F \rightarrow F$ can be used to define a fixed-point free action of the Jacobian $J = \Pic^0(C)$ on $F$ over $k$, which gives $F$ the structure of a principal homogeneous space for $J$. For a general non-degenerate pencil of quadrics, this principal homogeneous space has order $4$ in the group $H^1(k,J)$, and its square is the principal homogeneous space $\Pic^1(C)$. In this case, we have a $k$-rational Weierstrass point $O$ on $C$, which corresponds to the degenerate quadric $Q$ in the pencil. Hence $\Pic^1(C)$ has a rational point and the homogenous space $F_{\alpha}$ has order $2$ in $H^1(k,J)$. 

We can prove that $F_\alpha$ represents the image of $\alpha$ in the descent sequence as follows. The involution $\tau = \tau(O)$ is induced by the linear involution
$(\lambda, a) \to (\lambda, -a)$.  Hence the fixed points $P_{\alpha}$ of $\tau$ are just the $n$-dimensional subspaces $X$ of $L$ over $k^s$ which are isotropic for both quadrics $Q$ and $Q'$. This is the Fano variety $P_T = P_{\alpha}$ of dimension zero, which is the principal homogeneous space for $G_T = J[2]$ associated to the cohomology class $\alpha$. It follows that $F_{\alpha}$ is the image of $\alpha$ in $H^1(k,J)[2]$. For more details on this construction, see the Harvard Ph.D.\ thesis of X. Wang \cite{W}.  Note that the variety $F_{\alpha}$ has a $k$-rational point whenever $\alpha$ lies in the subgroup $J(k)/2J(k)$, but $F_{\alpha}$ only has a $k$-rational point fixed by the involution $\tau$ when $\alpha$ is the trivial class in $H^1(k,J[2])$.

We may now prove Proposition~\ref{kernel}. Assume that the class $\alpha$ is in the subgroup $J(k)/2J(k)$. Then its image in $H^1(k,J)[2]$ is trivial, and the homogenous space $F_{\alpha}$ has a $k$-rational point. Hence there is a $k$-subspace $Z$ of $L \oplus k$ which is isotropic for both $Q$ and $Q'$. Since it is isotropic for $Q'$, the subspace~$Z$ does not contain the line $0 \oplus k$. Its projection to the subspace~$L$ has dimension $n$ and is isotropic for~$Q$. This implies that the orthogonal space $L$ with bilinear form $(\lambda,\nu)_{\alpha} = \Trace(\alpha\lambda\nu/ f'(\beta))$ is split, so the class $\alpha$ is in the kernel of the map $\eta$.
\end{proof}

\noindent
We call such special $G(k)$-orbits on $V(k)$ (and the elements in them) {\it soluble}.  Thus Proposition~\ref{kernel} states that the $\SO(W)(k)$-orbits on the set of  soluble elements in $V(k)$ with characteristic polynomial~$f(x)$ are in bijection with the elements of $J(k)/2J(k)$. 

\begin{remark}{\em 
We note that the finite group scheme $D = J[2]$ does not determine the hyperelliptic curve~$C$. Indeed, for any class $d \in k^*/k^{*2}$, one has the hyperelliptic curve $C_d$ with equation $d y^2 = f(x)$ and Jacobian $J_d$. The determination of $J[2]$ shows that these Jacobians share the same $2$-torsion subgroup: $J_d[2] \cong J[2]$. This Jacobian $J_d$ acts simply transitively on the Fano variety $F_{\alpha,d}$ of the complete intersection of the two quadrics
$$Q(\lambda,a) = \Trace(\alpha\lambda^2/ f'(\beta)),$$
$$ Q'(\lambda,a) = \Trace(\alpha\beta\lambda^2/ f'(\beta)) + da^2.$$
Indeed, the discriminant of the quadric $uQ - vQ'$ in the pencil is equal to $dv^{2n+2} f(u/v)$. A similar argument to the one given in the proof of Proposition~\ref{kernel} shows that the subgroup $J_d(k)/2J_d(k)$ of $H^1(k,D)$ is also contained in the kernel of the map $\eta: H^1(k,D) \rightarrow H^1(k, \SO(W)).$}
\end{remark}

\section{Orbits over arithmetic fields}

In this section, we describe the orbits with a fixed characteristic polynomial, when $k$ is a finite, local, or global field.

\subsection{Orbits over a finite field}\label{finite}

We first consider the case when $k$ is finite, of odd order $q$. In this case, every non-degenerate quadratic space of odd dimension is split, so $H^1(k,\SO(W)) = 1$, and the orbits with a fixed separable characteristic polynomial $f(x)$ are in bijection with elements of the group
$H^1(k, D) = (L^*/L^{*2})_{N\equiv1}$.
This elementary abelian $2$-group has order $2^m$, where $m+1$ is the number of irreducible factors of $f(x)$ in $k[x]$. This is also equal to the order of the stabilizer
$H^0(k,D) = D(k) = L^*[2]_{N=1}$
of any point in the orbit over $k$.
Hence the number of vectors $T$ in $V$ with any fixed separable polynomial is equal to the order of the finite group $\SO(W)(q)$, and is given by the formula
$$\#\SO(W)(q) = q^{n^2}(q^{2n} -1)(q^{2n-2} - 1)\cdots(q^2 - 1).$$

Since $k$ is perfect, a similar argument applies to all {\it regular} orbits---those orbits for which the characteristic polynomial of $T$ is equal to its minimal polynomial. Any regular vector $T$ with characteristic and minimal polynomial $f(x)$ has stabilizer $D$, the finite \'etale group scheme which is the kernel of the map $\Res_{L/k}\mu_2\rightarrow \mu_2$ with $L = k[x]/ (f(x))$. This finite group scheme has order $2^r$ over $k$, where $r+1$ is the number of distinct roots of $f(x)$ in the separable closure $k^s$. (Here $0 \leq r \leq 2n$, as we do not assume that the minimal polynomial $f(x)$ is separable.) Since the number of regular orbits with this characteristic polynomial is the order of the finite group $H^1(k,D)$, and this order is equal to the order of the stabilizer $H^0(k,D)$ of any point in the orbit, the number of regular vectors $T$ with any fixed characteristic polynomial $f(x)$ is equal to the order of the finite group $\SO(W)(q)$.  Since there are $q^{2n}$ choices for the coefficients of the characteristic polynomial $f(x)$, we can count the total number of regular vectors in $V(q)$. This number is
$$ q^{2n}\cdot \#\SO(W)(q) = q^d - q^{d-2} + \cdots$$
where $d = \dim V$. Since equality of the characteristic and minimal polynomials is an algebraic condition which is independent of field extension, this count shows that irregularity is a codimension $2$ condition.

We return to the case where $f(x)$ is separable. By Lang's theorem, we have $H^1(k,J) = 0$, where $J$ is the (connected) Jacobian of the smooth hyperelliptic curve $C$ with equation $y^2 = f(x)$ over $k$.
Hence the homomorphism $J(k)/2J(k) \rightarrow H^1(k,D)$ is an isomorphism, and every orbit with characteristic polynomial $f(x)$ comes from a $k$-rational point on the Jacobian.

\subsection{Orbits over a non-archimedean local field}

Next, we consider the case when $k$ is a non-archimedean local field, with ring of integers $A$ and finite residue field $A/\pi A$. In this case, the spin cover gives a connecting homomorphism in Galois cohomology, which is an isomorphism by Kneser's theorem~\cite{Kne}, cf.~\cite[Thms.\ 6.4,~6.10, pp.\ 356--397]{PRAG}: 
$$H^1(k,\SO(W)) \cong H^2(k, \mu_2) \cong \Z/2 \Z.$$
One can show that the resulting map $\eta: H^1(k,D) \rightarrow H^1(k, \SO(W)) \cong \Z/2\Z$ is an even quadratic form whose associated bilinear form is the cup product on $H^1(k,J[2])$ induced from the Weil pairing $J[2] \times J[2] \rightarrow \mu_2$ (see \cite {W}, \cite{ PR}). The subgroup  $J(k)/2J(k)$ is a maximal isotropic subspace on which $\eta = 0$. This allows us to count the number of orbits with a fixed separable characteristic polynomial $f(x)$. 

We first do this count when the characteristic of the residue field is odd.
Let $m+1$ be the number of irreducible factors of $f(x)$ in $k[x]$, and let $A_L$ be the integral closure of the ring $A$ in $L$. Then the stabilizer $H^0(k,D) = D(k)$ has order $2^m$. This is the order of $H^2(k,D)$ by Tate's local duality theorem, and his Euler characteristic formula then shows that the group $H^1(k, D)$ has order $2^{2m}$ (see~\cite[\S II.5]{S}). The number of orbits with characteristic polynomial $f(x)$ is equal to the number of zeroes of the even quadratic form $\eta$ on this vector space, which is $2^{m-1}(2^m + 1) = 2^{2m-1} + 2^{m-1}$. The subgroup $J(k)/2J(k)$ has order $2^m$; it is a maximal isotropic subspace for this quadratic form. We note that $2^m$ is also the order of the subgroup $(A_L^*/A_L^{*2})_{N\equiv1}$ of units in $(L^*/L^{*2})_{N \equiv 1}$. We will see that these two subgroups of $(L^*/L^{*2})_{N \equiv 1}$ coincide in many cases, such as when the polynomial $f(x)$ has coefficients in $A$ and the order $A[x] / (f(x))$ is equal to the integral closure $A_L$ of $A$ in $L$.

Next assume that the characteristic of the residue field is even, and for simplicity that $k$ is an unramified extension of $\Q_2$.
Let $m+1$ be the number of irreducible factors of $f(x)$ in $k[x]$, and let $A_L$ be the integral closure of the ring $A$ in $L$. Then the stabilizer $H^0(k,D) = D(k)$ has order $2^m$. This is the order of $H^2(k,D)$ by Tate's local duality theorem, and his Euler characteristic formula now shows that the group $H^1(k, D)$ has order $2^{2m + 2n}$. The number of orbits with characteristic polynomial $f(x)$ is again equal to the number of zeroes of the even quadratic form $\eta$ on this vector space. The subgroup $J(k)/2J(k)$ has order $2^{m+n}$; it is a maximal isotropic subspace for the quadratic form. Here the order of the subgroup $(A_L^*/A_L^{*2})_{N\equiv1}$ of units in $(L^*/L^{*2})_{N \equiv 1}$ is equal to $2^{m + 2n}$, which is strictly larger than the order of $J(k)/2J(k)$.

\subsection{Orbits over $\C$ and $\R$}

The local field $k = \C$ is algebraically closed, so there is a unique (distinguished) orbit for each separable characteristic polynomial by Proposition~\ref{sepclosed}.

The classification of orbits is more interesting for the local field $k = \R$. We now work this out in some detail, as we will need an explicit description of the real orbits in order to use arguments from the geometry of numbers later in this paper. Over the real numbers the pointed set $H^1(\R,\SO(W))$ has $n+1$ elements, corresponding to the quadratic spaces of signature $(p,q)$ with $p+q = 2n+1$ and $q \equiv n$ (mod 2). The split space $W$ of dimension $2n+1$ and discriminant $\equiv 1$ has signature $(n+1,n)$ over $\R$, so $\SO(W) = \SO(n+1,n)$.

Let $2m+1$ be the number of real roots of $f(x)$, with $ 0 \leq m \leq n$. Then the polynomial $f(x)$ has exactly $(2m+1) + (n-m) = m + n + 1$ irreducible factors over $\R$, and $L = \R^{2m+1} \times \C^{n-m}$. The stabilizer of any self-adjoint operator $T$ with this characteristic polynomial is an elementary abelian $2$-group of order $2^{n+m}$, isomorphic to $H^0(\R,D) = L^*[2]_{N=1}$. The group $H^1(\R,D) = ((\R^*/\R^{*2})^{2m+1} \times (\C^*/\C^{*2})^{n-m})_{N \equiv 1}$ is an elementary abelian $2$-group of order $2^{2m}$. The real locus of the hyperelliptic curve $y^2 = f(x)$ has $m+1$ components, and the group $J(\R)/2J(\R) = J(\R)/J(\R)^0$ has order $2^m$. This is a maximal isotropic subspace of $H^1(\R, J[2])$ for the cup product pairing and each class in $ J(\R)/2J(\R)$
gives a soluble real orbit for the split group.

There are $\binom{2m+1}{m}$ orbits of $\SO(n+1,n)$ on the self-adjoint operators $T$ with characteristic polynomial $f(x)$. One can identify these orbits as follows. Such an operator $T$ has $2m+1$ distinct real eigenspaces of dimension $1$ and $(n-m)$ stable subspaces of dimension $2$, which correspond to the pairs of complex conjugate eigenvalues. The orthogonal space $W(\R)$ decomposes as an orthogonal direct sum of these eigenspaces, and the 2-dimensional eigenspaces each have signature $(1,1)$. It follows that the orthogonal sum of the 1-dimensional eigenspaces has signature $(m+1,m)$, and the subset of $m$ eigenspaces which are negative definite determines the real orbit of $T$. 

Let 
$$\lambda_0 < \lambda_1 < \cdots < \lambda_{2m}$$
be the real roots of $f(x)$. The operators $S$ in the distinguished orbit have eigenvectors $S(w_i) = \lambda_i.w_i$ with $\langle w_i, w_i \rangle = (-1)^i$. In other words, the signs of the inner products of the eigenvectors of an operator in the distinguished orbit are
\begin{equation}\label{distsigns}
 + ~ - ~ + ~ - \;\cdots\; + \,.
 \end{equation}
Using the explicit description of the map $\delta : J(\R)/2J(\R) \rightarrow H^1(\R, J[2])$, we see that for a point $P = (a,b)$ on the curve, with $\lambda_{2k} < a < \lambda_{2k+1}$, the class of the divisor $D = (P) - (O)$ gives an orbit with eigenvectors $w_i$ satisfying $\langle w_i, w_i \rangle = (-1)^i$ for $i \leq 2k$ and $\langle w_i, w_i \rangle = (-1)^{i+1}$
for $i \geq 2k+1$. It follows that the signs of the inner products of eigenvectors of an operator in a soluble real orbit have the form
\begin{equation}\label{solublesigns}
+ ~~ \epsilon_1~ ~\epsilon_2~~ \cdots~~ \epsilon_m
\end{equation}
where each $\epsilon_i$ is equal to either  $(-~+)$ or $(+~-)$.

The space $\R^{2n} - \{\Delta = 0\}$ of characteristic polynomials with nonzero discriminant has
$n + 1$ components $I(m)$ indexed by integers $0 \leq m \leq n$, where $2m+1$ is the number of real roots. We get a similar decomposition of the space $V (\R) - \{\Delta = 0\}$ of self-adjoint operators $T$ with characteristic polynomials in $I(m)$. We label the various components $V^{(m,\tau)}$ in $V(\R)$ which map to the component $I(m)$ in $\R^{2n}- \{\Delta = 0\}$ by integers $1 \leq \tau \leq \binom{2m+1}{m}$, and let $V^{(m,1)}$ be the set of operators which lie in the distinguished orbits. 

The integer $\tau$ indexes a choice of signs $\langle w_i, w_i \rangle = \pm1$ for the inner products of the $2m+1$ real eigenvectors $w_i$ of $T$, with $m$ signs equal to $-1$. We use $\tau=1$ to designate the distinguished orbit, where we have seen that the signs for the inner products, ordered for increasing real eigenvalues, are given by~(\ref{distsigns}). We use the first $2^m$ indices $\tau=1,\ldots, 2^m$ to denote the soluble orbits, whose signs are described in~(\ref{solublesigns}). Note that the stabilizer of every vector $T$ in the component $V^{(m,\tau)}$ has order $2^{m+n}$, independent of~$\tau$.

\subsection{Orbits over a global field}

Finally, we consider the case when $k$ is a global field. In this case the group $H^1(k,D) = H^1(k,J[2])$ is infinite. We will see why there are also infinitely many classes in the kernel of $\eta$, so infinitely many orbits of self-adjoint operators with characteristic polynomial $f(x)$.  First, we show that every class in the 2-Selmer group lies in the kernel of $\eta$.
\begin{proposition}\label{global}
Every class $\alpha$ in the $2$-Selmer group $S_2(J)\subset H^1(k,J[2])$ lies in the kernel of $\eta$, so corresponds to an orbit of self-adjoint operators over $k$ with characteristic polynomial $f(x)$.
\end{proposition}

\begin{proof}
Let $\alpha$ be a class in the $2$-Selmer group.  By definition, these are the classes in $H^1(k,J[2])$ whose restriction to $H^1(k_v,J[2])$ is in the image of $J(k_v)/2J(k_v)$ for every completion $k_v$. By Proposition~\ref{kernel}, the orthogonal space $U_v$ associated to the restriction of the  class $\alpha$ in $H^1(k_v,\SO(W))$ is split over $k_v$, for every valuation $v$ of $k$. By the theorem of Hasse and Minkowski, a non-degenerate orthogonal space $U$ of dimension $2n+1$ is split over $k$ if and only if $U \otimes k_v$ is split over every completion~$k_v$. Hence the orthogonal space $U$ associated to $\alpha$ is split over $k$, and $\alpha$ lies in the kernel of $\eta$. 
\end{proof}

We call the orbits corresponding to the 2-Selmer group {\it locally soluble}, as the Fano variety $F_T$ associated to $T$ has points over $k_v$ for all completions $v$. For each $f(x)$, these locally soluble orbits are finite in number, as the $2$-Selmer group is finite.

We have already remarked that, for any class $d \in k^*/k^{*2}$, the finite group scheme $J[2]$ is isomorphic to the $2$-torsion subgroup $J_d[2]$ of the Jacobian of the hyperelliptic curve $C_d$ with equation $d y^2 = f(x).$ Using similar methods, we can also obtain $k$-rational orbits with characteristic polynomial $f(x)$ from classes in the $2$-Selmer group of $J_d$ over $k$. Since the $2$-Selmer groups of these quadratic twists of $C$ become arbitrarily large finite elementary abelian $2$-groups (see~\cite{Bol} for the case $n=1$), the number of $k$-rational orbits with characteristic polynomial $f(x)$ is infinite. 

We have proven Theorem~\ref{orbit}, and the remarks following it.

\section{Nilpotent orbits and Vinberg's theory}

The results in this section will not be needed in the proof of the main theorems of this paper; we have included them here to place the orbits in this            
representation, as well as the hyperelliptic curves which appear in their study, into a larger context.                                                      

We say that an orbit $T$ of $\SO(W)$ on $V$ is {\it nilpotent} if the characteristic polynomial of $T$ is equal to $f(x) = x^{2n+1}$. We say that it is {\it regular nilpotent} if the minimal polynomial of $T$ is also equal to $x^{2n+1}$. One can show that the group $\SO(W)(k)$ acts simply-transitively on regular nilpotent elements in $V(k)$, i.e., there is a single orbit with trivial stabilizer. A representative operator $E$ on the standard basis is given by:
$$E: f_1\rightarrow f_2\rightarrow \cdots \rightarrow f_n \rightarrow u \rightarrow e_n \rightarrow \cdots\rightarrow e_2 \rightarrow e_1 \rightarrow 0.$$

We say that a nilpotent orbit $T$ is sub-regular if the minimal polynomial of $T$ is equal to $x^{2n}$. An example of a sub-regular nilpotent operator $E'$ is given on the standard basis by:
$$E': f_1\rightarrow f_2\rightarrow \cdots \rightarrow f_n \rightarrow e_n \rightarrow \cdots \rightarrow e_2 \rightarrow e_1 \rightarrow 0, ~ u \rightarrow 0.$$
The subregular nilpotent orbits of $\SO(W)(k)$ are indexed by the distinct classes $d \in k^*/k^{*2}$ and represented by the operators $d \cdot E'$.

When the characteristic of $k$ is either $0$ or is greater than $2n+1$, the representation $V$ of $\SO(W)$ appears in Vinberg's theory of torsion automorphisms $\theta$, where nilpotent orbits play a special role \cite{P}. We review this connection briefly here, as a model for future work on those automophisms $\theta$ which lift regular elliptic classes in the Weyl group. 

To do this, we change our notation for the rest of this section to conform with Vinberg, and let $G = \SL(W)$ with $W$ of dimension $2n+1$ over $k$. This reductive algebraic group has Lie algebra $\frak g = \{T \in \End(W): \Trace(T) = 0\}$. The outer (pinned) involution $\theta(g) = (g^*)^{-1}$ has fixed subgroup $G^{\theta} = G(0) = \SO(W)$ with Lie algebra $\frak g(0) = \{T \in \End(V): T + T^* = 0\}$. The fixed subgroup $\SO(W)$ acts on the non-trivial eigenspace for $\theta$ on the Lie algebra:
$$\frak g(1) = \{T \in \End(V): T^* = T, ~ \Trace(T) = 0 \} $$
and this is precisely the representation $V$ that we have been studying. The invariant polynomials for
$G(0)$ on $\frak g(1)$ in this case are the restriction of the invariant polynomials $c_2,c_3,\ldots,c_{2n+1}$ for $G$ on the adjoint representation $\frak g$. This gives a polynomial map to the geometric quotient:
$$p: \frak g(1) \rightarrow \frak g(1)/\!\!/ G(0) = \Spec k[c_2,\ldots,c_{2n+1}].$$

The regular nilpotent vector $E$ in $\frak g(1)$ can be completed to a principal $\frak{sl}_2$-triple $[E,F,H]$, with $F$ regular nilpotent in $\frak g(1)$ and $H$ regular semi-simple in $\frak g(0)$. Similarly, the subregular nilpotent vector $E'$ in $\frak g(1)$ can be completed to a subregular $\frak{sl}_2$-triple $[E',F',H']$, with $F'$ subregular nilpotent in $\frak g(1)$ and $H'$ semi-simple in $\frak g(0)$. The centralizer $z(F)$ of $F$ in the Lie algebra $\frak g$ is abelian of dimension $2n$ and is contained in $\frak g(1)$. If we restrict the invariant polynomials to the affine subspace $E + z(F)$ of $\frak g(1)$, the resulting map
$$p: E + z(F) \rightarrow \Spec k[c_2,\ldots, c_{2n+1}]$$
is a bijection, whose inverse is called the Kostant section \cite{P}. The isotropic subspace $X$ of $W$ spanned by $\{f_1,f_2,\ldots,f_n\}$ satisfies the condition: $T(X)$ is contained in $X^{\perp}$ for all operators $T$ in the Kostant section. It also contains no $T$-invariant subspace $Y$ except for $Y = 0$. Hence the Kostant section gives a representative for each distinguished orbit, once a representative $E$ of the regular nilpotent orbit has been chosen.

In his Harvard PhD thesis, J.\ Thorne \cite{Th} studies the centralizer $z(F')$ of a subregular nilpotent vector $F'$ in $\frak g$. This has dimension $2n+2$, and the eigenspace $z(F')(1)$ has dimension $2n+1$ in $\frak g(1)$. The restriction of the invariant polynomials to the affine subspace $E' + z(F')(1)$ of $\frak g(1)$ defines a family
$$p : E' + z(F')(1) \rightarrow \Spec k[c_2,\ldots,c_{2n+1}]$$
of affine curves over the geometric quotient, which can be identified with the affine hyperelliptic curves $y^2 = f(x)$. They are smooth over the open subset where $\Delta \neq 0$. If one starts instead with the subregular vector $d\cdot E'$ and forms the corresponding $\frak{sl}_2$-triple, one gets the family of affine curves $d y^2 = f(x)$, with the same $2$-torsion in their Picard groups. 

\section{Integral orbits}

Fix a separable polynomial $f(x) = x^{2n+1} + c_2x^{2n-1} + c_3x^{2n-2} + \cdots + c_{2n+1}$ over $\Q$, let $C$ be the hyperelliptic curve with equation $y^2 = f(x)$ and let $J$ be its Jacobian. We have seen that classes in the $2$-Selmer group of $J$  give locally soluble orbits for the split orthogonal group $\SO(W)(\Q)$ on~$V(\Q)$. However, since there are an infinite number of rational orbits and the Selmer group is finite, not all of the orbits over $\Q$ are locally soluble. In this section, we consider those orbits with an integral representative, in a sense to be defined below. We will show that this is a finite set that contains the locally soluble orbits.

\subsection{Orbits over $\Z_p$ for $p\neq 2$}

We begin by discussing the integral orbits over $\Z_p$, with $p \neq 2$. In this case the group $G = \SO(W)$ is smooth over $\Z_p$, and the lattice $W \otimes \Z_p$ has rank $2n+1$ and discriminant
$\disc(W) = (-1)^n\det(W) \equiv 1$ in $\Z_p/\Z_p^{*2}$.

\begin{lemma}
Assume that $p \neq 2$ and let $I$ be a $\Z_p$-module of rank $2n+1$ equipped with a symmetric bilinear form $I \times I \rightarrow \Z_p$ whose discriminant is the square of a unit. Then $I$ is isometric to $W$ over~$\Z_p$.
\end{lemma}

This follows from the fact that there are, up to isomorphism, only two orthogonal spaces over $\Q_p$ of dimension $2n+1$ and discriminant $\equiv 1$. The non-split orthogonal space does not contain a non-degenerate lattice $I$ over $\Z_p$, and the corresponding lattices in the split orthogonal space are all conjugate. The stabilizer of such a lattice is a hyperspecial maximal compact subgroup of $\SO(W)(\Q_p)$.

Using this lemma, we can give a simple classification of the orbits of $\SO(W)(\Z_p)$ on the self-adjoint operators $T: W \otimes \Z_p \rightarrow W \otimes \Z_p$ with a fixed characteristic polynomial $f(x)$ over $\Z_p$, which is assumed to be separable over $\Q_p$.  Let $R = \Z_p[x] / (f(x))$, which is a $\Z_p$-order in the \'etale $\Q_p$-algebra $L = \Q_p[x]/ (f(x))$. We say that a $\Z_p$-lattice $I$ in $L$ that spans $L$ over $\Q_p$ is a {\it fractional ideal for $R$} if it is stable under multiplication by $R$. Such an ideal $I$ has a norm $N(I)$, which is an element of $\Q_p^*/\Z_p^*$: if $p^aI$ is contained in $R$ and the order of the finite $\Z_p$-module $R/p^aI$ is $p^b$, then $N(I) \equiv p^{b-(2n+1)a}$. 
If $\alpha$ is an element of $L^*$, then we have the principal fractional ideal $(\alpha)$ consisting of all $R$-multiples of~$\alpha$. This has norm $\equiv N(\alpha)$. 

We will define an equivalence relation on pairs $(I, \alpha)$, where $I$ is a fractional ideal for $R$, $\alpha$ is an element of $L^*$ whose norm is a square in $\Q_p^*$, the square $I^2$ of the ideal $I$ is contained in the principal ideal $(\alpha)$, and the square of its norm satisfies $N(I)^2 \equiv N(\alpha)$  (mod $\Z_p^{*2})$.  We say that the pair $(I',\alpha')$ is {\it equivalent} to $(I,\alpha)$ if there is an element $c$ in $L^*$ with $I' = cI$ and $\alpha' = c^2\alpha$. 

If $R$ is maximal in $L$, every fractional ideal $I$ is principal and
every equivalence class is represented by a pair $(R,\alpha)$, where
$\alpha$ is a unit having square norm to $\Z_p^*$. In this case, the
set of equivalence classes form a finite group $(R^*/R^{*2})_{N\equiv
  1}$. In general, the equivalence classes form a finite set which
contains the finite group $(R^*/R^{*2})_{N\equiv 1}$ corresponding to
those pairs where the ideal $I$ is principal. Since the ring $R$ is
generated by a single element $\beta$ over $\Z_p$, it is Gorenstein of
dimension one. Hence the ideal $I$ is principal if and only if it is
{\it proper}, i.e., $\End_R(I) = R$ (see~\cite {Bass}).

The set of equivalence classes always maps to the group $(L^*/L^{*2})_{N \equiv 1}$, by taking the class of $(I,\alpha)$ to the class of $\alpha$. This map need not be injective---for example, when the ring $R$ is not maximal, a unit~$\alpha$ that is not a square in $R^*$ may become a square in $L^*$.

\begin{proposition}\label{podd}
Assume that $p \neq 2$ and that $f(x) = x^{2n+1} + c_2x^{2n-1} + c_3x^{2n-2} + \cdots + c_{2n+1}$ is a polynomial with coefficients in $\Z_p$ and nonzero discriminant.
Then the integral orbits of $\SO(W)(\Z_p)$ on self-adjoint operators $T$ in $V(\Z_p)$ with characteristic polynomial $f(x)$ correspond to the equivalence classes of pairs $(I, \alpha)$ for the order $R = \Z_p[x]/(f(x))$. 
The stabilizer  in $\SO(W)(\Z_p)$ of $\,T$ is isomorphic to $S^*[2]_{N=1}$, where $S = \End_R(I)$.
The integral orbit corresponding to the pair $(I,\alpha)$ maps to the rational orbit of $\,\SO(W)(\Q_p)$ on $V(\Q_p)$ corresponding to the class of $\alpha$ in $(L^*/L^{*2})_{N \equiv 1}$.
\end{proposition}

\begin{proof}
Starting with a pair $(I,\alpha)$ we construct an integral orbit as follows. Recall that  $L = \Q_p + \Q_p\beta + \cdots + \Q_p\beta^{2n}$. Define the symmetric bilinear pairing $I \times I \rightarrow \Z_p$ by $(\lambda, \nu)_{\alpha^{-1}} := $ the coefficient of $\beta^{2n}$ in the product $\alpha^{-1}\lambda\nu$. Since the square $I^2$ of the ideal $I$ is contained in the principal ideal $(\alpha)$, this product lies in the subring $R = \Z_p + \Z_p\beta + \cdots + \Z_p\beta^{2n}$ and the coefficient of $\beta^{2n}$ is integral. Since $N(I)^2 \equiv N(\alpha)$, this pairing has discriminant $\equiv 1$. By the above lemma, the bilinear module $I$ is isometric to $W$ over $\Z_p$. Since $I$ is an ideal of $R$, multiplication by $\beta$ gives a self-adjoint operator on $I$ with characteristic polynomial $f(x)$. Translating this operator to $W$ by the isometry gives the desired integral orbit. Since a self-adjoint $T: W \rightarrow W$ gives $W$ the structure of a torsion-free $\Z_p[T] = R$-module of rank one, every integral orbit arises in this manner. 
The centralizer of $T$ in $\End(W)$ is equal to $S = \End_R(I)$, and the elements in this centralizer that also lie in $\SO(W)$ are the elements that have order $2$ (since they are both self-adjoint and orthogonal) and norm 1 (since this is the condition that the determinant $=1$).
Finally, since the bilinear form defined above is equivalent to the form $(\lambda, \nu)_{\alpha}$ over $\Q_p$, it maps to the rational orbit of $\alpha$. This completes the proof of Proposition~\ref{podd}. 
\end{proof}

%

As a corollary, we obtain the following.

\begin{corollary}\label{unitsubgroup}
Assume that $f(x)$ has coefficients in $\Z_p$ and that the ring $R = \Z_p[x] / (f(x))$ is maximal in $L$. Then the integral orbits correspond to classes $\alpha$ in the unit subgroup $(R^*/R^{*2})_{N \equiv 1}$ of $(L^*/L^{*2})_{N \equiv 1}$. Every rational orbit whose class lies in the unit subgroup has a unique integral representative.
\end{corollary}

An important special case of the corollary is when the discriminant
$\disc(f)$ of the characteristic polynomial is a unit in
$\Z_p^*$. Then the hyperelliptic curve $C$ with affine equation $y^2 =
f(x)$ over $\Z_p$ has good reduction modulo $p$. The Jacobian $J$ of
$C$ also has good reduction, and the finite group scheme $J[2] $ is
\'etale over $\Z_p$. It is obtained by restriction of scalars from the
finite \'etale extension $R$ of $\Z_p$, i.e., $J[2]$ is isomorphic to
the kernel of the map $\Res_{R/\Z_p}(\mu_2)\rightarrow \mu_2$.  Taking
the \'etale cohomology of the exact sequence of group schemes $0
\rightarrow J[2] \rightarrow J \rightarrow J \rightarrow 0$ over
$\Z_p$, we see that the image of $J(\Q_p)/2J(\Q_p)$ in $H^1(\Q_p,J[2])
= (L^*/L^{*2})_{N\equiv1}$ is precisely the unit subgroup
$H^1(\Z_p,J[2]) = (R^*/R^{*2})_{N\equiv1}$. 


More generally, we have the following result, which certainly holds
(cf. \cite[Prop.\ 4.6]{Stoll}) whenever the discriminant of $f(x)$ is a
unit in $\Z_p^*$ or is exactly divisible by $p$.

\begin{corollary}\label{unitmax}
  Assume that $f(x)$ has coefficients in $\Z_p$ and that the order
  $\Z_p[x]/(f(x)) = R$ is maximal in the \'etale algebra $\Q_p[x]/(f(x)) =
  L$. Then the integral orbits with characteristic polynomial $f(x)$
  are in bijection with the soluble orbits over $\Q_p$, i.e., those
  whose classes lie in the image of the subgroup $J(\Q_p)/2J(\Q_p)$.
Moreover, the stabilizer in $\SO(W)(\Z_p)$ of an element in such an orbit is isomorphic to $J[2](\Q_p)$.                          
\end{corollary}

\begin{proof}
We know from Corollary~\ref{unitsubgroup} that the integral orbits
correspond to classes in the unit subgroup $(R^*/R^{*2})_{N \equiv 1}$
of $(L^*/L^{*2})_{N \equiv 1}$. Hence it suffices to show that the image
of $J(\Q_p)/2J(\Q_p)$ lies in this subgroup (as the two subgroups have
the same order). Write $L = \prod L_i$ as a product of fields, and 
let~$\beta_i$ denote the component of $\beta$ in $L_i$.  Then $R = \prod
R_i$, where $R_i = \Z_p[\beta_i]$ is the maximal order in $L_i$. We
will show that for any point $P = (x(P),y(P))$ on the curve $C$ over
$\Q_p$, the valuation of $x(P) - \beta_i$ in~$L_i$ is even for all
$i$, which will imply the first assertion of the corollary.

If the valuation of $x(P)$ is negative, it must be even. Indeed,
$y(P)^2 = x(P)^{2n+1} + c_2x(P)^{2n-1}+\cdots+c_{2n+1} = f(x(P))$, and
$f(x)$ has integral coefficients $c_k$. Since each $\beta_i$ is 
integral, the valuation of $x(P) - \beta_i$ is also negative and even,
equal to the valuation of $x(P)$. Next assume that the valuation of
$x(P)$ is $\geq 0$. If $L_i$ is not totally ramified, $\beta_i$ must
be a unit whose image in the residue field of $R_i$ is not in the
prime field $\F_p$. Hence the valuation of $x(P) - \beta_i$ is equal
to zero. If $L_i$ is totally ramified and the valuation of $x(P) -
\beta_i$ is strictly positive, then the valuation of $x(P) - \beta_j$
in $L_j$ must be zero for all $j \neq i$. Indeed, the images of $\beta_i$
and $\beta_j$ in their respective residue fields, when embedded in a
common algebraic closure of $\F_p$, cannot be equal or even
conjugate over $\F_p$.  Otherwise, the discriminant of the polynomial
$f(x)$ would be larger than the product of the discriminants of the
maximal orders $R_i$, which would contradict the maximality of the
ring $R$.
As $L_i$ is totally ramified, the valuation of $x(P) - \beta_i$
in~$L_i$ 
is equal to the valuation of its norm in $\Q_p$. Since the
valuation of $x(P) - \beta_j$ in $L_j$ is zero for all $j\neq i$,
the valuation of $x(P) - \beta_i$ in $L_i$ is equal to the valuation
of $N(x(P) - \beta) = y(P)^2$ in $\Q_p$, which is clearly even. 
Finally, to obtain the second assertion of the corollary we simply observe that, because~$R$ is maximal in $L$, we have $S^\ast[2]_{N=1}\cong R^\ast[2]_{N=1} \cong L^\ast[2]_{N=1} \cong J[2](\Q_p)$. 
This completes the proof.
\end{proof}

%

In general, given a $p$-adic rational orbit with integral
characteristic polynomial $f(x)$, we can ask if it is represented by
an integral self-adjoint operator $T: W \otimes \Z_p \rightarrow W
\otimes \Z_p$. That is, given an element $\alpha$ in $L^*$ with square
norm to $k^*$, which represents a class in the quotient group
$(L^*/L^{*2})_{N\equiv1}$ lying in the kernel of the map $\eta$, we
can ask if there is an ideal $I$ of $R$ with $I^2 \subset \alpha R$
and $(NI)^2 \equiv N(\alpha)$.  One situation in which we can prove
that an integral orbit exists is when the corresponding $p$-adic
rational orbit is soluble, i.e., for those elements $\alpha$ whose
classes are contained in the subgroup $J(\Q_p)/2J(\Q_p)$ of
$H^1(\Q_p,J[2]) = (L^*/L^{*2})_{N\equiv1}$.

We thank Michael Stoll for his help with the proof of the following key proposition:

\begin{proposition}\label{integralorbit}
Assume that $f(x)$ has coefficients in $\Z_p$ and nonzero discriminant. Let $R = \Z_p[x] / (f(x))$ and $L = \Q_p[x]/(f(x))$. If $\alpha\in L^*$ represents a class in the subgroup $J(\Q_p)/2J(\Q_p)$ of $(L^*/L^{*2})_{N\equiv1}$, then there is an ideal $I$ of $R$ with $I^2 \subset \alpha R$ and $(NI)^2 \equiv N(\alpha)$. Hence there is an integral orbit representing the rational orbit of $\alpha$.
\end{proposition}

\begin{proof}
We prove this by an explicit construction of the ideal $I$. 
We can represent any class in $J(\Q_p)/2J(\Q_p)$ by a Galois-invariant divisor of degree zero 
$$D  = (P_1) + (P_2) + \cdots + (P_m) - m(O)$$
with $m \leq n$, where the $P_i = (a_i,b_i)$ are defined over some finite Galois extension $K$ of $\Q_p$.  We may assume that the coordinates of all the $P_i$ lie in the integers of $K$ with $b_i \neq 0$, and that the $a_i$ are all distinct. Define the monic polynomial
$$P(x) = (x - a_1)(x - a_2)\cdots(x - a_m)$$
of degree $m$ over $\Q_p$. By our explicit description of the coboundary map from $J(\Q_p)/2J(\Q_p)$ to $H^1(\Q_p, J[2])$, we have $(-1)^mP(\beta) \equiv \alpha$ in $(L^*/L^{*2})_{N\equiv1}$. Next define $R(x)$ as the unique polynomial of degree $\leq m-1$ over $\Q_p$ with the property that
$$R(a_i) = b_i, ~1 \leq i \leq m.$$
Then $R(x)^2  - f(x) = h(x)P(x)$ in $\Q_p[x]$. This gives the Mumford representation of the divisor $D$, which is cut out on the curve by the two equations $P(x) = 0$ and $y = R(x)$.

The polynomial $P(x)$ clearly has coefficients in the ring $\Z_p$. When $R(x)$ also has coefficients in $\Z_p$, we define the ideal $I_D$ of $R = \Z_p[\beta]$ as the $R$-submodule of $L$  generated by the two elements $P(\beta)$ and $R(\beta)$. For example, when $m = 1$, the divisor $D$ has the form $(P) - (O)$ with $P = (a,b)$, and the ideal $I_D$ is generated by the two elements $(\beta - a)$ and $b$. 

We claim that $I_D^2$ is contained in the ideal generated by $\alpha = (-1)^mP(\beta)$ and that $N(I_D) \equiv b_1b_2\cdots b_m.$ Hence $N(I_D)^2$ generates the same ideal as $N(P(\beta)) = (-1)^mb_1^2b_2^2\cdots b_m^2$ and we have constructed an integral orbit mapping to the rational orbit of $\alpha$.

The fact that the ideal $I_D^2 = (P(\beta)^2, P(\beta)R(\beta), R(\beta)^2)$ is contained in the ideal generated by $P(\beta)$ follows from $f(\beta) = 0$, so $R(\beta)^2 = h(\beta)P(\beta)$. On the other hand, 
$$R/I_D = \Z_p[\beta]/I_D = \Z_p[x]/(f(x), P(x), R(x)) = \Z_p[x]/(P(x),R(x))$$
and so $N(I_D)\equiv \#\Z_p[x]/(P(x),R(x))$.  The latter cardinality may be computed as the determinant (viewed as a power of $p$) of the $\Z_p$-linear transformation $\times R(x):\Z_p[x]/(P(x))\to\Z_p[x]/(P(x))$.  This is equivalent to computing the determinant of the $K$-linear transformation $\times R(x):K[x]/(P(x))\to K[x]/(P(x))$.  But $P(x)$ splits over $K$, and we have
\begin{equation}\label{crt}
K[x]/(P(x)) \cong K[x]/(x-a_1)\times \cdots \times K[x]/(x-a_m)\cong K^m
\end{equation}
since the $a_i$ are distinct.  The image of $R(x)$ in $K^m$ above is given by $(R(a_1),\ldots,R(a_m))=(b_1,\ldots,b_m)$, and thus the determinant of $\times R(x)$ as a transformation of $K[x]/(P(x))\cong K^m$ is simply $b_1\cdots b_m$, as claimed.
This completes the construction of an integral orbit when the
polynomial $R(x)$ has coefficients in $\Z_p$.

If $R(x)$ is not integral, we consider the $2n+1$ roots of the monic polynomial $F(x) - R(x)^2$.
These contain the roots $a_1,a_2,\ldots,a_m$ of $P(x)$, and a consideration of the Newton polygon shows that  at most $m-2$ of the remaining roots $r_1,r_2,\dots,r_{m-2}$ have non-negative valuation. Let $s_i = R(r_i)$. Then the function $y - R(x)$ on the curve has principal divisor
$$\{(a_1,b_1) + ..+ (a_m,b_m) - m(O)\} + \{(r_1, s_1) +\cdots+ (r_{m-2},s_{m-2}) - (m-2)(O)\} + E$$
where the $x$-coordinates of the points in the divisor $E$ are the remaining roots. Since these all have negative valuation, the class of $E$ lies in $2J(\Q_p)$. Hence the class of the divisor $D$ is equivalent to the class of the divisor
$$D^* = (r_1,s_1) + \cdots+ (r_{m-2},s_{m-2}) - (m-2)(O)$$
in the quotient group $J(\Q_p)/2J(\Q_p)$. The proof now concludes by a descent on the number $m$ of points in $D$. Once $m \leq 1$ the polynomials  $P(x)$ and $R(x)$ are both integral. 
\end{proof}

We end with some brief remarks on integral orbits $(I,\alpha)$ in the case when $\End_R(I) = R'$ is strictly larger than $R$, so $I$ is not a free $R$-module.
When the order $R'$ is Gorenstein, $I$ is a free $R'$-module, so can be identified with $R'$.  Moreover, the dualizing module $\Hom(R',\Z_p)$ is free, and can be identified with $(1/\gamma) R'$ under the trace pairing, with $\gamma$ in $L^*$ chosen so that $N(\gamma) \equiv N(f'(\beta))$ in $\Q_p^*/\Q_p^{*2}$. The element $\gamma$ is well-defined up to multiplication by $R'^{*2}$, and defines an integral orbit via taking the equivalence class of the pair $(I,\alpha) = (R',\gamma/f'(\beta))$. So the integral orbits with characteristic polynomial $f(x)$ and endomorphisms by a Gorenstein order $R'$ containing $R$ in $L$ form a principal homogeneous space for the group $(R'^*/R'^{*2})_{N\equiv1}$. They map to the rational orbits with classes $\gamma/f'(\beta)$ in $(L^*/L^{*2})_{N\equiv1}$, where $(1/\gamma)$ generates
the dualizing module of $R'$ in $L$, and these rational classes need not lie in the unit subgroup of
$(L^*/L^{*2})_{N\equiv1}$.

\subsection{General integral orbits}

The lattice $W \otimes \Z_2$ has rank $2n+1$ and discriminant
$\disc(W) = (-1)^n\det(W) \equiv 1$ in $\Z_2/\Z_2^{*2}$. In this case we need a bit more to identify those lattices isometric to $W$ over $\Z_2$, as both orthogonal spaces over $\Q_2$ of dimension
$2n+1$ and discriminant $\equiv 1$ contain non-degenerate lattices with these invariants.

\begin{lemma}
Let $I$ be a $\Z_2$-module of rank $2n+1$ equipped with a symmetric bilinear form $I \times I \rightarrow \Z_2$ whose discriminant is the square of a unit. Then $I$ is isometric to $W$ over $\Z_2$ if and only if it contains an isotropic sublattice of rank $n$.
\end{lemma}

Indeed, the existence of such an isotropic sublattice is equivalent to the assumption that $I \otimes \Q_2$ is a split orthogonal space, and in this space all non-degenerate odd lattices are conjugate. The stabilizer of such a lattice is the normalizer of a parahoric subgroup of $\SO(W)(\Q_2)$.  

If the bilinear lattice $I$ contains a maximal isotropic sublattice over $\Z_2$, then we say that it is {\it split}~\cite{MH}. The same argument as the proof of Proposition~\ref{podd} yields the following.

\begin{proposition}\label{peven}
Assume that $f(x) = x^{2n+1} + c_2x^{2n-1} + c_3x^{2n-2} + \cdots + c_{2n+1}$ is a polynomial with coefficients in $\Z_2$ and nonzero discriminant in $\Q_2$.
Then the integral orbits of $G(\Z_2)$ on self-adjoint operators $T$ with characteristic polynomial $f(x)$ correspond to the equivalence classes of pairs $(I, \alpha)$ for the order $R = \Z_2[x]/(f(x))$, with the property that the bilinear form $(\; ,\, )_{\alpha^{-1}}$ on $I$ is split. The integral orbit corresponding to the pair $(I,\alpha)$ maps to the rational orbit of $\SO(W)(\Q_2)$ corresponding to the class of $\alpha$ in $(L^*/L^{*2})_{N \equiv 1}$.
\end{proposition}

The condition that the pair $(I,\alpha)$ gives a split orthogonal space is not empty. Consider the case when $f(x)$ has integral coefficients and the discriminant of $f(x)$ is a unit in $\Z_2^*$. Then the ring $R = A_L$ is the maximal order in $L$, and every pair $(I,\alpha)$ is equivalent to $(R,u)$ with $u$ in the unit subgroup $(A_L^*/A_L^{*2})_{N \equiv 1}$ of $(L^*/L^{*2})_{N\equiv 1}$ . This subgroup of has order $2^{m+2n}$, where $f(x)$ has $m+1$ irreducible factors over $\Q_2$. Not all of these orthogonal spaces can be split over $\Q_2$, as a maximal isotropic subgroup for the quadratic form $\eta$ on $(L^*/L^{*2})_{N\equiv 1}$  has order $2^{m+n}$. On the other hand, the image of classes in $J(\Q_2)/2J(\Q_2)$ give a subgroup of the units of order $2^{m+n}$, which do correspond to orbits.

We suspect that this remains true in general. More precisely, we conjecture that whenever $f(x)$ has coefficients in $\Z_2$ and $\alpha$ is a class in $J(\Q_2)/2J(\Q_2)$, the corresponding orbit over $\Q_2$ is represented by (at least one) integral orbit over $\Z_2$. One can try to imitate the construction of an ideal that we gave in the case where $p$ is odd. The Mumford representative works just as well, but the reduction to a divisor of smaller degree used the fact that when $P = (a,b)$ on $C$ has
${\rm ord}(a) < 0$, the class $(P) - (O)$ is divisible by $2$ in the Jacobian. For $p = 2$ this is true if we assume that the integral coefficients satisfy: $2^{4k}$ divides $c_k$, for $k = 2,3, \ldots, {2n+1}$.
So for integral polynomials of this form, each class in $J(\Q_2)/2J(\Q_2)$ is represented by at least one integral orbit.

Similarly, we have the following result on integral orbits over $\Z$.

\begin{proposition}
Assume that $f(x) = x^{2n+1} + c_2x^{2n-1} + c_3x^{2n-2} + \cdots + c_{2n+1}$ is a polynomial with coefficients in $\Z$ and nonzero discriminant in $\Q$.
Then the integral orbits of $G(\Z)$ on self-adjoint operators $T$ with characteristic polynomial $f(x)$ correspond to the equivalence classes of pairs $(I, \alpha)$ for the order $R = \Z[x]/(f(x))$, with the property that the bilinear form $(\; ,\, )_{\alpha^{-1}}$ on $I$ is split. The integral orbit corresponding to the pair $(I,\alpha)$ maps to the rational orbit of $\SO(W)(\Q)$ corresponding to the class of $\alpha$ in $(L^*/L^{*2})_{N \equiv 1}$.
\end{proposition}

Here one can check if the bilinear form is split simply by a calculation of the signature of the orthogonal space over $\R$ (see~\cite[Chapter V]{S3}).

In general, given a rational orbit with integral characteristic polynomial $f(x)$, we can ask if it is represented by an integral self-adjoint operator $T: W \rightarrow W$. That is, given an element $\alpha$ in $L^*$ that represents a class in the group $(L^*/L^{*2})_{N \equiv 1}$ lying in the kernel of $\eta$, we can ask if there is an ideal $I$ of $R = \Z[x]/(f(x))$ with  $I^2 \subset \alpha R$
and $(NI)^2 = N(\alpha)$. In this case, there is no need to check that the bilinear form defined by $\alpha$ on $I$ is split, as it is split over $\Q$. The number of integral orbits mapping to a fixed rational orbit is finite, as there are only finitely many ideals with a given norm. 

Fix a representative of the class $\alpha$ in the group $L^*$, and let $m(\alpha)$ be the number of ideals $I$ of $R$ with the property that $I^2 \subset \alpha R$
and $(NI)^2 = N(\alpha)$. Let $m_p(\alpha)$ be the number of ideals $I_p$ of $R_p:=R\otimes\Z_p$ with this property. For all  primes $p$ that do not divide $\Delta$ and where the element $\alpha$ is a unit, we must have $I_p = R_p$. Hence $m_p(\alpha_p) = 1$ for almost all $p$, where $\alpha_p$ denotes the image of $\alpha$ in $R_p$. Since an ideal of $R$ is determined by its localizations, and since the intersection of $L$ with the product $\prod_p I_p$ is an ideal $I$ of $R$ with the desired properties, we have the product formula
$$m(\alpha) = \prod_p m_p(\alpha).$$
The pair $(I,\alpha)$ gives the same orbit as the pair $(cI,\alpha)$ whenever $c$ is an element of $L^*$ with $c^2 = 1$ and $N(c) = 1$. On the other hand, the ideals $I$ and $cI$ are equal only when $c$ is a unit in the endomorphism ring $R(I)$ of the ideal $I$. So the number of distinct ideals giving the same integral orbit is the order of the finite quotient group $L^*[2]_{N=1}/R(I)^*[2]_{N=1}$. Note that the stabilizer of an operator in this integral orbit is the finite group $R(I)^*[2]_{N=1}$ and the stabilizer of an operator in the rational orbit $\alpha$ is the finite group $L^*[2]_{N=1}$. Hence we also have the formula
$$m(\alpha) = \sum_{\sigma\in O(\alpha)} \#\Stab_{G(\Q)}(\alpha)/\#\Stab_{G(\Z)}(\sigma)$$
where the sum is taken over the set $O(\alpha)$ of integral  orbits $\sigma$ mapping to the rational orbit 
corresponding to $\alpha$.

The same argument comparing ideals and integral orbits works locally, so combining the above formulas gives the following

\begin{proposition}\label{product}
Assume that the number of integral orbits representing a fixed rational orbit $\alpha$ is nonzero. Then
$$\#\Stab_{G(\Q)}(\alpha) \sum_{\sigma\in O(\alpha)} 1/\#\Stab_{G(\Z)}(\sigma) = 
\prod_p \#\Stab_{G(\Q_p)}(\alpha_p) \sum_{\sigma_p\in O_p(\alpha_p)} 1/\#\Stab_{G(\Z_p)}(\sigma_p)$$
where almost all factors in the product are equal to $1$.
\end{proposition}

Finally, we have the following result on locally soluble orbits.

\begin{proposition}\label{intorbit}
Let $f(x)$ have coefficients in $\Z$, and assume that $2^{4k}$ divides $c_k$ for all $k$. Then every class in the $2$-Selmer group of $J$ is represented by at least one integral orbit.
\end{proposition}

This follows from the construction of an ideal $I_p$ at each prime. The divisibility hypothesis was needed to construct $I_2$ and an integral orbit over $\Z_2$. It is not necessary when the order $R$ is maximal at $2$, and we suspect that it is not necessary in general.

\section{Construction of fundamental domains}


In the previous sections, we have studied the rational and integral orbits of the representation of $\SO(W)$ on $V$, and shown how the locally soluble orbits are naturally related to the 2-Selmer groups of the Jacobians of hyperelliptic curves.  In order to prove Theorem~1, we now turn to the question of counting these orbits.

In this section, we construct 
convenient fundamental domains for the action of $G(\Z)$ on $V(\R)$.  
In Section~10, we will then count the number of {``irreducible''} points in $V(\Z)$ in these fundamental domains having bounded height, which will then allow us to prove Theorem~1 in Sections~11 and 12 via the appropriate sieve methods.  

We say that an element in $V(\Z)$ is {\it reducible} if it either lies in a distinguished orbit over $\Q$ or has discriminant zero; we say that it is {\it irreducible} otherwise.  The {\it height} of an element $B\in V(\R)$ and of its associated characteristic polynomial $f(x)$ is defined as follows: if $B\in V(\R)$ has associated invariants $c_2,\ldots,c_{2n+1}$ given by
$$f(x)=\det(Ax-B)=x^{2n+1}+c_2x^{2n-1}+\cdots+c_{2n+1},$$
then the height $H(B)$ of $B$ and the height $H(f)$ of $f$ is defined by $$H(B):= H(f):=\max\{|c_k|^{{2n(2n+1)}/{k}}\}_{k=2}^{2n+1}.$$  

To describe our fundamental domains explicitly, we put natural coordinates on $V$.  Recall that $V$ may
be identified with the $n(2n+3)$-dimensional space of symmetric $(2n+1)\times(2n+1)$ matrices having anti-trace zero.  We use $b_{ij}$ to denote the $(i,j)$ entry of $B\in V$; thus the $b_{ij}$ ($i\leq j$; \,$(i,j)\neq (n+1,n+1)$) form a natural set of coordinates on $V$. 
 

\subsection{Fundamental sets for the action of $G'(\R)$ on $V(\R)$}

Recall that we have naturally partitioned the set of elements in $V(\R)$ with $\Delta\neq 0$ into $\sum_{m=0}^n{{2m+1}\choose m}$ components, which we denote by $V^{(m,\tau)}$ for $m=0,\ldots,n$ and $\tau=1,\ldots,{{2m+1}\choose m}$.  For a given value of~$m$, the component $V^{(m,\tau)}$ in $V(\R)$ maps to the component $I(m)$ of characteristic polynomials in $\R^{2n}$ having nonzero discriminant and $2m+1$ real roots. 
Let us define $G'(\R):=\Lambda\times G(\R)$, where $\Lambda=\{\lambda>0\}$.  Then we may view $V(\R)$ also as a representation of $G'(\R)$, where $\Lambda$ acts by scaling.  Furthermore, $G'(\R)$ acts also on each $V^{(m,\tau)}$.  Note that the action of $\lambda\in\Lambda$ on $V(\R)$ takes $c_k$ to $\lambda^{k}c_k$, so that $H(\lambda B) = \lambda^{2n(2n+1)} H(B)$.

We first construct fundamental sets $L^{(m,\tau)}$ for the action of $G'(\R)$ on $V^{(m,\tau)}$, which        will also then end up being fundamental sets for the action of $G(\R)$ on $\{B\in V^{(m,\tau)}:H(B)=1\}$.                                                                                                                   We begin with the case when $m = n$. Fix a self-adjoint operator $T: W \to W$ in the distinguished orbit        having a given characteristic polynomial $f(x)$ with $2n+1$ real roots. The polynomial $f(x)$ of $T$ is       separable, of trace~$0$ and split over $\R$.                                                                  Let $r_1 < r_2 < \cdots < r_{2n+1}$ be the roots of $f(x)$ and let $L = \R[T]$ be the algebra centralizing    $T$ in $\End(W)$. All of the operators in $L$ are self-adjoint and have the same eigenvectors $v_i$ in        $W$.                                                                                                          We thus obtain an isomorphism $L \to \R^{2n+1}$ which maps the polynomial $p(T)$ to the eigenvalues           $(x_1,\ldots,x_{2n+1})=(p(r_1),\ldots,p(r_{2n+1}))$.                                                          With these coordinates, consider the cone $C$ defined by the inequalities                                     $x_1 < x_2 < \cdots < x_{2n+1}$ and                                                                           $\sum x_i = 0$ in the subspace of elements of trace~0 in $L$. The operators in this cone represent the        distinguished orbits with split characteristic polynomial, as the eigenvectors for the eigenvalues $x_1 <     x_2 < \cdots < x_{2n+1}$ are the same as the eigenvectors for $T$.                                            If we scale the elements in $C$ to have height 1 we are in a bounded region. This is the set $L^{(n,1)}$      of $G(\R)$-representatives of height 1 that we take                                                           for $V^{(n,1)}.$                                                                                              Now take another $V^{(n,\tau)}$ where the characteristic polynomial splits: $\tau$ is given by a              collection of sign changes $\alpha$ in                                                                        $L^* = \R^{\ast2n+1}$. If $T'$ in the region has characteristic polynomial $f(x)$, then we can find an        element $T$ in the distinguished orbit and an element $g\in \GL(V)(\R)$ with $gTg^{-1} = T'$ and $g^*g =      \alpha \in L^*$. The elements $g(v_i)$ are now eigenvectors for $T'$ with the appropriately changed signs,    and we obtain a fundamental set $L^{(n,\tau)} = gL^{(n,1)}g^{-1}$. Since we are conjugating by a fixed $g$    (which is a continuous map), this set of $G'(\R)$-representatives for $V^{(n,\tau)}$ also lies in a compact region.                                                                                                                          

Next assume that $m < n$, and choose $T$ in the distinguished orbit with characteristic polynomial~$f(x)$.    In this case, the algebra $L = \R[T]$ is the product of $2m+1$ copies of $\R$ and $n-m$ copies of~$\C$. We    take the                                                                                                      elements of trace zero in $L$ which lie in the product of the cone defined by the inequalities $x_1 < x_2     < \cdots < x_{2m+1}$ in the real part with a fundamental domain for the action of the symmetric group         $S_{n-m}$ on the product of $n-m$ upper half planes minus all the diagonals in the complex part. Again, we    
scale to obtain a set $L^{(m,1)}$ of orbit representatives of height 1 for the action of $G'(\R)$ on $V^{(m,1)}$.                                 
As before, there exists $g\in \GL(V)(\R)$ such that $L^{(m,\tau)}:=gL^{(m,1)}g^{-1}$ is a fundamental set for the action of $G'(\R)$ on $V^{(m,\tau)}$, and $L^{(m,\tau)}$ also lies in a compact set.

We have thus chosen our fundamental sets
$L^{(m,\tau)}$  so that all elements in these sets have height~$1$ and all entries of elements in these sets 
are {bounded}, i.e., the $L^{(m,\tau)}$
all lie in a compact subset of $V(\R)$.  Note that for any fixed
$h\in G'(\R)$ (or for any $h$ lying in a fixed compact subset
$G_0\subset G'(\R)$), the set $hL^{(m,\tau)}$ is also a fundamental
set for the action of $G'(\R)$ on $V^{(m,\tau)}$ that is bounded (independent of $h\in G_0$).


\subsection{
Fundamental sets for the action of $G(\Z)$ on $G'(\R)$}

In \cite{Borel}, Borel (building on earlier work of Borel--Harish-Chandra~\cite{BoHa}) constructs a natural fundamental domain in $G(\R)$ for the left action of $G(\Z)$ on $G(\R)$, which immediately also gives us a fundamental domain $\FF$ in $G'(\R)$ for the left action of $G(\Z)$ on $G'(\R)$.  This set $\FF$ may be expressed in the form $$\FF=
\{ut\theta\lambda:u\in N'(t),\,t\in T',\,\theta\in K,\,\lambda\in\Lambda\},$$ where $N'(t)$ is an absolutely bounded measurable set, which depends on $t\in T'$, of unipotent lower triangular real $(2n+1)\times(2n+1)$ matrices; the set $T'$ is the subset of the torus of diagonal matrices with positive entries given by
\begin{equation}\label{nak}
T' = \left\{\left(\begin{array}{ccccccccc}
 t_1^{-1} & & & & & & \\
 & \ddots & & & & &  \\
  &  &  t_n^{-1}& & & &  \\
  & & & 1\, & & &  \\
 & & &  & \,t_n\, & &  \\
 & & & & & \ddots  &  \\
 & & & & & &\,t_1\, 
\end{array}\right): \;t_1/t_2>c, \;\,t_2/t_3>c, \; \ldots\,,\; t_{n-1}/t_n>c,\;\,t_n>c\right\};
\end{equation}
and $K$ is a maximal compact real subgroup of $G(\R)$.
In the above, $c$ denotes an absolute positive constant.  (Note that $t_i/t_{i+1}$ for $i=1,\ldots,n-1$ and $t_n$ form a set of simple roots for our choice of maximal torus $T$ in $G$.)

It will be convenient in the sequel to parametrize $T'$ in a slightly different way.  Namely, for $k=1,\ldots,n$, we make the substitution $t_k=s_ks_{k+1}\cdots s_n$.  Then we may speak of elements $s=(s_1,\ldots,s_n)\in T'$ as well, and the conditions on $s$ for $s=(s_1,\ldots,s_n)$ to be in $T'$ is simply that $s_k>c$ for all $k=1,\ldots,n$.


\subsection{Fundamental sets for the action of $G(\Z)$ on $V(\R)$}

For any $h\in G'(\R)$, by the previous subsection, 
we see that $\FF hL^{(m,\tau)}$ is the union of $2^{m+n}$
fundamental domains for the action of $G(\Z)$ on $V^{(m,\tau)}$;
here, we regard $\FF hL^{(m,\tau)}$ as a multiset, where the
multiplicity of a point $x$ in $\FF hL^{(m,\tau)}$ is given by the
cardinality of the set $\{g\in\FF\,\,:\,\,x\in ghL^{(m,\tau)}\}$.
Thus, as explained in~\cite[\S2.1]{BS}, a $G(\Z)$-equivalence class $x$ in $V^{(m,\tau)}$ is
represented in this multiset $\sigma(x)$ times, where
$\sigma(x)=\#\Stab_{G(\R)}(x)/\#\Stab_{G(\Z)}(x)$. In particular, $\sigma(x)$ is always a number between 1
and~$2^{m+n}$.

For any $G(\Z)$-invariant set $S\subset V^{(m,\tau)}\cap V(\Z)$, let $N(S;X)$
  denote the number of $G(\Z)$-equivalence classes of irreducible
  elements $B\in S$ satisfying $H(B)<X$.
  Then we conclude that, for any $h\in G'(\R)$, the product $2^{m+n}\cdot
N(S;X)$ is exactly equal to the number of irreducible integer
points in $\FF hL^{(m,\tau)}$ having height less than $X$,
with the slight caveat that the (relatively rare---see
Proposition~\ref{gzbigstab}) points with $G(\Z)$-stabilizers of
cardinality $r$ ($r>1$) are counted with weight $1/r$.

As mentioned earlier, the main obstacle to counting integer points of
bounded height in a single domain $\FF hL^{(m,\tau)}$ is that the
relevant region is not bounded, but rather has cusps going off to
infinity.  We simplify the counting in this cuspidal region by
``thickening'' the cusp; more precisely, we compute the number of
integer points of bounded height in the region $\FF hL^{(m,\tau)}$
by averaging over lots of such fundamental regions, i.e., by averaging
over the domains $\FF hL^{(m,\tau)}$ where $h$ ranges over a certain
compact subset $G_0\in G'(\R)$.  The method, which is an extension of the method of~\cite{dodpf}, is described next.



\section{Counting irreducible integral points of bounded height}

In this section, we derive asymptotics for the number of
$G(\Z)$-equivalence classes of irreducible elements of $V(\Z)$ having bounded invariants.  
We also describe how these
asymptotics change when we restrict to counting only those elements in $V(\Z)$ that satisfy a specified finite set of congruence
conditions.  

Let $L^{(m,\tau)}$ ($m=1,\ldots,n$, $\tau=1,\ldots,{{2m+1}\choose m}$) denote fundamental sets for the action of  $G'(\R)$ on $V^{(m,\tau)}$, as constructed in $\S9.1$, 
and let 
\[c_{m,\tau} = \frac{\Vol(\FF L^{(m,\tau)}\cap \{v\in V(\R):H(v)<1\})}{2^{m+n}}.\]
Then in this section we prove the following theorem:

\begin{theorem}\label{thmcount}
 Fix $m,\tau$.  For any $G(\Z)$-invariant set $S\subset V(\Z)^{(m,\tau)}:=V^{(m,\tau)}\cap V(\Z)$, let $N(S;X)$
  denote the number of $G(\Z)$-equivalence classes of irreducible
  elements $B\in S$ satisfying $H(B)<X$. Then
$$N(V(\Z)^{(m,\tau)};X)=c_{m,\tau}X^{(2n+3)/(4n+2)}+o(X^{(2n+3)/(4n+2)}).$$
\end{theorem}

\subsection{Averaging over fundamental domains}\label{avgsec}

Let $G_0$ be a compact left $K$-invariant set in $G'(\R)$ that is the
closure of a nonempty open set and in which every element has
determinant greater than or equal to $1$. 
Then for any $m,\tau$, we may write
\begin{equation}\label{latter}
N(V(\Z)^{(m,\tau)};X)=\frac{\int_{h\in G_0}\#\{x\in \FF hL\cap V(\Z)^{\irr}:
  H(x)<X\}dh\;}{2^{m+n}\int_{h\in G_0}dh}
\end{equation}
for any Haar measure $dh$ on $G'(\R)$, where $V(\Z)^\irr$ denotes the set of irreducible elements in
$V(\Z)$ and $L$ is equal to $L^{(m,\tau)}$.  The denominator of the
right hand side of (\ref{latter}) is an absolute constant $C_{G_0}^{(m,\tau)}$ greater than
zero.

More generally, for any $G(\Z)$-invariant subset $S \subset
V(\Z)^{(m,\tau)}$, let $N(S;X)$ denote again the number of irreducible
$G(\Z)$-orbits in $S$ having height less than $X$. Let $S^{\irr}$
denote the subset of irreducible points of $S$. Then $N(S;X)$ can be
similarly expressed as

\begin{equation}\label{eq9}
N(S;X)=\frac{\int_{h\in G_0}\#\{x\in \FF hL\cap S^{\irr}: H(x)<X\}dh\;}{C_{G_0}^{(m,\tau)}}.
\end{equation}
We use (\ref{eq9}) to define $N(S;X)$ even for
sets $S\subset V(\Z)$ that are not necessarily $G(\Z)$-invariant.

Now, given $x\in V_\R^{(m,\tau)}$, let $x_L$ denote the {unique} point in $L$
that is $G'(\R)$-equivalent to $x$. We have

\begin{equation}
N(S;X)=\frac{1}{C_{G_0}^{(m,\tau)}}\sum_{\substack{{x\in
      S^{\irr}}\\[.02in]{H(x)<X}}}\int_{h\in G_0} \#\{g \in \FF :
x=ghx_L\} dh.
\end{equation}
For a given $x\in S^{\irr}$, there exist a finite number of elements
$g_1,\ldots,g_k\in G'(\R)$ satisfying $g_jx_L=x$.  We then have
$$\int_{h\in G_0} \#\{g \in \FF :x=ghx_L\} dh=\displaystyle\sum_j\int_{h\in G_0} \#\{g \in \FF :gh=g_j\} dh=\displaystyle\sum_j\int_{h\in G_0\cap\FF^{-1}g_j}dh.$$
As $dh$ is an invariant measure on $G'(\R)$, we have
$$\begin{array}{rcl}
\displaystyle\sum_j\int_{h\in G_0\cap\FF^{-1}g_j}dh&=&\displaystyle\sum_j\int_{h\in G_0g_j^{-1}\cap\FF^{-1}}dh\\[.25in] &=&\displaystyle\sum_j\int_{g\in\FF}\#\{h \in G_0 :gh=g_j\} dg\\[.25in] &=&\displaystyle\int_{g\in \FF} \#\{h \in G_0 :x=ghx_L\}
 dg.\end{array}$$
Therefore,
\begin{eqnarray}\label{eqavg}
N(S;X)&=&\frac{1}{C_{G_0}^{(m,\tau)}}\,\sum_{\substack{x\in S^{\irr}\\[.02in]
    H(x)<X}} \int_{g\in\FF} \#\{h \in G_0 : x=ghx_L\}dg\\
&\!\!=\!\!&  \frac1{C_{G_0}^{(m,\tau)}}\int_{g\in\FF}
\#\{x\in S^\irr\cap gG_0L:H(x)<X\}dg\\[.075in] 
&\!\!=\!\!&  \frac1{C_{G_0}^{(m,\tau)}}\int_{g\in N'(s)T'\Lambda K}
\#\{x\in S^\irr\cap u s\lambda \theta G_0L:H(x)<X\}
dg.
\end{eqnarray}
Let us now fix our Haar measure $dg$ on $G'(\R)$ by setting
\begin{equation}
dg \,=\,  \prod_{k=1}^n t_k^{2k-2n-1} \cdot du\, d^\times t\,d^\times \lambda\, d\theta = 
\prod_{m=1}^n s_k^{k^2-2kn}\cdot du\, d^\times s\,d^\times \lambda\, d\theta\,,
\end{equation}
where $du$ is an invariant measure on the group $N$ of unipotent lower triangular matrices in $G(\R)$, and where we normalize the invariant measure $d\theta$ on $K$ so that $\int_{K} d\theta=1.$

Let us write $E(u,s,\lambda,X) = u s \lambda G_0L\cap\{x\in V^{(m,\tau)}:H(x)<X\}$.
As $KG_0=G_0$ and $\int_K d\theta =1 $, we have
\begin{equation}\label{avg}
N(S;X) = \frac1{C_{G_0}^{(m,\tau)}}\int_{g\in N'(s)T'\Lambda}                              
\#\{x\in S^\irr\cap E(u,s,\lambda,X)\}\prod_{k=1}^n s_k^{k^2-2kn}
\cdot du\, d^\times s\,d^\times \lambda\,.
\end{equation}

We note that the same counting method may be used even if we are interested in counting both
reducible and irreducible orbits in $V(\Z)$. For any set $S\subset V^{(m,\tau)}$, let $N^\ast(S;X)$ be defined by (\ref{avg}), but
where the superscript ``irr'' is removed:
\begin{equation}\label{avg2}
N^\ast(S;X) = \frac1{C_{G_0}^{(m,\tau)}}\int_{g\in N'(s)T'\Lambda}                              
\#\{x\in S\cap E(u,s,\lambda,X)\}\prod_{k=1}^n s_k^{k^2-2kn}
\cdot du\, d^\times s\,d^\times \lambda\,.
\end{equation}
 Thus for a $G(\Z)$-invariant set $S\subset V^{(m,\tau)}$, $N^\ast(S;X)$ counts 
the total (weighted) number of $G(\Z)$-orbits in $S$ having height less than $X$ (not just the irreducible ones). 

The expression (\ref{avg}) for $N(S;X)$, and its analogue (\ref{avg2}) for $N^\ast(S,X)$, will be useful to us in the sections
that follow.

\subsection{An estimate from the geometry of numbers}

To estimate the number of lattice points in $E(u,s,\lambda,X)$, we
have the following proposition due to
Davenport~\cite{Davenport1}. 

\begin{proposition}\label{davlem}
  Let $\mathcal R$ be a bounded, semi-algebraic multiset in $\R^n$
  having maximum multiplicity $m$, and that is defined by at most $k$
  polynomial inequalities each having degree at most $\ell$.  Let $\RR'$
  denote the image of $\RR$ under any $($upper or lower$)$ triangular,
  unipotent transformation of $\R^n$.  Then the number of integer
  lattice points $($counted with multiplicity$)$ contained in the
  region $\mathcal R'$ is
\[\Vol(\mathcal R)+ O(\max\{\Vol(\bar{\mathcal R}),1\}),\]
where $\Vol(\bar{\mathcal R})$ denotes the greatest $d$-dimensional 
volume of any projection of $\mathcal R$ onto a coordinate subspace
obtained by equating $n-d$ coordinates to zero, where 
$d$ takes all values from
$1$ to $n-1$.  The implied constant in the second summand depends
only on $n$, $m$, $k$, and $\ell$.
\end{proposition}
Although Davenport states the above lemma only for compact
semi-algebraic sets $\mathcal R\subset\R^n$, his proof adapts without
significant change to the more general case of a bounded semi-algebraic
multiset $\mathcal R\subset\R^n$, with the same estimate applying also to
any image $\mathcal R'$ of $\mathcal R$ under a unipotent triangular
transformation.

\subsection{Estimates on reducibility}\label{redsec}

In this subsection, we describe the relative frequencies with which
reducible and irreducible elements sit inside various parts of the
fundamental domain $\FF hL$, as $h$ varies over the compact region $G_0$.

We begin by describing some sufficient conditions that guarantee that
a point in $V(\Z)$ is reducible.  Recall that a point $B\in V(\Z)$ is called {\it reducible} if it lies in a distinguished rational orbit or it has discriminant zero.

\begin{lemma}\label{lem1}
  Let $B\in V(\Z)$ be an element such that, for some $k\in \{1,\ldots,n\}$, the top left $k \times (2n+1-k)$
block of entries of $B$ are all equal to $0$.  Then $\Disc(\Det(Ax-B))=0$, and thus $B$ is reducible. 
\end{lemma}

\begin{proof}
If the top left $k \times (2n+1-k)$ block of $B$ is zero, then $B$ (and indeed $Ax_0-B$ for any $x_0\in\C$) may be viewed as a block anti-triangular matrix, with square blocks of length $k$, $2n+1-2k$, and $k$ respectively on the anti-diagonal.  Furthermore, the two blocks of length $k$ will be transposes of each other.  There then must be a scalar multiple $Ax_0$ of $A$ such that $Ax_0+B$ has the property that one and thus both of these square blocks of length $k$ are singular.  It follows that $x_0$ is a double root of $\Det(Ax-B)$, yielding $\Disc(\Det(Ax-B))=0$, as desired. 
\end{proof}

We also have the following consequence of Proposition~\ref{dist}:

\begin{lemma}\label{lem2}
  Let $B=(b_{ij})\in V(\Z)$ be an element such that, for all pairs $(i,j)$ satisfying $i+j< 2n+1$, we have $b_{ij}=0$.  Then $B$ is reducible. 
\end{lemma}

We are now ready to give an estimate on the number of 
irreducible elements $B=(b_{ij})\in \FF h L \cap V(\Z)$, on average, satisfying
$b_{11}=0$:

\begin{proposition}\label{hard}
Let $h$ take a random value in $G_0$ uniformly with respect to the Haar measure
$dh$.  Then the expected number of 
irreducible elements $B\in\FF h L^{(m,\tau)}\cap V(\Z)$ such that 
$H(B)< X$  and $b_{11}=0$
is $O_\varepsilon(X^{\tfrac{2n+3}{4n+2} - \tfrac{1}{2n(2n+1)}+\varepsilon})$.
\end{proposition}

\begin{proof}
We follow the method of proof of \cite[Lemma~11]{dodpf}.  Let $U$ denote the set of all $n(2n+3)$ variables
$b_{ij}$ corresponding to the coordinates on $V(\Z)$.  Each variable $b\in U$ has a {\it weight}, defined as follows. The
action of $s=(s_1,\ldots,s_n)\cdot\lambda$ on $B\in
V(\R)$ causes each variable $b$ to multiply by a certain weight which we
denote by $w(b)$.  These weights $w(b)$ are evidently rational
functions in $\lambda,s_1,\ldots,s_n$.  If we use $w_1,\ldots,w_{2n}$ to denote the quantities $s_1,s_2,\ldots,s_n,s_n,\ldots, s_2,s_1$, respectively, then we have the explicit formula 
\begin{equation}\label{wtformula}
w(b_{ij})\,=\,\lambda \cdot s_1^{-2}\cdots s_n^{-2} \cdot(w_1\cdots w_{i-1})\cdot(w_1\cdots w_{j-1}).
\end{equation}
In particular, if $i+j=2n+2$, then we see that $w(b_{ij})=\lambda$.  If $i+j=2n+1$, with $i\leq n$, then $w(b_{ij})=\lambda s_i^{-1}$.
Finally, note that $\prod_{b\in U}w(b)=\lambda^{n(2n+3)}$.

Let $U_0\subset U$ be any subset of variables in $U$ containing $b_{11}$.  We now give an upper estimate on the total number of irreducible $B\in \FF h L^{(m,\tau)}\cap V(\Z)$ such that all variables in $U_0$ vanish, but all variables in $U\setminus U_0$ do not.  To this end, let 
$V(U_0)$ denote the set of all such $B\in V(\Z)$.  Furthermore, let $U_1$ denote the 
the set of variables having the minimal weights $w(b)$ among the variables $b\in
U\setminus U_0$ (by ``minimal weight'' in $U\setminus U_0$, we mean
there is no other variable $u\in U\setminus U_0$ with weight having
equal or smaller exponents for all parameters $\lambda, s_1,\ldots,s_n$).  As explained in~\cite{dodpf}, given $U_1$, it suffices to assume that $U_0$ in fact consists of all variables in $U$ having weight smaller than that of some variable in $U_1$.

In that case, if $U_1$ contains any variable $b_{ij}$ with $i+j>2n+1$, then every element $B\in V(U_0)$ will be reducible by Lemma~\ref{lem1}.  
If $U_1$ consists exactly of those variables $b_{ij}$ for which $i+j=2n+1$ (i.e., $U_0$ contains all variables $b_{ij}$ with $i+j<2n+1$), then every element $B\in V(U_0)$ will be reducible by Lemma~\ref{lem2}.
Thus, we may assume that all variables $b_{ij}\in U_1$ satisfy
$i+j\leq 2n+1$, and at least one variable $b_{ij}\in U_1$ satisfies $i+j<2n+1$. 

In particular, there exist variables $\beta_1,\ldots,\beta_n\in U_1$ such that, for all $k$, the weight $w(\beta_k)$ has equal or smaller exponents for all parameters $\lambda, s_1,\ldots,s_n$ when  compared to the weight $w(b_{k,2n+1-k})=\lambda s_k^{-1}$.

For each subset $U_0\subset U$ as above, we wish to show that
$N(V(U_0);X)$, as defined by (\ref{avg}), is $O_\varepsilon(X^{\tfrac{2n+3}{4n+2} - \tfrac{1}{2n(2n+1)}+\varepsilon})$.  
Since the set $N'(s)$ is absolutely bounded, the equality (\ref{avg2}) implies that
\begin{equation}\label{estv0s}
N^*(V(U_0);X)\ll
\int_{\lambda=c'}^{X^{1/(2n(2n+1))}}\!\!\!\int_{s_1,\ldots,s_n=c}^\infty
\sigma(V(U_0))
\prod_{k=1}^n s_k^{k^2-2kn}\cdot
d^\times\! s \,d^\times\!\lambda,
\end{equation}
where $\sigma(V(U_0))$ denotes the maximum possible number of integer points
in the region $E(u,s,\lambda,X)$ that also satisfy the conditions 
\begin{equation}\label{cond}
\mbox{$b=0$ for $b\in U_0$ and $|b|\geq 1$ for $b\in U_1$}.
\end{equation}

By our construction of $L$, all entries of elements $B\in G_0L$ are uniformly bounded.
Let $C$ be a constant that bounds the absolute value of all variables
$b\in U_1$ over all elements $B\in G_0L$.
Then, for an element $B\in E(u,s,\lambda,X)$, we have
\begin{equation}\label{condt}
|b|\leq C{w(b)}
\end{equation}
for all variables $b\in U_1$; therefore, the number of integer points in $E(u,s,\lambda,X)$
satisfying $b=0$ for all $b\in U_0$ and $|b|\geq 1$ for all $b\in U_1$ will be nonzero only if we have
\begin{equation}\label{condt1}
C{w(b)}\geq 1
\end{equation}
for all weights $w(b)$ such that $b\in U_1$.  
Thus if the condition (\ref{condt1}) holds for all weights $w(b)$
corresponding to $b\in U_1$, then---by the definition of $U_1$---we
will also have $Cw(b)\gg 1$ for all weights $w(b)$ such that $b\in
U\setminus U_0$.  In particular, note that we have 
\begin{equation}\label{betaest}
Cw(\beta_k)\gg 1
\end{equation}
for all $k=1,\ldots,n$.


Therefore, if the region $\BB=\{B\in
E(u,s,\lambda,X):b=0\;\forall b\in U_0;\; |b|\geq 1\; \forall
b\in U_1\}$ contains an integer point, then (\ref{condt1}) and
Proposition~\ref{davlem} together imply that the number of integer points
in $\BB$ is $O(\Vol(\BB))$ (where we are computing the volume in the $\R$-subspace of $V(\R)$
spanned by the coordinates $b\in U\setminus U_0$); this is because the volumes of all the projections of
$u^{-1}\BB$ will in that case also be $O(\Vol(\BB))$.  
Now clearly
\[\Vol(\BB)=O\Bigl(\prod_{b\in U\setminus U_0} w(b)\Bigr),\]
so we obtain
\begin{equation}\label{estv1s}
N(V(U_0);X)\ll
\int_{\lambda=c'}^{X^{1/(2n(2n+1))}} \!\!\!\int_{s_1,\ldots,s_n=c}^\infty
\prod_{b\in U\setminus U_0} w(b)
\, \prod_{k=1}^n s_k^{k^2-2kn}
\cdot d^\times\! s \,d^\times\!\lambda.
\end{equation}

We wish to show that the latter integral is bounded by $O_\varepsilon(X^{\tfrac{2n+3}{4n+2} - \tfrac{1}{2n(2n+1)}+\varepsilon})$ for every choice of $U_0$.  If, for example, the exponent of
$s_k$ in (\ref{estv1s}) is nonpositive for all $k$ in $\{1,\ldots,n\}$,
then it is clear that the resulting integral will be at most
$O_\varepsilon(X^{(n(2n+3)-|U_0|)/(2n(2n+1))+\varepsilon})$ in value, since each $s_k$ is bounded above by a power of $X$ by (\ref{betaest}).  

For cases where this nonpositive exponent condition does not hold, we observe that, due to (\ref{betaest}), the integrand in
(\ref{estv1s}) may be multiplied by any product $\pi$ of the variables $\beta_k$ ($k=1,\ldots,n$) without harm, and the
estimate (\ref{estv1s}) will remain true.  
Extend the notation $w$ multiplicatively, i.e., $w(ab)=w(a)w(b)$, and for any subset $U'\subset U$, write
$w(\prod_{b\in U'}b)=\lambda^{|U'|}\cdot \prod s_k^{-e_k(U')}$.  Then we set 
$$\pi = \pi(U_0):=\prod_{k=1}^n \beta_k^{\,\max\{0\,,\,e_k(U_0) +k^2-2kn\}}.$$
We may then apply the
inequalities (\ref{condt}) to each of the variables in $\pi$, yielding
\begin{equation}\label{estv2s}
N(V(U_0);X)\ll
\int_{\lambda=c'}^{X^{1/(2n(2n+1))}} \!\!\!\int_{s_1,\ldots,s_n=c}^\infty
\prod_{b\in U\setminus U_0} w(b)\;w(\pi)
\, \prod_{k=1}^n s_k^{k^2-2kn}
\cdot d^\times\! s \,d^\times\!\lambda.
\end{equation}
By the definition of $\beta_k$, we have $w(\beta_k)\ll \lambda s_k^{-1}$, implying
\begin{equation}\label{estv3}
N(V(U_0);X)\ll
\int_{\lambda=c'}^{X^{1/(2n(2n+1))}} \!\!\!\int_{s_1,\ldots,s_n=c}^\infty
\prod_{b\in U\setminus U_0} w(b)\;\lambda^{|\pi|}
\, \prod_{k=1}^n s_k^{k^2-2kn-\max\{0\,,\,e_k(U_0)+k^2-2kn\}}
\cdot d^\times\! s \,d^\times\!\lambda.
\end{equation}

In the above, we have chosen the factor $\pi$ so that the exponent of each $s_m$ in
the integrand of (\ref{estv3}) is nonpositive.  Thus we obtain from (\ref{estv3}) that
$N(V(U_0);X)=O_\varepsilon(X^{(n(2n+3)-\#U_0+\#\pi)/(2n(2n+1))+\varepsilon})$, where $\#\pi$ denotes the
total number of variables of $U$ appearing in $\pi$ (counted with
multiplicity). It thus suffices now to check that we always have $\#U_0-\#\pi(U_0)>0$.

Let $U^-$ denote the subset of variables $b_{ij}$ in $U$ such that $i+j<2n+1$.  
It then suffices to 
check the condition $\#U_0-\#\pi(U_0)>0$ for all proper nonempty subsets $U_0$ of $U^-$, and this is the content of the following combinatorial lemma.  

\begin{lemma}
Let $U_0\subseteq U^-$.  Then 
\begin{equation}\label{keyeq}
\sum_{k=1}^n \max\{0\,,\,e_k(U_0)+k^2-2kn\}\leq |U_0|
\end{equation}
with equality if and only if $U_0=\varnothing$ or $U_0=U^-$.
\end{lemma}

\begin{proof}
We prove the lemma by induction on $n$.  The lemma is directly verified for $n=1$.  If $n>1$, then we note that the weights for $s_{k}$ ($k>1$) agree with the weights for $s_{k-1}$ when $n$ is replaced by $n-1$.  Let $U_0(1)$ (resp.\ $U_0(\geq 2)$)  denote the set of all variables $b_{ij}\in U_0$ such that $i=1$ (resp.\ $i\geq 2$), and define $U^-(1)$ and $U^-(\geq 2)$ analogously.  Then, by the induction hypothesis, we have
\begin{equation}\label{inducthyp}
\sum_{k=2}^{n} \max\{0\,,\,e_{k}(U_0(\geq 2))+(k-1)^2-2(k-1)(n-1)\}\leq |U_0(\geq 2)|
\end{equation}
with equality if and only if $U_0(\geq 2)=\varnothing$ or $U_0(\geq 2)=U^-(\geq 2)$.

Let $|U_0(1)|=r$, so that $r\in\{0,\ldots,2n-1\}$.  Then to prove (\ref{keyeq}), we may assume that $U_0(1)=\{b_{11},b_{12},\ldots, b_{1r}\}$, for any other choice of $U_0(1)\subset U^-(1)$ of size $r$ could only decrease the sum on the left hand side of (\ref{keyeq}).

We first consider the case where $r<n$; in that case, we have
$$
\begin{array}{rcl}
 & &\max\{0\,,\,e_1(U_0(1))+1-2n\}+\displaystyle\sum_{k=2}^n \max\{0\,,\,e_k(U_0)+k^2-2kn\} 
 \\[.1in] &=& 0+\displaystyle\sum_{k=2}^n \max\{0\,,\,e_k(U_0(\geq2))+(k-1)^2-2(k-1)(n-1)+e_k(U_0(1))+(1-2n)\} \\[.1in]
 &\leq& |U_0(\geq2)| + \displaystyle\sum_{k=2}^n \max\{0\,,\,e_k(U_0(1))+(1-2n)\} \\[.275in]
&=& |U_0(\geq2)|\leq |U_0|,
\end{array}
$$
with equality if and only if $r$ is zero and all terms on the left hand sides of (\ref{keyeq}) and (\ref{inducthyp}) are zero, i.e., $U_0=\varnothing$.

Now consider the case where $r\geq n$.  In that case, we have similarly
$$
\begin{array}{rcl}
 & &\displaystyle\sum_{k=1}^n \max\{0\,,\,e_k(U_0)+k^2-2kn\} 
 \\[.1in] 
 &\leq& |U_0(\geq2)| + \displaystyle\sum_{k=1}^n \max\{0\,,\,e_k(U_0(1))+(1-2n)\} \\[.225in]
&=& |U_0(\geq2)|+2(r-n-1)-1\leq |U_0|,
\end{array}
$$
with equality if and only if $r$ is $2n-1$ and we have equality in (\ref{inducthyp}) with all terms nonzero, i.e., $U_0=U^-$.
\end{proof}

\noindent
This completes the proof of Proposition~\ref{hard}.
\end{proof}

We now give an estimate on the number of 
reducible elements $B=(b_{ij})\in \FF h L \cap V(\Z)$, on average, satisfying
$b_{11}\neq 0$:

\begin{proposition}\label{hard2}
Let $h$ take a random value in $G_0$ uniformly with respect to the measure
$dh$.  Then the expected number of 
reducible elements $B\in\FF h L^{(m,\tau)}\cap V(\Z)$ such that 
$H(B)< X$  and $b_{11}\neq 0$
is $o(X^{(2n+3)/(4n+2)})$.
\end{proposition}
%
We defer the proof of this lemma to the end of the section.

 We also have the following proposition which bounds the number of
$G(\Z)$-equivalence classes of irreducible elements in 
$V^{(m,\tau)}$
with height less than $X$ that have
large stabilizers inside $G(\Q)$; we again
defer the proof to the end of the section.

\begin{proposition}\label{gzbigstab}
  The number of $G(\Z)$-equivalence classes of irreducible
  elements in
  $V(\Z)$
having height less than $X$
whose stabilizer in $G(\Q)$ has size greater
  than $1$ is $o(X^{(2n+3)/(4n+2)})$.
\end{proposition}

\subsection{The main term}

Fix again $m,\tau$ and let $L=L^{(m,\tau)}$.  The work of the previous subsection shows that, in order to obtain Theorem~\ref{thmcount}, it suffices to count those integral elements $B\in \FF h L$ of bounded height for which $b_{11}\neq 0$, as $h$ ranges over $G_0$.

Let $\RR_X(hL)$ denote the region $\FF h L\cap \{B\in
V(\R):H(B)<X\}$.  Then we have the following result counting the
number of integral points in $\RR_X(hL)$, on average, satisfying
$b_{11}\neq 0$:

\begin{proposition}\label{nonzerob11}
  Let $h$ take a random value in $G_0$ uniformly with
  respect to the Haar measure $dh$.  Then the expected
  number of elements $B\in\FF hL\cap V(\Z)$ such that
  $|H(B)|< X$ and $b_{11}\neq0$ is $\Vol(\RR_X(L))
  + O(X^{\tfrac{2n+3}{4n+2}-\tfrac{1}{2n(2n+1)}})$.
\end{proposition}

\begin{proof}
Following the notation in the proof of Proposition~\ref{hard}, let $V^{(m,\tau)}(\!\varnothing\!)$ denote the subset of $V(\Z)^{(m,\tau)}$ such that $b_{11}\neq 0$.  
We wish to show that
\begin{equation}\label{toprove2}
N^*(V^{(m,\tau)}(\!\varnothing\!);X)=\Vol(\RR_X(L)) + O(X^{\tfrac{2n+3}{4n+2}-\tfrac{1}{2n(2n+1)}}).
\end{equation}
We have
\begin{equation}\label{translate2}
  N^*(V^{(m,\tau)}(\!\varnothing\!);X)=\frac{1}{C^{(m,\tau)}_{G_0}}\int_{\lambda=c'}^{X^{1/(2n(2n+1))}} \!\!\!
\int_{s_1,\ldots,s_n=c}^\infty
\int_{u\in    N'(s)} 
  \sigma(V(\!\varnothing\!))
  \prod_{k=1}^n s_k^{k^2-2kn}
  \cdot du\,d^\times\! s \,d^\times\!\lambda,
\end{equation}
where $\sigma(V(\!\varnothing\!))$ denotes the number of integer points
in the region $E(u,s,\lambda,X)$ satisfying $|b_{11}|\geq 1$.
Evidently, the number of integer points in $E(u,s,\lambda,X)$
with $|b_{11}|\geq 1$ can be nonzero only if we have
\begin{equation}\label{condt2}
C{w(b_{11})}=C\cdot\frac{\lambda}{s_1^2s_2^2\cdots s_n^2}\geq 1.
\end{equation}
Hence, if the region $\BB=\{B\in
E(u,s,\lambda,X):|b_{11}|\geq 1\}$
 contains an integer point, then (\ref{condt2}) and
Proposition~\ref{davlem} imply that the number of integer points in $\BB$
is $\Vol(\BB)+O(\Vol(\BB)/w(b_{11}))$, because all
smaller-dimensional projections of $u^{-1}\BB$ are clearly bounded by
a constant times the projection of $\BB$ onto the hyperplane $b_{11}=0$ (since
$b_{11}$ has minimal weight).

Therefore, since $\BB=E(u,s,\lambda,X)-\bigl(E(u,s,\lambda,X)-\BB\bigr)$, 
we may write
\begin{eqnarray}\label{bigint}\nonumber
\!N^\ast(V^{(m,\tau)}(\!\varnothing\!);X) &\!\!\!\!\!\!\!=\!\!\!\!\!\!& \!\frac1{C^{(m,\tau)}_{G_0}}
\!\!\int_{\lambda=c'}^{X^{1/(2n(2n+1))}}\!\!\!\!\!
\int_{s_1,\ldots,s_n=c}^{\infty} \int_{u\in N'(s)}\!\!
\Bigl(\Vol\bigl(E(u,s,\lambda,X)\bigr)\!-\!\Vol\bigl(E(u,s,\lambda,X)\!-\!\BB\bigr) 
 \\[.01in] & & \,\,\,\,\,\,\,\,\,\,
\,\,\,\,\,\,\,+O(\max\{\lambda^{n(2n+3)-1}s_1^2s_2^2\cdots s_n^2,1\})
\Bigr)   
\, \prod_{k=1}^n s_k^{k^2-2kn}\cdot du\,
d^\times s\, d^\times \lambda.
\end{eqnarray}
The integral of the first term in (\ref{bigint}) is $\int_{h\in G_0}
\Vol(\RR_X(hL))dh$.
Since $\Vol(\RR_X(hL))$ does not
depend on the choice of $h\in G_0$,
the latter integral is simply $C^{(m,\tau)}_{G_0}\cdot \Vol(\RR_X(L))$.

To estimate the integral of the second term in (\ref{bigint}), let $\BB'=
E(u,s,t,X)-\BB$, and for each $|b_{11}|\leq 1$, let $\BB'(b_{11})$ be
the subset of all elements $B\in\BB'$ with the given value of
$b_{11}$.  Then the $(n(2n+3)-1)$-dimensional volume of $\BB'(b_{11})$ is at most 
$O\Bigl(\prod_{b\in U\setminus\{b_{11}\}}w(b)\Bigr)$, and so we
have the estimate 
\[\Vol(\BB') \ll \int_{-1}^1 \prod_{b\in
  U\setminus\{b_{11}\}}w(b)
\,\,db_{11} = O\Bigl(\prod_{b\in U\setminus\{b_{11}\}}w(b)\Bigr).\]
The second term of the integrand in (\ref{bigint}) can thus be
absorbed into the third term.
 
Finally, one easily computes the
integral of the third term in (\ref{bigint}) to be
$O(X^{\tfrac{2n+3}{4n+2}-\tfrac{1}{2n(2n+1)}})$.  We thus obtain
\begin{equation}\label{obtain}
N^\ast(V(\Z)^{(m,\tau)};X) = 
\Vol(\RR_X(L)) 
+ O(X^{\tfrac{2n+3}{4n+2}-\tfrac1{2n(2n+1)}}),
\end{equation}
as desired.
\end{proof}

Propositions~\ref{hard}, \ref{hard2}, \ref{gzbigstab}, and \ref{nonzerob11} now yield Theorem~\ref{thmcount}.

\subsection{Computation of the volume}\label{secvol}

Fix again $m,\tau$.  In this subsection, we describe how to compute
the volume of $\mathcal R_X(L^{(m,\tau)})$.  

To this end, let $R^{(m,\tau)}:=\Lambda L^{(m,\tau)}$.
For each 
$(c_2,\ldots,c_{2n+1})\in I(m)$, the set $R^{(m,\tau)}$ contains exactly one point $p^{(m,\tau)}(c_2,\ldots,c_{2n+1})$
having invariants $c_2,\ldots,c_{2n+1}$. Let $R^{(m,\tau)}(X)$ denote the set of
all those points in $R^{(m,\tau)}$ having height less than $X$. Then
$\Vol(\mathcal R_X(L^{(m,\tau)}))=\Vol(\FF^{\ss}\cdot R^{(m,\tau)}(X))$,
where $\FF^{\ss}$ denotes the fundamental domain $N'T'K$ for the action
of $G(\Z)$ on $G(\R)$
(here $N'$, $T'$, and~$K$ are as in~(\ref{nak})).

The set $R^{(m,\tau)}$ is in canonical one-to-one correspondence with the
set $I(m)\subset \R^{2n}$.
We thus obtain a
natural measure on each of these sets $R^{(m,\tau)}$, given by
$dr=dc_2\cdots dc_{2n+1}$.  Let $\omega$ be a differential which generates the rank
$1$ module of top-degree differentials of $G$ over $\Z$.
We begin with the following key proposition, which describes how one can change measure
from $dv$ on $V$ to $\omega(g)\,dr$ on $G\times R$:
\begin{proposition}\label{genjac}
  Let $K$ be $\C$, $\R$, or $\Z_p$ for any prime $p$. Let $dv$ be the
  standard additive measure on $V(K)$. Let $R$ be an open subset of
  $K^{2n}$ and let $s:R\to V(K)$ be a continuous function such that the
  invariants of $s(c_2,\ldots,c_{2n+1})$ are given by $c_2,\ldots,c_{2n+1}$.  Then 
  there exists a rational nonzero constant~$\mathcal J$ $($independent of $K,$ $R,$ and $s)$ such that, \,for any
  measurable function $\phi$ on $V(K)$, we have
  \begin{equation}\label{eqjacpgl2}
\int_{v \in G(K)\cdot
  s(R)}\phi(v)dv\,=\,|\mathcal J|\int_{R}\int_{G(K)}
\phi(g\cdot s(c_2,\ldots,c_n))\,\omega(g) \,dr
  \end{equation}
where we regard $G(K)\cdot s(R)$ as a
multiset.
\end{proposition}
The proof of Proposition~\ref{genjac} is identical to that of \cite[Prop.~3.11]{BS} (where we use the important fact that the dimension $n(2n+3)$ of $V$ is equal to the sum of the degrees of the invariants $c_2,\ldots,c_{2n+1}$ for the action of $G$ on $V$).

Proposition \ref{genjac} may now be used
to give a convenient expression for the volume of the multiset~$\mathcal R_X(L^{(m,\tau)})$:
\begin{eqnarray}\nonumber
\!\!\!\!\int_{\mathcal R_X(L^{(m,\tau)})}\!\!\!\!\!dv=\int_{\FF^\ss\cdot R^{(m,\tau)}(X)}\!\!\!\!\!dv&=&
|\mathcal J|\cdot\int_{R^{(m,\tau)}(X)}\int_{\FF^\ss}\omega(g)\,dr \\ \nonumber &=&|\mathcal J|\cdot \Vol(G
(\Z)\backslash G
(\R))\cdot \int_{R^{(m,\tau)}(X)}dr \\ \label{volexp}&=& |\mathcal J|\cdot \Vol(G
(\Z)\backslash G
(\R))\cdot \int_{{\scriptstyle{(c_2,\ldots,c_{2n+1})\in I(m)}}\atop{\scriptstyle{H(c_2,\ldots,c_{2n+1})<
X}}}dc_2\cdots dc_{2n+1}.\;\;\;\;\;
\end{eqnarray}

\subsection{Congruence conditions}\label{secbqcong}

In this subsection, we prove the following version of Theorem \ref{thmcount} where we
count elements of $V(\Z)$ satisfying some finite set of congruence conditions:

\begin{theorem}\label{cong2}
Suppose $S$ is a subset of $V(\Z)$ defined by 
congruence conditions modulo finitely many prime powers. Then 
\begin{equation}\label{ramanujan}
N(S\cap V^{(m,\tau)};X)
  \;=\; N(V(\Z)^{(m,\tau)};X)\cdot
  \prod_{p} \mu_p(S)+o(X^{(2n+3)/(4n+2)}),
\end{equation}
where $\mu_p(S)$ denotes the $p$-adic density of $S$ in $V(\Z)$.

\end{theorem}

To obtain Theorem~\ref{cong2}, suppose $S$ is defined by congruence conditions modulo some integer $m$. Then $S$ may be viewed as the union of (say) $k$ translates ${\mathcal L}_1,\ldots,{\mathcal L}_k$ of the lattice $m\cdot V(\Z)
$. For each such lattice translate ${\mathcal L}_j$, we may use formula (\ref{avg}) and the discussion following that formula to compute $N(S;X)$, but where each $d$-dimensional volume is scaled by a factor of $1/m^d$ to reflect the fact that our new lattice has been scaled by a factor of $m$. For a fixed value of $m$, we thus obtain
\begin{equation}\label{lat}
N({\mathcal L}_j\cap V^{(m,\tau)};X) = m^{-n(2n+3)} \Vol(\mathcal R_X(L)) +o(X^{(2n+3)/(4n+2)}).
\end{equation}
Summing (\ref{lat}) over $j$, and noting that $km^{-n(2n+3)} = \prod_p \mu_p(S)$, yields Theorem~\ref{cong2}.

We will also have occasion to use the following weighted version of
Theorem \ref{cong2}; the proof is identical.

\begin{theorem}\label{cong3}
  Let $p_1,\ldots,p_k$ be distinct prime numbers. For $j=1,\ldots,k$,
  let $\phi_{p_j}:V(\Z)\to\R$ be a $G(\Z)$-invariant function on
  $V(\Z)$ such that $\phi_{p_j}(x)$ depends only on the congruence
  class of $x$ modulo some power $p_j^{a_j}$ of $p_j$.  Let
  $N_\phi(V(\Z)^{(m,\tau)};X)$ denote the number of irreducible
  $G(\Z)$-orbits in $V(\Z)^{(m,\tau)}$ having height bounded by $X$,
  where each orbit $G(\Z)\cdot B$ is counted with weight
  $\phi(B):=\prod_{j=1}^k\phi_{p_j}(B)$. Then 
\begin{equation}
N_\phi(V(\Z)^{(m,\tau)};X)
  = N(V(\Z)^{(m,\tau)};X)
  \prod_{j=1}^k \int_{B\in V({\Z_{p_j}})}\tilde{\phi}_{p_j}(B)\,dB+o(X^{(2n+3)/(4n+2)}),
\end{equation}
where $\tilde{\phi}_{p_j}$ is the natural extension of ${\phi}_{p_j}$
to $V(\Z_{p_j})$ by continuity, $dB$ denotes the additive
measure on $V(\Z_{p_j})$ normalized so that $\int_{B\in
  V(\Z_{p_j})}dB=1$, and where the implied constant in the error term
depends only on the local weight functions ${\phi}_{p_j}$.
\end{theorem}

\subsection{Proof of Propositions~\ref{hard2} and \ref{gzbigstab}}

We may use the results of the previous subsection to prove Propositions~\ref{hard2} and \ref{gzbigstab}.  Indeed, to prove Proposition~\ref{hard2}, we note that if an element $B\in V(\Z)$ 
is reducible over $\Q$ then it also must be reducible modulo $p$ 
(i.e., its image in $V(\F_p)$ lies in a distinguished orbit or has discriminant zero)
for every $p$.

Let $S^\red$ denote the set of elements in $V(\Z)$ that are reducible over $\Q$, and let $S^\red_p$ denote the set of all elements in $V(\Z)$ that are reducible modulo $p$.  Then $S^\red\subset \cap_p S^\red_p$.  Let $S^\red(Y)=\cap_{p<Y}S^\red_p$, and let us use as before $V^{(m,\tau)}(\!\varnothing\!)$ to denote the set $B\in V(\Z)^{(m,\tau)}$ such that $b_{11}\neq 0$.  Then the proof of Theorem~\ref{cong2} (without assuming Propositions~\ref{hard2} and \ref{gzbigstab}!) gives that
\begin{equation}\label{SYcount}
N^\ast(S^\red(Y)\cap V^{(m,\tau)}(\!\varnothing\!);X)
  \;\leq\; 2^{2n}\cdot N^\ast(V^{(m,\tau)}(\!\varnothing\!);X)\cdot
  \prod_{p<Y} \mu_p(S^\red_p)+o(X^{(2n+3)/(4n+2)}).
\end{equation}
To estimate $\mu_p(S^\red_p)$, we recall from Subsection~\ref{finite} that the number of elements $B\in V(\F_p)$ having any given associated separable polynomial $f(x)=\det(Ax-B)$ is $\#\SO(W)(\F_p)$.  The number of these elements $B$ that lie in the distinguished orbit is $\#\SO(W)(\F_p)/2^m$, where $m+1$ is the number of irreducible factors of $f$ in $\F_p[x]$.  For $p>2n+1$, it is an elementary and well-known calculation that the number of separable polynomials $$f(x)=x^{2n+1}+c_{2}x^{2n-1}+\cdots+c_{2n+1}$$ over $\F_p$ that are reducible over $\F_p$ is $p^{2n}\cdot 2n/(2n+1)+O(p^{2n-1})$, where the implied $O$-constant is independent of $p$.  Since such a reducible polynomial $f(x)$ has at least $2$ factors in $\F_p[x]$, we see that at least $\#\SO(W)(\F_p)/2$ of the elements $B$ having associated polynomial $f(x)$ do not lie in a distinguished orbit.  We conclude that
$$
\mu_p(S^\red_p) \leq 1 - \frac{2n}{2n+1}\cdot\frac{1}{2} + O(1/p).
$$
Combining with (\ref{SYcount}), we see that
$$\lim_{X\to\infty}\frac{N^\ast(S^\red\cap V^{(m,\tau)}(\!\varnothing\!);X)}{X^{(2n+3)/(4n+2)}}
 \;\ll\;  \prod_{p<Y} \mu_p(S^\red_p) \;\ll \; \prod_{p<Y}\Bigl(1-\frac{n}{2n+1}+O(1/p)\Bigr).
 $$
 When $Y$ tends to infinity, the product on the right tends to 0, proving Proposition~\ref{hard2}.
  
We may proceed similarly with Proposition~\ref{gzbigstab}.  If an element $B\in V(\Z)$ with nonzero discriminant has a non-trivial stabilizer, then the associated polynomial $f(x)=\det(Ax-B)$ must be reducible in $\Q[x]$, and thus must be reducible in $\F_p[x]$ for all $p$.  

Let $S^\bigstab\subset V(\Z)$ denote the elements $B\in V(\Z)$ such that $f(x)=\det(Ax-B)$ is reducible in $\Q[x]$, and let $S_p^\bigstab\subset V(\Z)$ denote the elements $B\in V(\Z)$ such that $f(x)$ is reducible in $\F_p[x]$.  Let $S^\bigstab(Y)=\cap_{p<Y} S^\bigstab_p$.  Then we have by the same argument that 
$$\lim_{X\to\infty}\frac{N^\ast(S^\bigstab\cap V^{(m,\tau)}(\!\varnothing\!);X)}{X^{(2n+3)/(4n+2)}}
 \;\ll\; \prod_{p<Y} \mu_p(S^\bigstab_p) \;\ll\; \prod_{p<Y}\Bigl(1-\frac{2n}{2n+1}+O(1/p)\Bigr).
 $$
Letting $Y$ tend to infinity, and noting Proposition~\ref{hard}, now proves also Proposition~\ref{gzbigstab}.

\section{Sieving to Selmer elements}

For each prime $p$, let
$\Sigma_p$ be a closed subset of elements $(c_2,\ldots,c_{2n+1})$ in $\Z_p^{2n}\setminus\{\Delta=0\}$ whose boundary has measure zero.  To such a collection
$(\Sigma_p)_p$ of local specifications, we may associate the set $F_\Sigma$ of hyperelliptic curves, where 
$C(c_2,\ldots,c_{2n+1})\in F_\Sigma$ if and only if
$(c_2,\ldots,c_{2n+1})\in\Sigma_p$ for~all~$p$. We say then that $F_\Sigma$ is a family of
hyperelliptic curves $C(c_2,\ldots,c_{2n+1})$ over $\Q$ that is {\it defined by congruence conditions}.

We may also impose ``congruence conditions at infinity'' on elements $(c_2,\ldots,c_{2n+1})$ in $\Z^{2n}$.
Let additionally $\Sigma_\infty$ denote a union of components $I(m)$ (as defined in $\S6.3$), where $m$ ranges over the elements of some subset of $\{0,\ldots,n\}$.  
Then we may again associate to the collection $(\Sigma_\nu)_\nu$ of local specifications the set $F_\Sigma$ of hyperelliptic curves over $\Q$, where $C(c_2,\ldots,c_{2n+1})\in F_\Sigma$ if and only if
$(c_2,\ldots,c_{2n+1})\in\Sigma_\nu$ for~all~$\nu$ (including $\nu=\infty$).  We will still say in that case that 
$F_\Sigma$ is defined by congruence conditions.

If $F$ is a set of
hyperelliptic curves (\ref{Ccdef}) over $\Q$ defined by local congruence conditions, then we
use $\Inv(F)$ to denote the set $\{(c_2(C),\ldots,c_{2n+1}(C)):C\in F\}\subset \Z^{2n}$.   
We denote the $p$-adic closure of $\Inv(F)$ in $\Z_p^{2n}\setminus\{\Delta=0\}$ by $\Inv_p(F)$. 
We say that such a set $F$ of hyperelliptic curves over $\Q$ is {\it large at
  $p$} if $\Inv_p(F)$ contains all elements $(c_2,\ldots,c_{2n+1})\in\Z_p^{2n}$ 
  such that $p^2\nmid \Delta(c_2,\ldots,c_{2n+1})$.  
  Finally, we say that such a set $F$ of
hyperelliptic curves is {\it large} if it is large at all but finitely
many primes $p$.  

In this section
we prove the following
strengthening of Theorem~\ref{main}:
\begin{theorem}\label{congmain}
  When all hyperelliptic curves over $\Q$ of genus $n$ with a 
  rational Weierstrass point in any large family are ordered by height, the average
size of the  $2$-Selmer group is $3$. 
\end{theorem}
Theorem~\ref{congmain} is proven via an appropriate sieve applied to the counts of $G(\Z)$-orbits on $V(\Z)$ having bounded height which we obtained in Section~10. 

Note that the set of all hyperelliptic curves in (\ref{Ccdef}) with indivisible coefficients is large.  So too is the set of such hyperelliptic curves $C:y^2=f(x)$ having semistable reduction.  Any family of hyperelliptic curves in (\ref{Ccdef}) defined by finitely many congruence conditions on the coefficients is also a large family.  Thus Theorem~\ref{congmain} applies to quite general families of hyperelliptic curves.

\subsection{A weighted set $U(F)$ in $V(\Z)$ corresponding to a large family $F$}

It follows from Theorem~\ref{kernel} and Proposition \ref{intorbit} that
non-identity elements of the 2-Selmer group of the Jacobian of the
hyperelliptic curve $C=C(c_2,\ldots,c_{2n+1})$ over $\Q$ defined by (1) are in bijective
correspondence with $G(\Q)$-equivalence classes of irreducible locally
soluble elements $B\in V(\Z)$ having invariants $2^8c_2$, $2^{12}c_3$, $\ldots$,
$2^{4(2n+1)}c_{2n+1}$;  in this bijection, we have $H(B)=2^{8n(2n+1)}H(C)$.  In Section~10, we computed the asymptotic number of
$G(\Z)$-equivalence classes of irreducible elements $B\in V(\Z)$
having bounded height.
In order to use this to compute the number of irreducible locally soluble $G(\Q)$-equivalence classes of elements $B\in V(\Z)$ having invariants in 
\begin{equation}\label{newinvts}
\{(2^8c_2,\ldots,2^{4(2n+1)}c_{2n+1}):(c_2,\ldots,c_{2n+1})\in \Inv(F)\}
\end{equation}
and bounded height (where $F$ is any large family), we need to count each 
$G(\Z)$-orbit, $G(\Z)\cdot B$,
with a weight of $1/n(B)$, where $n(B)$ is equal to the number of
$G(\Z)$-orbits inside the $G(\Q)$-equivalence class of $B$
in $V(\Z)$.  

To count the number of irreducible locally soluble $G(\Z)$-orbits having invariants in the set~(\ref{newinvts}) and  bounded height, where each orbit $G(\Z)\cdot B$ is weighted by $1/n(B)$, it suffices to count the number of such $G(\Z)$-orbits of
 bounded height such that each orbit
$G(\Z)\cdot B$ is weighted instead by $1/m(B)$, where
$$m(B):=\sum_{B'\in O(B)} \frac{\#\Aut_\Q(B')}{\#\Aut_\Z(B')}\; = 
\sum_{B'\in O(B)} \frac{\#\Aut_\Q(B)}{\#\Aut_\Z(B')}\;
;$$ here $O(B)$ denotes a set of orbit representatives for the action of $G(\Z)$ on the $G(\Q)$-equivalence class of $B$ in $V(\Z)$, and $\Aut_\Q(B)$ (resp.\
$\Aut_\Z(B)$) denotes the stabilizer of $B$ in $G(\Q)$ (resp.\
$G(\Z)$). The reason it suffices to weight by $1/m(B)$ instead of $1/n(B)$ is that we have shown in the
proof of Proposition~\ref{gzbigstab} that all but a negligible number $o(X^{(2n+3)/(4n+2)})$ of
irreducible $G(\Z)$-orbits having height less than $X$ have trivial
stabilizer in $G(\Q)$ (and thus also in $G(\Z)$), while the number of hyperelliptic curves in $F$ of bounded height is $\gg X^{(2n+3)/(4n+2)}$.

We use $U(F)$ to denote the set of all locally soluble elements in $V(\Z)$ having invariants in the set~(\ref{newinvts}), i.e.,
\begin{equation}\label{udef}
U(F) := \{\mbox{loc.\ sol.\ elts.\ in $V(\Z)$ having inv'ts.\ $(2^8c_2,\ldots,2^{4(2n+1)}c_{2n+1})$}\}
\,|\,C(c_2,\ldots,c_{2n+1})\in F\}.
\end{equation}
We assign to each element $B\in U(F)$ the weight $1/m(B)$.  Then we conclude that the weighted number of irreducible $G(\Z)$-orbits of height less than $2^{8n(2n+1)}X$ in $U(F)$ is asymptotically equal to the number of non-identity 2-Selmer elements of all hyperelliptic curves of height less than $X$ in $F$.  

The global weights $m(B)$ assigned to elements $B\in U(F)$ are useful for the following reason.  
For a prime $p$ and any element $B\in V(\Z_p)$, 
define the local weight $m_p(B)$ by
$$m_p(B):=\sum_{B'\in O_p(B)} \frac{\#\Aut_{\Q_p}(B)}{\#\Aut_{\Z_p}(B')},$$ where $O_p(B)$ denotes 
a set of orbit representatives for the action of $G(\Z_p)$ on the $G(\Q_p)$-equivalence class of $B$ in $V(\Z_p)$, and 
$\Aut_{\Q_p}(B)$
(resp.\ $\Aut_{\Z_p}(B)$) denotes the stabilizer of $B$ in $G(\Q_p)$ (resp.\
$G(\Z_p)$). 
Then by Proposition~\ref{product}, we have the following identity:
 \begin{equation}\label{prodformula}
 m(B)=\prod_pm_p(B).
 \end{equation}
Thus the global weights of elements in $U(F)$ are products of local weights, so we may express the global weighted density of elements $U(F)$ in $V(\Z)$ as products of local weighted densities of the closures of the set $U(F)$ in $V(\Z_p)$.  We consider these local weighted densities next. 

\subsection{Local densities of the weighted sets $U(F)$}

Suppose that $F$ is a large family of elliptic curves, and for each place $\nu$ of $\Q$, let $F_\nu$ denote the resulting family of curves defined by congruence conditions over $\Z_\nu$ (where we follow the usual convention that
$\Z_\nu=\Q_\nu=\R$ when $\nu=\infty$).
Let $U(F)$ denote the set in $V(\Z)$ attached to $F$,
as defined by (\ref{udef}), 
and for any place $\nu$ of $\Q$ define $U_\nu(F)$ analogously by
\begin{equation}\label{updef}
U_\nu(F) := \{\mbox{sol.\ elts.\ in $V(\Z_\nu)$ having inv'ts.\ $(2^8c_2,\ldots,2^{4(2n+1)}c_{2n+1})$}\}
\,|\,C(c_2,\ldots,c_{2n+1})\in F_\nu\}.
\end{equation}

In the case $\nu=p$ is a finite prime,  we assign to each element $B\in
U_p(F)$ the weight $1/m_p(B)$.  In this subsection, we determine the weighted
$p$-adic density of $U_p(F)$ in $V(\Z_p)$
in terms of a {\it local $(p$-adic$)$ mass} $M_p(V,F)$ involving the
elements of $\Jac(C)/2\Jac(C)$ for curves $C$ in $F_p$.
To do so, we require the following proposition, which is a
reformulation of the change-of-measure assertion of
Proposition~\ref{genjac}, with $\Z_p$ in place of $\R$:
\begin{proposition}\label{jacG}
 Let $\mathcal J$ be the constant 
 of Proposition $\ref{genjac}$, let $p$ be a prime, and let 
$\phi$ be a continuous function
  on $V({\Z_p})$. Then
  \begin{equation}\label{eqjacg}
    \int_{V_{\Z_p}}\!\phi(B)dB=|\mathcal J|_p\!
\int_{\substack{(c_2,\ldots,c_{2n+1})\in \Z_p^{2n}\\\!\!\Delta(c_2,\ldots,c_{2n+1})\neq 0}}\Bigl(\!\!\sum_{B\in\!\!\textstyle{\frac{V_{\Z_p}(c_2,\ldots,c_{2n+1})}{G(\Z_p)}}}\!\frac{1}{\#\Aut_{\Z_p}(B)}\!\int_{g\in G(\Z_p)}\!\!\phi(g\cdot B)\;\!\omega(g)\!\Bigr)dc_2\cdots dc_{2n+1},
  \end{equation}
  where $\frac{V_{\Z_p}(c_2,\ldots,c_{2n+1})}{G(\Z_p)}$ denotes a set of
  representatives for the action of $G(\Z_p)$ on elements in
  $V_{\Z_p}$ having invariants $c_2,\ldots,c_{2n+1}$.
\end{proposition}
The proof of Proposition~\ref{jacG} is identical to that of \cite[Prop.~3.12]{BS}. 
We observe that if $\phi$ is supported only on elements that have nonzero
discriminant and are soluble, and if $\phi(B)$ is additionally
weighted by $1/m_p(B)$, then Equation (\ref{eqjacg}) takes on a particularly nice form:
\begin{corollary}\label{corjac}
  Let $p$ be a prime and let $\phi$ be a continuous $G(\Q_p)$-invariant
  function on $V_{\Z_p}$ such that every element $B\in
  V_{\Z_p}$ in the support of $\phi$ has nonzero discriminant, is soluble, and satisfies $2^{4j}\mid c_j(B)$ for all $j\in\{2,\ldots,2n+1\}$.
 Then
  \begin{equation}\label{eqjacpgl3}
    \int_{V_{\Z_p}}\!\!\frac{\phi(B)}{m_p(B)}dB=|\mathcal J|_p\Vol(G(\Z_p))\!\!
\int_{{C=C(c_2,\ldots,c_{2n+1})}}\!\frac{1}{\# \Jac(C)[2](\Q_p)}\Bigl(\sum_{\sigma\in\!\! \textstyle\frac{\Jac(C)(\Q_p)}{2\Jac(C)(\Q_p)}}\!\!\!\phi(B_\sigma)\Bigr)dc_2\cdots dc_{2n+1},
  \end{equation}
  where $B_\sigma$ is any element in $V_{\Z_p}$ that corresponds to
  $\sigma$ under the correspondence of Proposition~$\ref{kernel}$. $($The
  existence of such an $f_\sigma\in V_{\Z_p}$ is guaranteed by
  Proposition~$\ref{intorbit}.)$
\end{corollary}

\begin{proof}
  Proposition \ref{jacG} implies that the left side of (\ref{eqjacpgl3}) is equal to
  \begin{equation}\label{eqjaccorf}
    \begin{array}{rcl}
& &
|\mathcal J|_p\displaystyle
\int_{\substack{(c_2,\ldots,c_{2n+1})\in \Z_p^{2n}\\\!\!\Delta(c_2,\ldots,c_{2n+1})\neq 0}}\Bigl(\sum_{B\in\textstyle{\frac{V_{\Z_p}(c_2,\ldots,c_{2n+1})}{G(\Z_p)}}}\frac{1}{\#\Aut_{\Z_p}(B)}\int_{g\in G(\Z_p)}\frac{\phi(g\cdot B)}{m_p(B)}\;\!\omega(g)\Bigr)dc_2\cdots dc_{2n+1}
\\[0.4in]
&=&|\mathcal J|_p\Vol(G(\Z_p))
\displaystyle
\int_{\substack{(c_2,\ldots,c_{2n+1})\in \Z_p^{2n}\\\!\!\Delta(c_2,\ldots,c_{2n+1})\neq 0}}\Bigl(\sum_{B\in\textstyle{\frac{V_{\Z_p}(c_2,\ldots,c_{2n+1})}{G(\Z_p)}}}\displaystyle\frac{\phi(B)}{m_p(B)\#\Aut_{\Z_p}(B)}\Bigr)dc_2\cdots dc_{2n+1}
    \end{array}
  \end{equation}
 since both $\phi$ and $m_p$ are
  $G(\Z_p)$-invariant. 
  For $B\in
  V_{\Z_p}$, let $B=B_1,B_2,\ldots,B_k$ be the set of all elements in
  $\frac{V_{\Z_p}(c_2,\ldots,c_{2n+1})}{G(\Z_p)}$ that are
  $G(\Q_p)$-equivalent to $B$. Then, since $\phi$ and $m_p$ are $G(\Q_p)$-invariant, we have
\begin{eqnarray*}
\sum_{i=1}^k\frac{\phi(B_i)}{m_p(B_i)\#\Aut_{\Z_p}(B_i)}&\!\!=\!\!&\frac{1}{m_p(B)}\sum_{i=1}^k\frac{\phi(B)}{\#\Aut_{\Z_p}(B_i)}=
\left(\sum_{i=1}^k\frac{\#\Aut_{\Q_p}(B)}{\#\Aut_{\Z_p}(B_i)}\right)^{-1}\!\sum_{i=1}^k\frac{\phi(B)}{\#\Aut_{\Z_p}(B_i)}\\ &\!\!=\!\!&\frac{\phi(B)}{\#\Aut_{\Q_p}(B)}.
\end{eqnarray*}
Therefore, 
\begin{equation}\label{eqjaccors}
\int_{V_{\Z_p}}\frac{\phi(B)}{m_p(B)}dB=
|\mathcal J|_p\Vol(G(\Z_p))
\displaystyle
\int_{\substack{(c_2,\ldots,c_{2n+1})\in \Z_p^{2n}\\\!\!\Delta(c_2,\ldots,c_{2n+1})\neq 0}}\Bigl(\sum_{B\in\textstyle{\frac{V_{\Z_p}(c_2,\ldots,c_{2n+1})}{G(\Q_p)}}}
\displaystyle\frac{\phi(B)}{\#\Aut_{\Q_p}(B)}\Bigr)dc_2\cdots dc_{2n+1},  
\end{equation}
where $\frac{V_{\Z_p}(c_2,\ldots,c_{2n+1})}{G(\Q_p)}$ analogously denotes a set consisting of one
representative from each $G(\Q_p)$-equivalence class in $V_{\Z_p}$ having invariants
$c_2,\ldots,c_{2n+1}$. Proposition~\ref{kernel} 
states that soluble elements in
$\frac{V_{\Z_p}(c_2,\ldots,c_{2n+1})}{G(\Q_p)}$ are in bijective correspondence
with elements in $\Jac(C)(\Q_p)/2\Jac(C)(\Q_p)$, where $C=C(c_2,\ldots,c_{2n+1})$; meanwhile, Proposition~\ref{isom} 
states that $\Aut_{\Q_p}(B)$ is isomorphic to
$\Jac(C)[2](\Q_p)$. Corollary \ref{corjac} thus follows
from (\ref{eqjaccors}).
\end{proof}

We may now prove the following proposition which determines the necessary
local $p$-adic masses.


\begin{proposition}\label{denel}
 Let $\mathcal J$ be the constant 
 of Proposition $\ref{genjac}$, and let $F$ be any large family of hyperelliptic curves. Then
$$\int_{U_p(F)}\frac{1}{m_p(v)}dv\,=\,|2^{4n(2n+3)}\mathcal J|_p\cdot \Vol(G(\Z_p))\cdot
M_p(V,F),$$
where
$$M_p(V,F)\,:=\,\displaystyle{
\int_{C=C(c_2,\ldots,c_{2n+1})\in F_p}
\sum_{\sigma\in \textstyle\frac{\Jac(C)(\Q_p)}{2\Jac(C)(\Q_p)}}
\frac1{\#\Jac(C)[2](\Q
_p)}dc_2\cdots dc_{2n+1}}.$$
\end{proposition}

\begin{proof}
  The set $U(F)$ consists of the soluble elements in $V(\Z)$
  having invariants $2^8c_2,$ $2^{12}c_3,$ $\ldots,$
  $2^{4(2n+1)}{c_{2n+1}}$ 
where $(c_2,\ldots,c_{2n+1})\in
  \Inv_p(F)$. Proposition~\ref{denel} thus follows directly from
  Corollary~\ref{corjac} since $\Jac(C)(c_2,\ldots,c_{2n+1})(\Q_p)$ 
is isomorphic to
  $\Jac(C)(2^8c_2,2^{12}c_3,\ldots,2^{4(2n+1)}{c_{2n+1}})(\Q_p)$ 
and the volume of
$\{(2^8c_2,2^{12}c_3,\ldots,2^{4(2n+1)}{c_{2n+1}})|(c_2,\ldots,c_{2n+1})
\in
  \Inv_p(F)\}$
is equal to $|2^{4n(2n+3)}|_p\cdot\Vol(\Inv_p(F))$.
\end{proof}

\noindent
We may also define a local mass at the infinite place:
\begin{equation}
M_\infty(V,F;X):=\displaystyle\int_{\substack{C=C(c_2,\ldots,c_{2n+1})\in F_\infty\\H(C)<X}}
\sum_{\sigma\in \textstyle\frac{\Jac(C)(\R)}{2\Jac(C)(\R)}}
\frac1{\#\Jac(C)[2](\R)}dc_2\cdots dc_{2n+1}.
\end{equation}

In the analogous manner, if $F$ is a large family of hyperelliptic curves, 
then we may define $M_p(F)$ to be the measure of $\Inv_p(F)$
with respect to the measure $dc_2\cdots dc_{2n+1}$ on $\Z_p^{2n}$, where the
measure $dc_m$ on $\Z_p$ is normalized so that the total measure is
$1$.  That is, we have
\begin{equation}\label{ecden}
\;\;\;\;\;\;M_p(F):=\int_{C=C(c_2,\ldots,c_{2n+1})\in F_p}dc_2\cdots dc_{2n+1}.
\end{equation}
Similarly, we define
\begin{equation}
M_\infty(F;X):=\displaystyle\int_{\substack{C=C(c_2,\ldots,c_{2n+1})\in F_\infty\\H(C)<X}}
\!dc_2\cdots dc_{2n+1}.
\end{equation}

In Section 12, we will be interested in comparing the masses $M_p(V,F)$ and $M_p(F)$, and \linebreak
$M_\infty(V,F;X)$ and $M_\infty(F;X)$.

\subsection{A key uniformity estimate, and a squarefree sieve}\label{unisec}

For each prime $p$, let $W_p$ denote the set of elements $B$ in $V(\Z)$ such that $p^2\mid\Delta(B)$.  Then the following proposition is proven in \cite{geosieve}:

\begin{proposition}\label{uniformity}
There exists $\delta>0$ such that, \,for any $M>0$, \,we have 
$$\sum_{p>M} N(W_p; X) = O(X^{(2n+3)/(4n+2)}/M^\delta),$$
where the implied constant is independent of $X$ and $M$. 
\end{proposition}

Proposition~\ref{uniformity} allows us to prove a more general congruence version of 
Theorem~\ref{cong2}, namely, one which allows
appropriate infinite sets of congruence conditions to be imposed and
which also allows weighted counts of lattice points (where weights are
also assigned by congruence conditions). 
Specifically, let us say that a
function $\phi:V(\Z)\to[0,1]\subset\R$ is {\it defined by congruence
  conditions} if, for all primes $p$, there exist functions
$\phi_p:V({\Z_p})\to[0,1]$ satisfying the following conditions:
\begin{itemize}
\item[(1)] For all $B\in V(\Z)$, the product $\prod_p\phi_p(B)$ converges to $\phi(B)$;
\item[(2)] For each prime $p$, the function $\phi_p$ is 
locally constant outside some closed set $S_p \subset V({\Z_p})$ of measure zero.
\end{itemize}
Such a function $\phi$ is called {\it acceptable} if, for sufficiently
large primes $p$, we have $\phi_{p}(B)=1$ whenever $p^2\nmid
\Delta(B)$.  For example, the
characteristic function of the set of elements $B\in V(\Z)$
having squarefree discriminant is an acceptable function.

We then have the following version of Theorem~\ref{cong3}, in which we
allow weights to be defined by certain infinite sets of congruence
conditions:
\begin{theorem}\label{thsquarefreebq}
  Let $\phi:V(\Z)\to[0,1]$ be an acceptable function that is defined by
  congruence conditions via the local functions $\phi_{p}:V({\Z_p})\to[0,1]$. Then, with
  notation as in Theorem~$\ref{cong3}$, we have:
\begin{equation}
N_\phi(V(\Z)^{(m,\tau)};X)
  = N(V(\Z)^{(m,\tau)};X)
  \prod_{p} \int_{B\in V({\Z_{p}})}\phi_{p}(B)\,dB+o({X^{(2n+3)/(4n+2)}}).
\end{equation}
\end{theorem}
The proof of Theorem~\ref{thsquarefreebq} is identical to that of \cite[Thm.~2.21]{BS}; the idea is to use Theorem~\ref{cong3} to impose more and more congruence conditions, while using Proposition~\ref{uniformity} to uniformly 
bound the 
error term.

\subsection{Weighted count of elements in $U(F)$ having bounded height}

For a large family $F$, we may now describe the asymptotic number of
$G(\Z)$-orbits in $U(F)$ having bounded height, where as before each
element $B\in U(F)$ is counted with weight $1/m(B)$:

\begin{theorem}\label{ufcount}
Let $F$ be any large family of hyperelliptic curves.  Then $N_{1/m}(U(F);2^{8n(2n+1)}X)$, the weighted number of $G(\Z)$-orbits in $U(F)$ having height less than $2^{8n(2n+1)}X$, is given by
\begin{equation}\label{eval}
\begin{array}{rl}
\!\!N_{1/m}(U(F);2^{8n(2n+1)}X) &\!\!\!=
\displaystyle{\sum_{m=0}^n\sum_{\tau=1}^{2^m} N(U_\infty(F)\cap
  V(\Z)^{(m,\tau)}
; 2^{8n(2n+1)}X)\cdot
\prod_p\int_{U_p(F)}\frac{1}{m_p(v)}dv}
\\[.325in]
&\hspace{2in}+\,o(X^{(2n+3)/(4n+2)}).
\end{array}
\end{equation}
\end{theorem}
Theorem~\ref{ufcount} follows from Corollary~\ref{unitmax}, Equation (\ref{prodformula}), and Theorem~\ref{thsquarefreebq}.

It remains to evaluate expression (\ref{eval}) in terms of the total number of hyperelliptic curves in $F$ having height less than $X$.

\section{Proof of Theorem \ref{main}}

\subsection{The number of hyperelliptic curves of bounded height}

\begin{theorem}\label{ccount}
Let $F$ be a large family of hyperelliptic curves.  Then the number of curves $C$ in $F$ with $H(C)<X$ is
given by
\[
M_\infty(F;X)
\cdot \prod_p M_p(F)
+ o(X^{(2n+3)/(4n+2)}) .\]
\end{theorem}
This theorem is proven in much the same manner as
Theorem~\ref{ufcount} (but is easier), again using the uniformity
estimate of Proposition~\ref{uniformity}. 


\subsection{Evaluation of the average size of the 2-Selmer group}\label{finalsec}

We now have the following theorem, from which Theorem \ref{congmain} (and thus Theorem \ref{main}) will be seen to follow.
\begin{theorem}\label{main2}
Let $F$ be a large family of hyperelliptic curves. Then we have
\begin{equation}\label{eqthsec5}
\displaystyle\lim_{X\to\infty}\frac{\displaystyle\sum_{\substack{C\in F\\H(C)<X
}}(\#S_2(J)-1)}{\displaystyle\sum_{\substack{C\in F\\H(C)<X}}1}=
\frac{|\mathcal J|\,\Vol(G(\Z)\backslash G(\R))
M_\infty(V,F;X)}{M_\infty(F;X)} 
\cdot\, \displaystyle\prod_p\frac{|\mathcal J|_p \Vol(G(\Z_p))
M_p(V,F)}{M_p(F)}.
\end{equation}

\end{theorem}
\begin{proof}
This follows by combining Proposition~\ref{denel}, Theorem~\ref{ufcount}, Theorem~\ref{thmcount}, expression (\ref{volexp}) for the volume $\Vol(\RR_X(L^{(m,\tau)}))=\Vol(\FF L^{(m,\tau)}\cap \{v\in V(\R):H(v)<X\})=2^{m+n}c_{m,\tau}X^{(2n+3)/(4n+2)}$ (noting that $2^{m+n}=\#J[2](\R)$, where $J=\Jac(C)(c_2,\ldots, c_{2n+1})$ for any $(c_2,\ldots,c_{2n+1})\in I(m)$), and Theorem~\ref{ccount}.
\end{proof}

\noindent
In order to evaluate the right hand side of the expression in Theorem~\ref{main2}, we
use the following fact (see \cite[Lemma~3.1]{BK}):
\begin{lemma}\label{lembk}
Let $J$ be an abelian variety over $\Q_p$ of dimension $n$. Then
$$\#(J(\Q_p)/2J(\Q_p))=
\left\{\begin{array}{cl}
\#J[2](\Q_p) & {\mbox{\em if }} p\neq 2;\\[.1in]
2^n\cdot\#J[2](\Q_p)& {\mbox{\em if }} p= 2.
\end{array}\right.$$
\end{lemma}
\begin{proof}
  It follows from the theory of formal groups 
  that there exists a subgroup
  $M\subset J(\Q_p)$ of finite index that is isomorphic to $\Z_p^n$ \cite{M}. Let
  $H$ denote the finite group $J(\Q_p)/M$ .
 Then by applying the snake lemma to the following diagram
$$\xymatrix{
0\ar[d]\ar[r]&M\ar[d]^{[2]}\ar[r]&J(\Q_p)\ar[d]^{[2]}\ar[r]&H\ar[d]^{[2]}\ar[r]&0
\ar[d]\\
0\ar[r]&M\ar[r]&J(\Q_p)\ar[r]&H\ar[r]&0}$$
we obtain the exact sequence
$$0\to M[2]\to J(\Q_p)[2]\to H[2]\to M/2M\to J(\Q_p)/2J(\Q_p)\to H/2H\to 0.$$
Since $H$ is a finite group and $M$ is isomorphic to $\Z_p^n$, Lemma \ref{lembk}  follows.
\end{proof}

The expression on the right hand side in Theorem~\ref{main2} thus
reduces simply to the Tamagawa number
$\tau(G)=\Vol(G(\Z)\backslash G(\R))\cdot \prod_p
\Vol(G(\Z_p)$ of $G$, which is 2.  This completes the proof of
Theorem~\ref{congmain}, and thus of Theorem~\ref{main}.

\vspace{.05in}
In fact, the proof of Theorem~\ref{main} shows more.  For example, we may study the distribution
of the nonidentity Selmer elements inside $J(\Q_\nu)/2J(\Q_\nu)$ for any (finite or infinite) place $\nu$ of $\Q$. 
In this regard, we have: 

\begin{theorem}\label{equi}Fix $\nu$. 
Let $F$ be a large family of hyperelliptic curves $C$ of genus $n$ such that 
\begin{itemize}
\item[{\rm (a)}]
the cardinality of $\Jac(C)(\Q_\nu)/2\Jac(C)(\Q_\nu)$ is a constant $k$ for all $C$ in $F;$ and
\item[{\rm (b)}] the set $U_\nu(F)\subset V(\Z_\nu)$, defined by
  $(\ref{updef})$, can be partitioned into $k$ open sets $\Omega_i$
  such that$:$
\begin{itemize}
\item[{\rm (i)}] 
for all $i$, if two elements in $\Omega_i$ have the same invariants
$c_2,\ldots,c_{2n+1}$, then they are $G(\Q_\nu)$-equivalent$;$ and
\item[{\rm (ii)}]
 for all $i\neq j$, we have $(G(\Q_\nu)\cdot \Omega_i)\cap (G(\Q_\nu)\cdot
  \Omega_j)=\varnothing$.
\end{itemize}
\end{itemize}
$($In particular, the groups $\Jac(C)(\Q_\nu)/2\Jac(C)(\Q_\nu)$
are naturally identified for all $C$ in $F$.$)$ Then when the
hyperelliptic curves $C$ in $F$ are ordered by height, the images of the
non-identity $2$-Selmer elements under the canonical map
\begin{equation}
S_2(\Jac(C))\to \Jac(C)(\Q_\nu)/2\Jac(C)(\Q_\nu)
\end{equation} are equidistributed.
\end{theorem}

We first prove that for any given value $c\in
\Z_\nu^{2n}\setminus\{\Delta=0\}$ for the invariants
$(c_2,\ldots,c_{2n+1})$, there always exists a sufficiently small
$\nu$-adic neighborhood $W$ of $c$ in $\Z_\nu^{2n}$ such that the
corresponding (large) family $F=F(W)$, consisting of all hyperelliptic
curves having invariants in $W$, satisfies both 
(a) and~(b).
Indeed, when $\nu=\infty$, if $c\in I(m)$ then we simply let $W=I(m)$, let $k=2^m=
\#(\Jac(C(c))(\R)/2\Jac(C(c))(\R))$,  and let $\Omega_1,\ldots,\Omega_k$ be equal to $V^{(m,1)}, \ldots,V^{(m,k)}$, respectively.

If $\nu$ is a finite prime $p$, let $k$ be the cardinality of
$\Jac(C(c))(\Q_p)/2\Jac(C(c))(\Q_p)$.  Then the set $Y$ of soluble
elements in the inverse image of $(2^8c_2,\ldots,2^{4(2n+1)}c_{2n+1})$
under the map $\pi:V(\Z_p) \to \Z_p^{2n}$, given by taking invariants,
is the disjoint union of $k$ nonempty compact sets $Y_1,\ldots,Y_k$,
namely, the $k$\, $G(\Q_p)$-equivalence classes in $V(\Z_p)$
comprising $Y$.

Let 
$Z_1,\ldots,Z_k\subset V(\Z_p)\setminus\{\Delta=0\}$ be disjoint
neighborhoods of $Y_1,\ldots,Y_k$, respectively, in $V(\Z_p)$ such
that each $Z_i$ consists of soluble elements and is the union of
$G(\Q_p)$-equivalence classes in $V(\Z_p)$.
Such~$Z_i$ can be constructed by noting that if $\varepsilon$ is
sufficiently small, then the $\varepsilon$-neighborhoods
$B_{\varepsilon}(Y_i)$ of the $Y_i$'s are disjoint and consist only of
elements that have nonzero discriminant and are soluble.
The set $\{ g \in G(\Q_p) \,|\, g B_\varepsilon(Y_i)\cap
V(\Z_p)\neq \varnothing \}$ 
is then compact.  Indeed, for a single stable element $v\in V(\Z_p)$,
the set
\begin{equation}\label{star}
\{ g \in G(\Q_p)\,|\, gv \in V(\Z_p) \}
\end{equation}
is compact, as it is the union of finitely many right cosets
$G(\Z_p)\cdot g_1,\ldots,G(\Z_p)\cdot g_m$ of $G(\Z_p)$; this is
because there are only finitely many $G(\Z_p)$-orbits in the
$G(\Q_p)$-orbit of $v$, by Propositions~\ref{podd} and~\ref{peven}.
Next, note that the set given by (\ref{star}) is constant in a
neighborhood of $v$.  This is because the number of $G(\Z_p)$-orbits in
$G(\Q_p)\cdot v'$ is a constant for elements $v'$ in a sufficiently
small neighborhood of $v$, e.g., we may take a neighborhood of $v$ small enough
so that the associated ring in Proposition~\ref{podd} or~\ref{peven}
is a constant.  The compactness of
\begin{equation}\label{star2}
\{ g \in G(\Q_p) \,|\, gS \cap V(\Z_p)=\varnothing \}
\end{equation}
now follows for any compact set $S$ of stable elements (by covering $S$
with neighborhoods where (\ref{star}) is constant, and then taking a
finite subcover).  Hence $(G(\Q_p)\cdot B_{\varepsilon}(Y_i))\cap
V(\Z_p)$ is both open and compact, and so is a bounded distance away
from $Y_j$ for all $j\neq i$. By shrinking $\varepsilon$ if necessary,
we can then ensure that $(G(\Q_p)\cdot B_{\varepsilon}(Y_i))\cap
B_{\varepsilon}(Y_j)=\varnothing$ for all $i\neq j$, and therefore
$(G(\Q_p)\cdot B_{\varepsilon}(Y_i))\cap (G(\Q_p)\cdot
B_{\varepsilon}(Y_j))=\varnothing$ for all $i\neq j$.  We then set
$Z_i=(G(\Q_p)\cdot B_{\varepsilon}(Y_i))\cap V(\Z_p)$.

Let $W'=\{(b_2,\ldots,b_{2n+1})\in\Z_p^{2n}\,|\,(2^8b_2,\ldots,2^{4(2n+1)}b_{2n+1})\in\cap_i\,\pi(Z_i)\}$.
Since~$\pi$ is an open mapping
on $V(\Z_p)\setminus\{\Delta=0\}$, we see that $W'$ is an open set in $\Z_p^{2n}$
containing $c$.  Let $W\subset W'$ be an open neighborhood of $c$
small enough so that for all hyperelliptic curves $C$ having
invariants in $W$, we have $\#(\Jac(C)(\Q_p)/2\Jac(C)(\Q_p))=k$.
Such a neighborhood $W$ exists because 
the size of $\Jac(C)(\Q_p)/2\Jac(C)(\Q_p)$, where $C$ is given by $y^2=f(x)$,     
depends only on the shape of the factorization of $f(x)$ over $\Q_p$, i.e., the number of irreducible factors and their degrees, so this size is locally 
constant as a function of the coefficients of $f$.
Then $F(W)$ satisfies both (a) and (b), with
$\Omega_i=Z_i\cap\pi^{-1}(\{2^8b_2,\ldots,2^{4(2n+1)}b_{2n+1})\,|\,(b_2,\ldots,b_{2n+1})\in W\})$.


To prove Theorem~\ref{equi}, on the right side of the equation in Theorem~\ref{main2},
we replace the sum over all $\sigma$ in $\Jac(C)(\Q_\nu)/2\Jac(C)(\Q_\nu)$ in the definition of $M_\nu(V,F)$ 
by the corresponding sum over $\sigma$ lying in any subset $\Sigma\subset \Jac(C)(\Q_\nu)/2\Jac(C)(\Q_\nu)$.  (By property (b), we are still counting elements in a weighted subset in $V(\Z)$ defined by congruence conditions, so Theorem~\ref{thsquarefreebq} again applies.)  This gives us the average number of nonidentity Selmer elements that map to $\Sigma$, and we see that the result is proportional to the size of $\Sigma$, proving Theorem~\ref{equi}.  The same argument also allows one to show equidistribution of non-identity 2-Selmer elements in
$\prod_{\nu\in S}\Jac(C)(\Q_\nu)/2\Jac(C)(\Q_\nu)$ for any finite set~$S$ of places of $\Q$, provided that our large family $F$ of hyperelliptic curves lies in the intersection of sufficiently small $\nu$-adic discs ($\nu\in S$) so that both (a) and (b) are satisfied for all $\nu\in S$.

\vspace{.1in}
There is an alternative interpretation of many of these calculations (and cancellations!) in terms of an adelic geometry-of-numbers argument. 
This interpretation was introduced by Poonen in~\cite{Poo} and we briefly describe it here. Let $\mathbb A$ be the ring of adeles of $\Q$ and let $dv$ be the unique Haar measure on $V(\mathbb A)$ which is the counting measure on the discrete subgroup $V(\Q)$ and gives the compact quotient $V(\mathbb A)/V(\Q)$ volume one. Let $S$ be the vector group over $\Q$ of dimension $2n$ where the invariant polynomials take values, and let $ds$ be the unique Haar measure on $S(\mathbb A)$ which is counting measure on the discrete subgroup $S(\Q)$ and gives the compact quotient $S(\mathbb A)/S(\Q)$ volume one.

Define the compact subset $S(\mathbb A)_{<X}$ as the product of the elements in $S(\Z_p)$ that define a $p$-indivisible polynomial having nonzero discriminant with the elements in $S(\mathbb R)$ that define a polynomial having nonzero discriminant and bounded height $H < X$. Define the subset $V(\mathbb A)_{<X}$ as the subset of elements in $V(\mathbb A)$ locally soluble everywhere and having characteristic polynomial in $S(\mathbb A)_{<X}$. The group $G(\Q)$ acts on $V(\mathbb A)_{<X}$ and the orbit space has finite volume. The adelic volume formula is then
$$\mathrm{vol}(V(\mathbb A)_{<X}/G(\Q), dv) = 2\cdot \mathrm{vol}(S(\mathbb A)_{<X}, ds).$$
The factor $2$ in this formula is the Tamagawa number of $G$: the volume of $G(\mathbb A)/G(\Q)$ with respect to Tamagawa measure $\omega(g)$. The proof of this volume formula involves the compatibility of the three measures $ds$, $dv$, and $\omega(g)$ (as in Proposition~\ref{genjac}) as well as a product formula for the Herbrand quotients $\#(J(\Q_v)/2J(\Q_v)) / \#J[2](\Q_v)$ (as in Proposition~\ref{lembk}). Geometry-of-numbers arguments (as in Sections~9 and 10) and a sieve (as in Section 11) can then be used to estimate (as in Section~12) the average number ($= 2$) of irreducible locally soluble rational orbits, and so the average number of non-trivial classes in the $2$-Selmer group.


\subsection*{Acknowledgments}
We are very grateful to John Cremona, Joe Harris, Wei Ho, Bjorn Poonen, Arul Shankar, Michael Stoll, Jack Thorne, and Jerry Wang for their help.

\def\noopsort#1{}
\providecommand{\bysame}{\leavevmode\hbox to3em{\hrulefill}\thinspace}

\end{document}